 \documentclass[12pt,a4wide]{article}
\usepackage{hyperref}
\usepackage{latexsym, amsmath, amsthm, amsfonts, amssymb}
\usepackage[mathscr]{eucal}
\usepackage{tikz}
\usepackage{tikz-cd}
\def\apo           {\mbox{\sc s}}
\def\be            {\begin{equation}}
\def\bearl         {\begin{array}{l}}
\def\bearll        {\begin{array}{ll}}
\newcommand\Binom[2]{\mbox{\large$\binom{#1}{#2}$}}
\def\blue          {dotted}
\newcommand\bu[2]  {\begin{array}{c}{}\\[-24pt]#1\\[-4pt]\bullet\\[-7pt]{\scriptstyle q^{#2}}\eear}

\def\cdo           {\,{.}\,}

\def\cir           {\,{\circ}\,}
\def\coa           {_\triangleright}
\def\complex       {{\ensuremath{\mathbb C}}}
 
\def\dim           {\mathrm{dim}}
\def\DG            {{\ensuremath{\mathscr DG}}}
\def\DGa           {{\ensuremath{{\mathscr DG}^\star_\triangleright}}}
\def\DGN           {\varPi_{\mathscr DG}^\Ng}
\def\DGs           {{\ensuremath{{\mathscr DG}^\star_{}}}}
\def\divfn         {\ensuremath{\mathrm{d}}}

\def\dsty          {\displaystyle }

\def\ee            {\end{equation}}

\def\eear          {\end{array}}
\def\End           {\mathrm{End}}
\def\eps           {\varepsilon}
\def\eq            {\,{=}\,}
\def\expG          {\ensuremath{\mathrm{e}_G}}

\def\Gamm          {\Gamma_{\!e}}
\def\gc            {g_c}

\def\Hom           {\mathrm{Hom}}
\def\Hs            {\ensuremath{H^\star_{\phantom|}}}
\newcommand\hsp[1] {\mbox{\hspace{#1 em}}}
\def\id            {\mbox{\sl id}}

\def\idscs         {\mbox{\scriptsize\sl id}}

\def\iN            {\,{\in}\,}

\def\ko            {{\ensuremath{\Bbbk}}}

\def\Mod           {\mbox{\rm -mod}}
\def\Mods          {\mbox{\rm\scriptsize -mod}}

\def\nE            {\,{\ne}\,}
\newcommand\negspace[1] {~\\[-#1pt]} 

\newcommand\nxl[1] {\\[#1mm]}
\newcommand\Nxl[1] {\\[-1.3em]\\[#1mm]}

\def\oti           {\,{\otimes}\,}

\def\qquand        {\qquad{\rm and}\qquad}
\def\red           {dashed}

\def\Sp            {\mathrm{Sp}}
\def\slz           {\ensuremath{\mathrm{SL}(2,\zet)}}
\def\Symplectic    {Sympl.}

\def\Times         {\,{\times}\,}
\def\To            {\,{\to}\,} 

\def\zet           {\ensuremath{\mathbb Z}}

\def\inv           {^{-1}}
\renewcommand\b[2] {b_{#1|#2}}
\renewcommand\d[2] {\beta_{#1|#2}}
\newcommand\del[1] {\delta_{#1}}

\def\chark      {\ensuremath{\mathrm{char}(\Bbbk)}}
\def\com        {\ensuremath{\mathrm{Com}}}
\def\im         {\ensuremath{\mathrm{Im}}}

\def\NN	        {\ensuremath{\mathbb{N}}}
\def\Ng         {{\mathrm N}} 

\def\bh	        {\ensuremath{\mathbf{h}}}
\def\by	        {\ensuremath{\mathbf{y}}}
\def\bg	        {\ensuremath{\mathbf{g}}}
\def\bx	        {\ensuremath{\mathbf{x}}}
\def\mcg	{\ensuremath{\mathrm{Map}}}
\def\tor	{\ensuremath{\mathfrak{T}}}

\newtheorem{thm}{Theorem}[section]
\newtheorem{lemma}[thm]{Lemma}
\newtheorem{prop}[thm]{Proposition}
\newtheorem{cor}[thm]{Corollary}

\theoremstyle{definition}
\newtheorem{rem}[thm]{Remark}
\newtheorem{defn}[thm]{Definition}

\numberwithin{equation}{section}

\tikzset{object/.style={circle, minimum size=5mm, very thick, draw=black!50, fill=green!20!white}, listobj/.style={rectangle, minimum size=2mm, rounded corners=3mm, very thick, fill=blue!14!white}}

\newcommand\diagZtwoZtwoTwoorbits[4] 
  { \begin{tikzpicture}[scale=2,commutative diagrams/every diagram]
  \node (a) [listobj] at (0,0)   {\scriptsize $(0|0)$};
  \node (b) [listobj] at (#1,0)  {\scriptsize $(1|0)$};
  \node (c) [listobj] at (0,#2)  {\scriptsize $(0|1)$};
  \node (d) [listobj] at (#1,#2) {\scriptsize $(1|1)$};
  \node	(j)           at (#3,#4) {$\big(\zet/2\zet\big)^2$: $2$ orbits};
  \path[commutative diagrams/.cd, every arrow, every label]
  (a) edge [\red, out=155,in=205,loop=2mm]        (a)
  (c) edge [\red, out=155,in=205,loop=2mm]        (c)
  (b) edge [\red,<->]                             (d)
  (a) edge [\blue, thick, out=65,in=120,loop=1mm] (a)
  (d) edge [\blue, thick, out=25,in=335,loop=2mm] (d)
  (b) edge [\blue, thick, <->] (c)
  ;
  \end{tikzpicture}
 }

\newcommand\diagZthreeZthreeTwoorbits[4] 
 { \begin{tikzpicture}[scale=2,commutative diagrams/every diagram]
  \node	(a)  [listobj] at (0,0)	    {\scriptsize $(0|0)$};
  \node	(b)  [listobj] at (#1,0)    {\scriptsize $(1|0)$};
  \node	(c)  [listobj] at (2*#1,0)  {\scriptsize $(2|0)$};
  \node	(d)  [listobj] at (0,#2)    {\scriptsize $(0|1)$};
  \node	(e)  [listobj] at (#1,#2)   {\scriptsize $(1|1)$};
  \node	(f)  [listobj] at (2*#1,#2) {\scriptsize $(2|1)$};
  \node	(g)  [listobj] at (0,2*#2)  {\scriptsize $(0|2)$};
  \node	(h)  [listobj] at (#1,2*#2) {\scriptsize $(1|2)$};
  \node	(i)  [listobj] at (2*#1,2*#2) {\scriptsize $(2|2)$};
  \node	(j)            at (#3,#4) {$\big(\zet/3\zet\big)^2$: $2$ orbits};
  \path[commutative diagrams/.cd, every arrow, every label]
  (a) edge [\red, out=155,in=205,loop=2mm] (a)
  (d) edge [\red, out=155,in=205,loop=2mm] (d)
  (g) edge [\red, out=155,in=205,loop=2mm] (g)
  (h) edge [\red]                          (e)
  (e) edge [\red]                          (b)
  (b) edge [\red, bend right]              (h)
  (c) edge [\red]                          (f)
  (f) edge [\red]                          (i)
  (i) edge [\red, bend left]               (c)
  (a) edge [\blue, thick, out=65,in=120,loop=1mm] (a)
  (b) edge [\blue, thick]                  (d)
  (d) edge [\blue, thick]                  (c)
  (c) edge [\blue, thick,bend left=15]     (g)
  (g) edge [\blue, thick]                  (b)
  (e) edge [\blue, thick]                  (f)
  (f) edge [\blue, thick,bend left]        (i)
  (i) edge [\blue, thick]                  (h)
  (h) edge [\blue, thick,bend right]       (e)
  ;
  \end{tikzpicture}
 }

\newcommand\diagZfourZfourThreeorbits[4] 
 { \begin{tikzpicture}[scale=2,commutative diagrams/every diagram]
  \node	(a) [listobj] at (0,0)       {\scriptsize $(0|0)$};
  \node	(b) [listobj] at (#1,0)      {\scriptsize $(1|0)$};
  \node	(c) [listobj] at (2*#1,0)    {\scriptsize $(2|0)$};
  \node	(d) [listobj] at (3*#1,0)    {\scriptsize $(3|0)$};
  \node	(e) [listobj] at (0,#2)      {\scriptsize $(0|1)$};
  \node	(f) [listobj] at (#1,#2)     {\scriptsize $(1|1)$};
  \node	(g) [listobj] at (2*#1,#2)   {\scriptsize $(2|1)$};
  \node	(h) [listobj] at (3*#1,#2)   {\scriptsize $(3|1)$};
  \node	(i) [listobj] at (0,2*#2)    {\scriptsize $(0|2)$};
  \node	(j) [listobj] at (#1,2*#2)   {\scriptsize $(1|2)$};
  \node	(k) [listobj] at (2*#1,2*#2) {\scriptsize $(2|2)$};
  \node	(l) [listobj] at (3*#1,2*#2) {\scriptsize $(3|2)$};
  \node	(m) [listobj] at (0,3*#2)    {\scriptsize $(0|3)$};
  \node	(n) [listobj] at (#1,3*#2)   {\scriptsize $(1|3)$};
  \node	(o) [listobj] at (2*#1,3*#2) {\scriptsize $(2|3)$};
  \node	(p) [listobj] at (3*#1,3*#2) {\scriptsize $(3|3)$};
  \node	(q)           at (#3,#4) {$\big(\zet/4\zet\big)^2$: $3$ orbits};
  \path[commutative diagrams/.cd, every arrow, every label]
  (a) edge [\red, out=155,in=205,loop=2mm] (a)
  (e) edge [\red, out=155,in=205,loop=2mm] (e)
  (i) edge [\red, out=155,in=205,loop=2mm] (i)
  (m) edge [\red, out=155,in=205,loop=2mm] (m)
  (n) edge [\red]               (j)
  (j) edge [\red]               (f)
  (f) edge [\red]               (b)
  (b) edge [\red,bend left=15]  (n)
  (o) edge [\red,bend left]     (g)
  (k) edge [\red,bend left]     (c)
  (c) edge [\red,bend left]     (k)
  (g) edge [\red,bend left]     (o)
  (d) edge [\red]               (h)
  (h) edge [\red]               (l)
  (l) edge [\red]               (p)
  (p) edge [\red,bend left]     (d)
  (a) edge [\blue, out=65,in=120,loop=1mm] (a)
  (b) edge [\blue]              (e)
  (e) edge [\blue]              (d)
  (d) edge [\blue,bend left=10]	(m)
  (m) edge [\blue]              (b)
  (c) edge [\blue,bend left=15]	(i)
  (i) edge [\blue,bend left=15]	(c)
  (f) edge [\blue,bend right]   (h)
  (h) edge [\blue,bend left]    (p)
  (p) edge [\blue,bend right]   (n)
  (n) edge [\blue,bend left]    (f)
  (j) edge [\blue]              (g)
  (g) edge [\blue]              (l)
  (l) edge [\blue]              (o)
  (o) edge [\blue]              (j)
  (k) edge [\blue, out=60,in=110,loop=2mm] (k)
   ;
  \end{tikzpicture}
 }

\newcommand\diagZfiveZfiveTwoorbits[4] 
 { \begin{tikzpicture}[scale=2,commutative diagrams/every diagram]
  \node	(a) [listobj] at (0,0)       {\scriptsize $(0|0)$};
  \node	(b) [listobj] at (#1,0)      {\scriptsize $(1|0)$};
  \node	(c) [listobj] at (2*#1,0)    {\scriptsize $(2|0)$};
  \node	(d) [listobj] at (3*#1,0)    {\scriptsize $(3|0)$};
  \node	(e) [listobj] at (4*#1,0)    {\scriptsize $(4|0)$};
  \node	(f) [listobj] at (0,#2)      {\scriptsize $(0|1)$};
  \node	(g) [listobj] at (#1,#2)     {\scriptsize $(1|1)$};
  \node	(h) [listobj] at (2*#1,#2)   {\scriptsize $(2|1)$};
  \node	(i) [listobj] at (3*#1,#2)   {\scriptsize $(3|1)$};
  \node	(j) [listobj] at (4*#1,#2)   {\scriptsize $(4|1)$};
  \node	(k) [listobj] at (0,2*#2)    {\scriptsize $(0|2)$};
  \node	(l) [listobj] at (#1,2*#2)   {\scriptsize $(1|2)$};
  \node	(m) [listobj] at (2*#1,2*#2) {\scriptsize $(2|2)$};
  \node	(n) [listobj] at (3*#1,2*#2) {\scriptsize $(3|2)$};
  \node	(o) [listobj] at (4*#1,2*#2) {\scriptsize $(4|2)$};
  \node	(p) [listobj] at (0,3*#2)    {\scriptsize $(0|3)$};
  \node	(q) [listobj] at (#1,3*#2)   {\scriptsize $(1|3)$};
  \node	(r) [listobj] at (2*#1,3*#2) {\scriptsize $(2|3)$};
  \node	(s) [listobj] at (3*#1,3*#2) {\scriptsize $(3|3)$};
  \node	(t) [listobj] at (4*#1,3*#2) {\scriptsize $(4|3)$};
  \node	(u) [listobj] at (0,4*#2)    {\scriptsize $(0|4)$};
  \node	(v) [listobj] at (#1,4*#2)   {\scriptsize $(1|4)$};
  \node	(x) [listobj] at (2*#1,4*#2) {\scriptsize $(2|4)$};
  \node	(y) [listobj] at (3*#1,4*#2) {\scriptsize $(3|4)$};
  \node	(z) [listobj] at (4*#1,4*#2) {\scriptsize $(4|4)$};
  \node	(aa)          at (#3,#4) {$\big(\zet/5\zet\big)^2$: $2$ orbits};
  \path[commutative diagrams/.cd, every arrow, every label]
  (a) edge [\red, out=155,in=205,loop=2mm] (a)
  (f) edge [\red, out=155,in=205,loop=2mm] (f)
  (k) edge [\red, out=155,in=205,loop=2mm] (k)
  (p) edge [\red, out=155,in=205,loop=2mm] (p)
  (u) edge [\red, out=155,in=205,loop=2mm] (u)
  (v) edge [\red]                          (q)
  (q) edge [\red]                          (l)
  (l) edge [\red]                          (g)
  (g) edge [\red]                          (b)
  (b) edge [\red,bend left=15]             (v)
  (x) edge [\red,bend left=20]             (m)
  (r) edge [\red,bend left=20]             (h)
  (m) edge [\red,bend left=20]             (c)
  (c) edge [\red,bend left=15]             (r)
  (h) edge [\red,bend left=15]             (x)
  (d) edge [\red,bend left=20]             (n)
  (i) edge [\red,bend left=20]             (s)
  (n) edge [\red,bend left=20]             (y)
  (y) edge [\red,bend left=15]             (i)
  (s) edge [\red,bend left=15]             (d)
  (e) edge [\red]                          (j)
  (j) edge [\red]                          (o)
  (o) edge [\red]                          (t)
  (t) edge [\red]                          (z)
  (z) edge [\red,bend left]                (e)
  (a) edge [\blue, out=65,in=120,loop=1mm] (a)
  (b) edge [\blue]                         (f)
  (f) edge [\blue]                         (e)
  (e) edge [\blue,bend left=10]	           (u)
  (u) edge [\blue]                         (b)
  (c) edge [\blue,bend left=15]	           (k)
  (k) edge [\blue]                         (d)
  (d) edge [\blue,bend right=10]           (p)
  (p) edge [\blue]                         (c)
  (g) edge [\blue,bend right]              (j)
  (j) edge [\blue,bend left]               (z)
  (z) edge [\blue,bend right]              (v)
  (v) edge [\blue,bend left]               (g)
  (h) edge [\blue]                         (o)
  (o) edge [\blue]                         (y)
  (y) edge [\blue]                         (q)
  (q) edge [\blue]                         (h)
  (i) edge [\blue]                         (t)
  (t) edge [\blue]                         (x)
  (x) edge [\blue]                         (l)
  (l) edge [\blue]                         (i)
  (n) edge [\blue]                         (s)
  (s) edge [\blue]                         (r)
  (r) edge [\blue]                         (m)
  (m) edge [\blue]                         (n)
   ;
  \end{tikzpicture}
 }

 \newcommand\diagZthreeZthreeModPtwo[4]
 {
  \begin{tikzpicture}[scale=2,commutative diagrams/every diagram]
  \node	(a) [listobj] at (0,0)     {\scriptsize $[(0|0)]$};
  \node	(b) [listobj] at (#1,0)    {\scriptsize $[(1|0)]$};
  \node	(d) [listobj] at (0,#2)    {\scriptsize $[(0|1)]$};
  \node	(e) [listobj] at (#1,#2)   {\scriptsize $[(1|1)]$};
  \node	(f) [listobj] at (2*#1,#2) {\scriptsize $[(2|1)]$};
  \node	(j)           at (#3,#4) {$\big(\zet/3\zet\big)^2/\mathrm{p}_2$};
  \path[commutative diagrams/.cd, every arrow, every label]
  (a) edge [\red, out=155,in=205,loop=2mm] (a)
  (d) edge [\red, out=155,in=205,loop=2mm] (d)
  (e) edge [\red] (b)
  (b) edge [\red] (f)
  (f) edge [\red] (e)
  (a) edge [\blue, out=65,in=120,loop=1mm] (a)
  (b) edge [\blue,<->]           (d)
  (e) edge [\blue,<->,bend left] (f)
   ;
  \end{tikzpicture}
  \label{fig:Z3modP2}
 }

\newcommand\diagZfourZfourModPthree[4]
 { \begin{tikzpicture}[scale=2,commutative diagrams/every diagram]
  \node	(a) [listobj] at (0,0)       {\scriptsize $[(0|0)]$};
  \node	(b) [listobj] at (#1,0)      {\scriptsize $[(1|0)]$};
  \node	(c) [listobj] at (2*#1,0)    {\scriptsize $[(2|0)]$};
  \node	(e) [listobj] at (0,#2)      {\scriptsize $[(0|1)]$};
  \node	(f) [listobj] at (#1,#2)     {\scriptsize $[(1|1)]$};
  \node	(g) [listobj] at (2*#1,#2)   {\scriptsize $[(2|1)]$};
  \node	(h) [listobj] at (3*#1,#2)   {\scriptsize $[(3|1)]$};
  \node	(i) [listobj] at (0,2*#2)    {\scriptsize $[(0|2)]$};
  \node	(j) [listobj] at (#1,2*#2)   {\scriptsize $[(1|2)]$};
  \node	(k) [listobj] at (2*#1,2*#2) {\scriptsize $[(2|2)]$};
  \node	(q)           at (#3,#4) {$\big(\zet/4\zet\big)^2/\mathrm{p}_3$};
  \path[commutative diagrams/.cd, every arrow, every label]
  (a) edge [\red, out=155,in=205,loop=2mm]  (a)
  (e) edge [\red, out=155,in=205,loop=2mm]  (e)
  (i) edge [\red, out=155,in=205,loop=2mm]  (i)
  (g) edge [\red, out=150,in=185,loop=2mm]  (g)
  (b) edge [\red]  (h)
  (c) edge [<->,\red,bend right]  (k)
  (h) edge [\red]  (j)
  (j) edge [\red]  (f)
  (f) edge [\red]  (b)
  (a) edge [\blue, out=65,in=120,loop=1mm]  (a)
  (b) edge [<->,\blue]  (e)
  (c) edge [<->,\blue,bend left]  (i)
  (f) edge [<->,\blue,bend right=20]  (h)
  (g) edge [<->,\blue]  (j)
  (k) edge [\blue, out=25,in=-25,loop=2mm]  (k)
   ;
  \end{tikzpicture}
 }

 \newcommand\diagZfiveZfiveModPtwo[4]
 { \begin{tikzpicture}[scale=2,commutative diagrams/every diagram]
  \node	(a) [listobj] at (0,0)     {\scriptsize $[(0|0)]$};
  \node	(b) [listobj] at (#1,0)    {\scriptsize $[(1|0)]$};
  \node	(f) [listobj] at (0,#2)    {\scriptsize $[(0|1)]$};
  \node	(g) [listobj] at (#1,#2)   {\scriptsize $[(1|1)]$};
  \node	(h) [listobj] at (2*#1,#2) {\scriptsize $[(2|1)]$};
  \node	(i) [listobj] at (3*#1,#2) {\scriptsize $[(3|1)]$};
  \node	(j) [listobj] at (4*#1,#2) {\scriptsize $[(4|1)]$};
  \node	(aa) at (#3,#4) {$\big(\zet/5\zet\big)^2/\mathrm{p}_2$};
  \path[commutative diagrams/.cd, every arrow, every label]
  (a) edge [\red, out=155,in=205,loop=2mm]  (a)
  (f) edge [\red, out=155,in=205,loop=2mm]  (f)
  (b) edge [\red]  (j)
  (j) edge [\red,bend right=20]	  (h)
  (h) edge [\red]  (i)
  (i) edge [\red,bend left]  (g)
  (g) edge [\red]  (b)
  (a) edge [\blue, out=120,in=65,loop=1mm]  (a)
  (b) edge [<->,\blue]  (f)
  (g) edge [<->,\blue,bend left]  (j)
  (h) edge [\blue,out=160,in=200,loop=2mm]  (h)
  (i) edge [\blue,out=16,in=-14,loop=2mm]  (i)
   ;	
  \end{tikzpicture}
 }

 \newcommand\diagZfiveZfiveModPfour[4]
 { \begin{tikzpicture}[scale=2,commutative diagrams/every diagram]
  \node	(a) [listobj] at (0,0)       {\scriptsize $[(0|0)]$};
  \node	(b) [listobj] at (#1,0)      {\scriptsize $[(1|0)]$};
  \node	(c) [listobj] at (2*#1,0)    {\scriptsize $[(2|0)]$};
  \node	(f) [listobj] at (0,#2)      {\scriptsize $[(0|1)]$};
  \node	(g) [listobj] at (#1,#2)     {\scriptsize $[(1|1)]$};
  \node	(h) [listobj] at (2*#1,#2)   {\scriptsize $[(2|1)]$};
  \node	(i) [listobj] at (3*#1,#2)   {\scriptsize $[(3|1)]$};
  \node	(j) [listobj] at (4*#1,#2)   {\scriptsize $[(4|1)]$};
  \node	(k) [listobj] at (0,2*#2)    {\scriptsize $[(0|2)]$};
  \node	(l) [listobj] at (#1,2*#2)   {\scriptsize $[(1|2)]$};
  \node	(m) [listobj] at (2*#1,2*#2) {\scriptsize $[(2|2)]$};
  \node	(n) [listobj] at (3*#1,2*#2) {\scriptsize $[(3|2)]$};
  \node	(o) [listobj] at (4*#1,2*#2) {\scriptsize $[(4|2)]$};
  \node	(aa) at	(#3,#4) {$\big(\zet/5\zet\big)^2/\mathrm{p}_4$};
  \path[commutative diagrams/.cd, every arrow, every label]
  (a) edge [\red, out=155,in=205,loop=2mm]  (a)
  (f) edge [\red, out=155,in=205,loop=2mm]  (f)
  (k) edge [\red, out=155,in=205,loop=2mm]  (k)
  (b) edge [\red]  (j)
  (j) edge [\red]  (o)
  (o) edge [\red,bend right]  (l)
  (l) edge [\red]  (g)
  (g) edge [\red]  (b)
  (c) edge [\red]  (n)
  (n) edge [\red]  (h)
  (h) edge [\red]  (i)
  (i) edge [\red]  (m)
  (m) edge [\red,bend right]  (c)
  (a) edge [\blue, out=120,in=65,loop=1mm]  (a)
  (b) edge [<->,\blue]	(f)
  (c) edge [<->,\blue,bend left=20]  (k)
  (g) edge [<->,\blue,bend right=15] (j)
  (l) edge [<->,\blue]  (i)
  (h) edge [<->,\blue]  (o)
  (m) edge [<->,\blue]  (n)
   ;
  \end{tikzpicture}
 }

 \newcommand\diagDevenPoddMuOneInvCe[1]
 { \begin{tikzpicture}[scale=1.5,commutative diagrams/every diagram]
  \node (a) [listobj] at (180:#1) {\scriptsize $(x^p|y)$};
  \node (b) [listobj] at (120:#1) {\scriptsize $(y|x^p)$};
  \node (c) [listobj] at (60:#1)  {\scriptsize $(y|x^py)$};
  \node (d) [listobj] at (0:#1)   {\scriptsize $(xy|x^{p+1}y)$};
  \node (e) [listobj] at (-60:#1) {\scriptsize $(xy|x^p)$};
  \node (f) [listobj] at (-120:#1){\scriptsize $(x^p|xy)$};
  \path[commutative diagrams/.cd, every arrow, every label]
  (a) edge [<->,\blue,bend left] (b)
  (c) edge [<->,\blue,bend left] (d)
  (e) edge [<->,\blue,bend left] (f)
  (b) edge [<->,\red,bend left]	 (c)
  (d) edge [<->,\red,bend left]	 (e)
  (f) edge [<->,\red,bend left]	 (a)
   ;
  \end{tikzpicture}
 }

 \newcommand\diagDevenPevenMuOneInvCe[2]
 { \begin{tikzpicture}[scale=1.5,commutative diagrams/every diagram]
 \node (a) [listobj] at (-#1,#2) {\scriptsize$(x^p|y)$};
 \node (b) [listobj] at (0,#2)   {\scriptsize$(y|x^p)$};
 \node (c) [listobj] at (#1,#2)  {\scriptsize$(y|x^py)$};
 \node (e) [listobj] at (-#1,0)  {\scriptsize$(x^p|xy)$};
 \node (f) [listobj] at (0,0)    {\scriptsize$(xy|x^p)$};
 \node (g) [listobj] at (#1,0)   {\scriptsize$(xy|x^{p+1}y)$};
 \path[commutative diagrams/.cd, every arrow, every label]
 (a) edge [\red,out=145,in=215,loop=5mm]  (a)
 (b) edge [\blue,<->]  (a)
 (c) edge [\red,,<->]  (b)
 (c) edge [\blue,out=-35,in=35,loop=5mm]  (c)
 (e) edge [\red,out=145,in=215,loop=5mm]  (e)
 (f) edge [\blue,,<->] (e)
 (g) edge [\red,,<->]  (f)
 (g) edge [\blue,out=-35,in=35,loop=5mm]  (g);
  \end{tikzpicture}
 }

 \newcommand\diagDevenPevenMuOneInvCxKodd[3]
 { \begin{tikzpicture}[scale=1.5,commutative diagrams/every diagram]
 \node (s)           at (#2,#3)  {$\mu_1\inv(c_{x^{2k}})\,,~k\iN 2\zet{+}1$:};
 \node (a) [listobj] at (180:#1) {\scriptsize$(x^k|y)$};
 \node (b) [listobj] at (120:#1) {\scriptsize$(y|x^k)$};
 \node (c) [listobj] at (60:#1)  {\scriptsize$(y|x^ky)$};
 \node (d) [listobj] at (0:#1)   {\scriptsize $(xy|x^{k+1}y)$};
 \node (e) [listobj] at (-60:#1) {\scriptsize $(xy|x^k)$};
 \node (f) [listobj] at (-120:#1){\scriptsize $(x^k|xy)$};
 \path[commutative diagrams/.cd, every arrow, every label]
 (a) edge [<->,\blue,bend left]  (b)
 (c) edge [<->,\blue,bend left]  (d)
 (e) edge [<->,\blue,bend left]  (f)
 (b) edge [<->,\red,bend left]   (c)
 (d) edge [<->,\red,bend left]   (e)
 (f) edge [<->,\red,bend left]   (a)
  ;
  \end{tikzpicture}
 } 

 \newcommand\diagDevenPevenMuOneInvCxkk[2]
 { \begin{tikzpicture}[scale=1.5,commutative diagrams/every diagram]
  \node	(s)  at	(-4.5,-0.6) {$\mu_1\inv(c_{x^{2k}})$\,,~$k\iN 2\zet$:};
  \node	(a)  [listobj]	at (-#1,0)  {\scriptsize$(x^k|y)$};
  \node	(b)  [listobj]	at (0,0)    {\scriptsize$(y|x^k)$};
  \node	(c)  [listobj]	at (#1,0)   {\scriptsize$(y|x^ky)$};
  \node	(e)  [listobj]	at (-#1,-#2){\scriptsize$(x^k|xy)$};
  \node	(f)  [listobj]	at (0,-#2)  {\scriptsize$(xy|x^k)$};
  \node	(g)  [listobj]	at (#1,-#2) {\scriptsize$(xy|x^{k+1}y)$};
  \path[commutative diagrams/.cd, every arrow, every label]
  (a) edge  [\red,out=145,in=215,loop=5mm] (a)
  (b) edge  [\blue,<->] (a)
  (c) edge  [\red,<->]  (b)
  (c) edge  [\blue,out=-35,in=35,loop=5mm] (c)
  (e) edge  [\red,out=145,in=215,loop=5mm] (e)
  (f) edge  [\blue,<->] (e)
  (g) edge  [\red,<->]  (f)
  (g) edge  [\blue,out=-35,in=35,loop=5mm] (g);
  \end{tikzpicture}
  }

 \newcommand\diagQppppPodddiconj[4]
 { \begin{tikzpicture}[commutative diagrams/every diagram]
  \node	(a) [listobj]   at (-2*#1,0)   {\scriptsize$(e|y)$};
  \node	(b) [listobj]   at (-#1,0)     {\scriptsize$(y|e)$};
  \node	(c) [listobj]   at (#4,#2)     {\scriptsize$(y|y^3)$};
  \node	(d) [listobj]   at (-#4,-#2)   {\scriptsize$(y|y)$};
  \node	(e) [listobj]   at (#1,0)      {\scriptsize$(y|y^2)$};
  \node	(f) [listobj]   at (2*#1,0)    {\scriptsize$(y^2|y)$};
  \node	(aa) [listobj]  at (-2*#1,-2*#2-#3) {\scriptsize$(e|xy)$};
  \node	(bb) [listobj]  at (-#1,-2*#2-#3)   {\scriptsize$(xy|e)$};
  \node	(cc) [listobj]  at (-#4,-#2-#3)     {\scriptsize$(xy|y^2xy)$};
  \node	(dd) [listobj]  at (#4,-3*#2-#3)    {\scriptsize$(xy|xy)$};
  \node	(ee) [listobj]  at (#1,-2*#2-#3)    {\scriptsize$(xy|y^2)$};
  \node	(ff) [listobj]  at (2*#1,-2*#2-#3)  {\scriptsize$(y^2|xy)$};		
  \path[commutative diagrams/.cd, every arrow, every label]
  (a)  edge [\red,out=145,in=215,loop=5mm]  (a)
  (b)  edge [\blue]  (a)
  (b)  edge [\red,bend left]  (c)
  (c)  edge [\red,bend left]  (e)
  (c)  edge [\blue]  (d)
  (e)  edge [\red,bend left]  (d)
  (d)  edge [\red,bend left]  (b)
  (e)  edge [\blue]  (f)		
  (aa) edge [\red,out=145,in=215,loop=5mm]  (aa)
  (bb) edge [\blue]  (aa)
  (bb) edge [\red,bend left]  (cc)
  (cc) edge [\red,bend left]  (ee)
  (cc) edge [\blue]  (dd)
  (ee) edge [\red,bend left]  (dd)
  (dd) edge [\red,bend left]  (bb)
  (ee) edge [\blue]  (ff)		
  (a)  edge [\blue]  (bb)
  (aa) edge [\blue]  (b)
  (f)  edge [\blue]  (ee)
  (ff) edge [\blue]  (e)
  (d)  edge [\blue]  (cc)
  (dd) edge [\blue,bend right=15]  (c)
  (f)  edge [\red,<->, bend left]  (ff)
  ;
  \end{tikzpicture}
  }

 \newcommand\diagQuaternionsB[7]
 { \begin{tikzpicture}[rotate=90,commutative diagrams/every diagram]
  \node	(s) at           (#4,#5)   {$\mu_1\inv(c_{x^{4l+2}}):$};
  \node	(a) [listobj] at (-#2-#3,-#7) {\tiny$(x^{p-2l-1}|xy)$};
  \node	(b) [listobj] at (-#3,0)   {\tiny$(xy|x^{p-2l-1})$};
  \node	(c) [listobj] at (0,#1)	   {\tiny$(xy|x^{2l+2}y)$};
  \node	(d) [listobj] at (0,-#1)   {\tiny$(xy|x^{p-2l}y)$};
  \node	(e) [listobj] at (#3,0)    {\tiny$(xy|x^{2l+1})$};
  \node	(f) [listobj] at (#2+#3,-#6) {\tiny$(x^{2l+1}|xy)$};
  \node	(g) [listobj] at (#2+#3,#6)  {\tiny$(x^{2l+1}|y)$};
  \node	(h) [listobj] at (2*#2+#3,0)    {\tiny$(y|x^{2l+1})$};
  \node	(i) [listobj] at (2*#2+2*#3,-#1){\tiny$(y|x^{2l+1}y)$};
  \node	(j) [listobj] at (2*#2+2*#3,#1) {\tiny$(y|x^{p-2l-1}y)$};
  \node	(k) [listobj] at (2*#2+3*#3,0)  {\tiny$(y|x^{p-2l-1})$};
  \node	(l) [listobj] at (3*#2+3*#3,-#7){\tiny$(x^{p-2l-1}|y)$};
  \path[commutative diagrams/.cd, every arrow, every label]
  (a) edge [\blue,<->]       (b)
  (b) edge [\red,bend left]  (c)
  (c) edge [\red,bend left]  (e)
  (e) edge [\red,bend left]  (d)
  (d) edge [\red,bend left]  (b)
  (e) edge [\blue,<->]       (f)
  (f) edge [\red,<->]        (g)
  (g) edge [\blue,<->]       (h)
  (c) edge [\blue,bend left,<->]  (j)
  (i) edge [\blue,bend left,<->]  (d)
  (h) edge [\red,bend left]  (j)
  (j) edge [\red,bend left]  (k)
  (k) edge [\red,bend left]  (i)
  (i) edge [\red,bend left]  (h)
  (k) edge [\blue,<->]       (l)
  (l) edge [\red,bend left=80,<->]  (a)
   ;
  \end{tikzpicture}
 }

 \newcommand\diagQuaternionsC[6]
 { \begin{tikzpicture}[commutative diagrams/every diagram]
  \node	(s)           at  (#3,#4)  {$\mu_1\inv(c_{x^{4l}}):$};
  \node	(t)           at  (#5,#6)  {$p\in 2\zet +1$};
  \node	(a) [listobj] at  (-#1,0)    {\scriptsize$(x^{2l}|y)$};
  \node	(b) [listobj] at  (0,#2)     {\scriptsize$(xy|x^{2l})$};
  \node	(c) [listobj] at  (0,-#2)    {\scriptsize$(y|x^{2l})$};
  \node	(d) [listobj] at  (#1,0)     {\scriptsize$(x^{2l}|xy)$};
  \node	(e) [listobj] at  (-#1,-2*#2){\scriptsize$(y|x^{2l}y)$};
  \node	(f) [listobj] at  (#1,-2*#2) {\scriptsize$(y|x^{p-2l}y)$};
  \node	(g) [listobj] at  (0,-3*#2)  {\scriptsize$(y|x^{p-2l})$};
  \node	(h) [listobj] at  (-#1,-4*#2){\scriptsize$(x^{p-2l}|xy)$};
  \node	(i) [listobj] at  (#1,-4*#2) {\scriptsize$(x^{p-2l}|y)$};
  \node	(j) [listobj] at  (0,-5*#2)  {\scriptsize$(xy|x^{p-2l})$};
  \node	(k) [listobj] at  (-3*#1,-2*#2) {\scriptsize$(xy|x^{p-2l}xy)$};
  \node	(l) [listobj] at  (3*#1,-2*#2)  {\scriptsize$(xy|x^{2l}xy)$};	
  \path[commutative diagrams/.cd, every arrow, every label]
  (a)  edge  [\red,out=175,in=225,loop=2mm]  (a)
  (d)  edge  [\red,out=5,in=-35,loop=2mm]  (d)
  (a)  edge  [\blue,bend left] (b)
  (b)  edge  [\blue,bend left] (d)
  (d)  edge  [\blue,bend left] (c)
  (c)  edge  [\blue,bend left] (a)
  (e)  edge  [\red,bend left]  (c)
  (c)  edge  [\red,bend left]  (f)
  (f)  edge  [\red,bend left]  (g)
  (g)  edge  [\red,bend left]  (e)
  (h)  edge  [\blue,bend left] (g)
  (g)  edge  [\blue,bend left] (i)
  (i)  edge  [\blue,bend left] (j)
  (j)  edge  [\blue,bend left] (h)
  (j)  edge  [\red,bend right=40]  (l)
  (l)  edge  [\red,bend right=40]  (b)
  (b)  edge  [\red,bend right=40]  (k)
  (k)  edge  [\red,bend right=40]  (j)
  (l)  edge  [\blue]               (f)
  (f)  edge  [\blue]               (e)
  (e)  edge  [\blue]               (k)
  (k)  edge  [\blue,bend right=15] (l)
  (h)  edge  [\red,<->]	           (i)
  ;
  \end{tikzpicture}
  }

 \newcommand\diagQuaternionsE[4]
  { \begin{tikzpicture}[scale=1.5,commutative diagrams/every diagram]
  \node	(s)             at  (#3,#4)   {$\mu_1\inv(c_{y^2})$\,,~$p\iN 4\zet$\,:};
  \node (a)  [listobj]  at  (-#1,0)   {\scriptsize$(x^{p/2}|y)$};
  \node (b)  [listobj]  at  (0,0)     {\scriptsize$(y|x^{p/2})$};
  \node (c)  [listobj]  at  (#1,0)    {\scriptsize$(y|x^{p/2}y)$};
  \node (e)  [listobj]  at  (-#1,-#2) {\scriptsize$(x^{p/2}|xy)$};
  \node (f)  [listobj]  at  (0,-#2)   {\scriptsize$(xy|x^{p/2})$};
  \node (g)  [listobj]  at  (#1,-#2)  {\scriptsize$(xy|x^{p/2+1}y)$};
  \path[commutative diagrams/.cd, every arrow, every label]
  (a)  edge  [\red,out=145,in=215,loop=5mm]  (a)
  (a)  edge  [\blue,<->]  (b)
  (b)  edge  [\red,<->]   (c)
  (c)  edge  [\blue,out=-35,in=35,loop=5mm]  (c)
  (e)  edge  [\red,out=145,in=215,loop=5mm]  (e)
  (e)  edge  [\blue,<->]  (f)
  (f)  edge  [\red,<->]   (g)
  (g)  edge  [\blue,out=-35,in=35,loop=5mm]  (g);
  \end{tikzpicture}
  }

 \newcommand\diagQuaternionsF[4]
  { \begin{tikzpicture}[scale=1.5,commutative diagrams/every diagram]
  \node	(s)             at  (-4.2,0) {$\mu_1\inv(c_{y^2})$\,,~$p\iN 4\zet{+}2$\,:};
  \node	(a)  [listobj]  at  (3*#1:#2)	  {\scriptsize$(x^{p/2}|y)$};
  \node	(b)  [listobj]  at  (2*#1:#2)  {\scriptsize$(y|x^{p/2})$};
  \node	(c)  [listobj]  at  (#1:#2)  {\scriptsize$(y|x^{p/2}y)$};
  \node	(d)  [listobj]  at  (0:#2)  {\scriptsize $(xy|x^{p/2+1}y)$};
  \node	(e)  [listobj]  at  (-#1:#2)	  {\scriptsize $(xy|x^{p/2})$};
  \node	(f)  [listobj]  at  (-2*#1:#2)  {\scriptsize $(x^{p/2}|xy)$};
  \path[commutative diagrams/.cd, every arrow, every label]
  (a)  edge  [<->,\blue,bend left]  (b)
  (c)  edge  [<->,\blue,bend left]  (d)
  (e)  edge  [<->,\blue,bend left]  (f)
  (b)  edge  [<->,\red,bend left]   (c)
  (d)  edge  [<->,\red,bend left]   (e)
  (f)  edge  [<->,\red,bend left]   (a)
  ;
  \end{tikzpicture}
  }

 \newcommand\diagSfiveMuinvCa[4]
  { \hspace*{-1.7em} \begin{tikzpicture}[scale=1.5,commutative diagrams/every diagram]
  \node (s)             at  (#3,#4)     {$\mu_1\inv(c_a)$:};
  \node (a)  [listobj]  at  (-#1,0)     {\scriptsize$(b^2|a)$};
  \node (b)  [listobj]  at  (-2*#1,#2)  {\scriptsize$(b^2|b^a)$};
  \node (c)  [listobj]  at  (-2*#1,-#2) {\scriptsize$(b^2|b^c)$};
  \node (d)  [listobj]  at  (#1,0)      {\scriptsize$(b|a)$};
  \node (f)  [listobj]  at  (2*#1,#2)   {\scriptsize$(b^2|b^d)$};
  \node (e)  [listobj]  at  (2*#1,-#2)  {\scriptsize$(b^2|ac)$};
  \node (g)  [listobj]  at  (0,-#2)     {\scriptsize$(a|b^2)$};
  \node (h)  [listobj]  at  (0,#2)      {\scriptsize$(a|b)$};
  \path[commutative diagrams/.cd, every arrow, every label]
  (a)  edge  [\red,bend left]   (c)
  (c)  edge   [\red,bend right] (b)
  (b)  edge  [\red,bend left]   (a)
  (d)  edge  [\red,bend left]   (f)
  (f)  edge   [\red,bend right] (e)
  (e)  edge  [\red,bend left]   (d)
  (a)  edge  [\blue]     (h)
  (h)  edge  [\blue]     (d)
  (d)  edge  [\blue]     (g)
  (g)  edge  [\blue]     (a)
  (b)  edge  [\blue,in=165,out=15]  (f)
  (f)  edge  [\blue,bend left]      (e)
  (e)  edge  [\blue,in=-15,out=195] (c)
  (c)  edge  [\blue,bend left]      (b)
  ;
	\end{tikzpicture}
  }

 \newcommand\diagSfiveMuinvCn[5]
 { \hspace*{-1.6em} \begin{tikzpicture}[commutative diagrams/every diagram]
  \node (b2)   [listobj]  at  (0,-#2-2*#3)   {\scriptsize$(b|b^{(c^2)})$};
  \node (b3)   [listobj]  at  (#1,-#2-#3) {\scriptsize$(b|c)$};
  \node (b4)   [listobj]  at  (-#1,-#2-#3){\scriptsize$(b^2|d)$};
  \node (b5)   [listobj]  at  (0,#2+2*#3)    {\scriptsize$(b|(b^2)^{(c^2)})$};
  \node (b6)   [listobj]  at  (-#1,#2+#3) {\scriptsize$(b|d)$};
  \node (b7)   [listobj]  at  (#1,#2+#3)  {\scriptsize$(b^2|c)$};
  \node (b10)  [listobj]  at  (-#1,0)   {\scriptsize$(a|c^2)$};
  \node (b11)  [listobj]  at  (#1,0)    {\scriptsize$(a|c^d)$};
  \node (b12)  [listobj]  at  (#1+#4+3*#5,#2)   {\scriptsize$(c|b)$};
  \node (b13)  [listobj]  at  (#1+#4+4*#5,0)   {\scriptsize$(c|b^2)$};
  \node (b14)  [listobj]  at  (#1+#4,0)   {\scriptsize$(c|a^{(b^2)})$};
  \node (b15)  [listobj]  at  (#1+#4+#5,#2)   {\scriptsize$(c|d)$};
  \node (b16)  [listobj]  at  (#1+#4+2*#5,-#2)  {\scriptsize$(c^4|d)$};
  \node (b17)  [listobj]  at  (-#1-#4-3*#5,#2)  {\scriptsize$(d|b^2)$};
  \node (b18)  [listobj]  at  (-#1-#4-4*#5,0)  {\scriptsize$(d|b)$};
  \node (b19)  [listobj]  at  (-#1-#4,0)  {\scriptsize$(d|a^{(b^2)})$};
  \node (b20)  [listobj]  at  (-#1-#4-2*#5,-#2) {\scriptsize$(d|c)$};
  \node (b21)  [listobj]  at  (-#1-#4-#5,#2)  {\scriptsize$(d|c^4)$};
  \path[commutative diagrams/.cd, every arrow, every label]
  (b2)	 edge  [\red,bend right]  (b3)
  (b3)	 edge  [\red,bend right]  (b4)
  (b4)	 edge  [\red,bend right]  (b2)
  (b5)	 edge  [\red,bend right]  (b6)
  (b6)	 edge  [\red,bend right]  (b7)
  (b7)	 edge  [\red,bend right]  (b5)
  (b12)	 edge  [\red,bend right]  (b15)
  (b15)	 edge  [\red,bend right]  (b14)
  (b14)	 edge  [\red,bend right]  (b16)
  (b16)	 edge  [\red,bend right]  (b13)
  (b13)	 edge  [\red,bend right]  (b12)
  (b17)	 edge  [\red,bend right]  (b18)
  (b18)	 edge  [\red,bend right]  (b20)
  (b20)	 edge  [\red,bend right]  (b19)
  (b19)	 edge  [\red,bend right]  (b21)
  (b21)	 edge  [\red,bend right]  (b17)
  (b10)  edge  [\red,<->]         (b11)
  (b4)	 edge  [\blue,<->,in=250,out=180] (b18)
  (b3)	 edge  [\blue,<->,in=290,out=0]	  (b13)
  (b10)	 edge  [\blue,<->]   (b19)
  (b11)	 edge  [\blue,<->]   (b14)
  (b21)	 edge  [\blue,<->]   (b15)
  (b16)	 edge  [\blue,<->]   (b20)
  (b17)	 edge  [\blue,<->,in=180,out=70]  (b6)
  (b12)  edge  [\blue,<->,in=0,out=110]   (b7)
  (b2)   edge  [\blue,<->]   (b5)
  ;
  \end{tikzpicture}
  }

 \newcommand\diagQuaternionsD[8]
  { \begin{tikzpicture}[commutative diagrams/every diagram]
  \node	(s)             at  (#3,#4)   {$\mu_1\inv(c_{x^{4l}}):$};
  \node	(t)             at  (#5,#6)   {$p\in 2\zet$};
  \node	(a)  [listobj]	at  (-#1-#8,0){\scriptsize$(x^{2l}|y)$};
  \node	(b)  [listobj]	at  (-#1,0)   {\scriptsize$(y|x^{2l})$};
  \node	(c)  [listobj]	at  (0,#2)    {\scriptsize$(y|x^{p-2l}y)$};
  \node	(d)  [listobj]	at  (0,-#2)   {\scriptsize$(y|x^{2l}y)$};
  \node	(e)  [listobj]	at  (#1,0)    {\scriptsize$(y|x^{p-2l})$};
  \node	(f)  [listobj]	at  (#1+#8,0) {\scriptsize$(x^{p-2l}|y)$};
  \path[commutative diagrams/.cd, every arrow, every label]
  (a)  edge  [\red,out=155,in=205,loop=2mm]		(a)
  (a)  edge  [\blue,<->]     		(b)
  (b)  edge  [\red,bend left]		(c)
  (c)  edge  [\red,bend left]		(e)
  (c)  edge  [\blue,<->]	        	(d)
  (e)  edge  [\red,bend left]		(d)
  (d)  edge  [\red,bend left]		(b)
  (e)  edge  [\blue,<->]     		(f)
  (f)  edge  [\red,out=-25,in=25,loop=2mm]		(f);
  \begin{scope}[yshift=#7 cm]
  \node	(aa)  [listobj]	at  (-#1-#8,0) {\scriptsize$(x^{2l}|xy)$};
  \node	(bb)  [listobj]	at  (-#1,0)   {\scriptsize$(xy|x^{2l})$};
  \node	(cc)  [listobj]	at  (0,#2)    {\scriptsize$(xy|x^{p-2l}xy)$};
  \node	(dd)  [listobj]	at  (0,-#2)   {\scriptsize$(xy|x^{2l}xy)$};
  \node	(ee)  [listobj]	at  (#1,0)    {\scriptsize$(xy|x^{p-2l})$};
  \node	(ff)  [listobj]	at  (#1+#8,0) {\scriptsize$(x^{p-2l}|xy)$};
  \path[commutative diagrams/.cd, every arrow, every label]
  (aa)  edge  [\red,out=155,in=205,loop=2mm]		(aa)
  (aa)  edge  [\red,\blue,<->]		(bb)
  (bb)  edge  [\red,bend left]		(cc)
  (cc)  edge  [\red,bend left]		(ee)
  (cc)  edge  [\blue,<->]	        	(dd)
  (ee)  edge  [\red,bend left]		(dd)
  (dd)  edge  [\red,bend left]		(bb)
  (ee)  edge  [\blue,<->]        		(ff)
  (ff)  edge  [\red,out=-25,in=25,loop=2mm]		(ff);
  \end{scope}
  \end{tikzpicture}
  }

 \newcommand\diagSfiveMuinvCc[4]
  { \begin{tikzpicture}[commutative diagrams/every diagram]
  \node (c1)   [listobj]  at  (-2*#1,0)   {\scriptsize$(b|bd^2)$};
  \node (c2)   [listobj]  at  (-2*#1-#2,-#3)  {\scriptsize$(b|d^2)$};
  \node (c3)   [listobj]  at  (-2*#1-#2,#3)  {\scriptsize$(b|d^3)$};
  \node (c5)   [listobj]  at  (0,0)  {\scriptsize$(a|c)$};
  \node (c6)   [listobj]  at  (-#1,0)  {\scriptsize$(a|bcb)$};
  \node (c8)   [listobj]  at  (#1,0)  {\scriptsize$(c|a)$};
  \node (c9)   [listobj]  at  (#1+#2,-#3)  {\scriptsize$(c|[c,b])$};
  \node (c10)  [listobj]  at  (#1+#2,#3)  {\scriptsize$(c|[b,c])$};
  \node (c11)  [listobj]  at  (2*#1+#2,-#4)  {\scriptsize$(c|d^2)$};
  \node (c12)  [listobj]  at  (2*#1+#2,#4)  {\scriptsize$(c|d^3)$};
  \path[commutative diagrams/.cd, every arrow, every label]
  (c1)  edge  [\red,bend right]  (c3)
  (c3)  edge  [\red,bend right]  (c2)
  (c2)  edge  [\red,bend right]  (c1)
  (c5)  edge  [\red,<->]         (c6)
  (c8)  edge  [\red,bend right]  (c9)
  (c9)  edge  [\red,bend right]  (c11)
  (c11) edge  [\red,bend right]  (c12)
  (c12)	edge  [\red,bend right]  (c10)
  (c10)	edge  [\red,bend right]  (c8)
  (c1)  edge  [\blue,<->]  (c6)
  (c5)  edge  [\blue,<->]  (c8)
  (c2)  edge  [\blue,<->]  (c9)
  (c10) edge  [\blue,<->]  (c3)
  (c11)	edge  [\blue,<->,bend left]  (c12)
  ;
  \end{tikzpicture}
  }

\begin{document}

\title{Mapping class group representations \\ from Drinfeld doubles of finite groups}

\author{~\\ {\bf Jens Fjelstad}\,$^{\,a,b}$ ~~and~~ {\bf J\"urgen Fuchs}\,$^{\,c}$ \\[9mm]
    \large $^a$ Department of Mathematics, Uppsala University\\
    \large Box 480, \ S\,--\,751\,06\, Uppsala \\[9pt] 
    \large $^b$ Department of Mathematics, \"Orebro University\\
    \large Fakultetsgatan 1, \ S\,--\,701\,82\, Orebro \\[9pt]
    \large $^c$ Teoretisk fysik, \ Karlstads Universitet\\
    \large Universitetsgatan 21, \ S\,--\,651\,88\, Karlstad \\~
}

\def\today{}

\maketitle

\begin{abstract}
We investigate representations of mapping class groups of surfaces that arise 
from the untwisted Drinfeld double of a finite group $G$, focusing on surfaces 
without marked points or with one marked point. We obtain concrete descriptions 
of such representations in terms of finite group data. This allows us to establish
various properties of these representations. In particular we show that they have 
finite images, and that for surfaces of genus at least $3$ their restriction 
to the Torelli group is non-trivial iff $G$ is non-abelian.
\end{abstract}

\newpage
\tableofcontents
\newpage

%%%%%%%%%%%%%%%%%%%%%%%%%%%%%%%%%%%%%%%%%%%%%%%%%%%%%%%%%%%%%%%%

\section{Introduction}

Any factorizable ribbon Hopf algebra provides projective representations of 
mapping class groups of compact oriented surfaces of arbitrary genus with 
a finite collection of marked boundary circles \cite{lyub11}. 
Following (part of) the literature, we refer to such representations of mapping class 
groups as \emph{quantum representations}. 
If the Hopf algebra is semisimple, then these quantum representations are
associated with three-dimensional topological field theories \cite{retu2}.
Both for mathematical aspects of topological field theory and for 
interesting applications, like in quantum computing \cite{nssfd},
a detailed understanding of the quantum representations is desirable.

In this paper we investigate a subset of the 
quantum representations obtained from the class of factorizable ribbon Hopf algebras 
that consists of untwisted Drinfeld doubles of finite groups. The aim is to concretely 
describe these representations in terms of finite group data. 
For surfaces of genus 1 with empty boundary already much is known \cite{kssb,cogR}. 
The corresponding quantum representations have finite image; they are in fact 
permutation representations, and explicit expressions for the modular 
$S$- and $T$-matrices have been obtained. Moreover, the kernel of the representation
is then a congruence subgroup of $\slz$ \cite{coga2,bant14}.
For surfaces of genus $\Ng \,{>}\, 1$ with empty boundary one can invoke the 
expected relation with principal bundles described in the next paragraph to conclude
that the quantum representations have finite image, and that generically they do not factor
through to the symplectic group $\Sp(2\Ng,\zet)$ \cite{gann24}. For genus $0$ with $n$ marked 
boundary circles it is known \cite{etrw} that the quantum representations have finite 
image, and that the representations constructed from
the double of a finite $p$-group factor through a $p$-group.

Let $G$ be a finite group and $\DG$ its (untwisted) Drinfeld double over an 
algebraically closed field \ko. 
For $\ko \eq \complex$ the quantum representations obtained from $\DG$ using the methods of
\cite{lyub11} coincide with those obtained from the Reshetikhin-Turaev
topological field theory associated with the modular tensor category 
$\DG\Mod$ \cite{TUra}. They are also expected to match those obtained from the 
conformal blocks of a holomorphic orbifold CFT \cite{dvvv} as well as those
obtained from quantum gauge theory with finite gauge group \cite{diwi2}.
While coincidence with the representations obtained from the construction in 
\cite{diwi2} may be concluded from the statement in Example 4.5 of \cite{balKi}, 
we are not aware of decisive results about the equivalence for
the case of holomorphic orbifold CFT, though. 
  
In the case of \cite{diwi2} 
one considers the moduli space of principal $G$-bun\-dles over surfaces. For a closed 
oriented surface this moduli space can be identified with the quotient $\Hom(\pi_1,G)/G$, 
with $G$ acting by conjugation. Here $\pi_1$ denotes the fundamental group of the surface,
which at genus $\Ng$ is obtained from the free group on $2\Ng$ generators by a single 
relation, so that the moduli space can be identified with the set of $2\Ng$-tuples 
$(g_1,x_1,g_2,x_2,...\,,g_\Ng,x_\Ng)$ satisfying $\prod_{i=1}^\Ng[g_i,x_i] \eq e$, modulo 
conjugation ($[g,x] \eq g\inv x\inv gx$ is the group commutator). The mapping class group 
$\mcg_{\Ng,0}$ of the surface acts on the moduli space through its inclusion in the 
group $\mathrm{Out}(\pi_1)$ of outer automorphisms of the fundamental group.
When the boundary of the surface consists of a single circle, one instead 
considers the moduli space of $G$-bundles for which 
the monodromy around the boundary circle is restricted to a fixed conjugacy class $c$. 
This moduli space has a similar description, namely as the set of $2\Ng$-tuples 
$(g_1,x_1,...\,,g_\Ng,x_\Ng)$ satisfying $\prod_{i=1}^\Ng[g_i,x_i]\,{\in}\, c$, and again 
comes with an action of the relevant mapping class group, $\mcg_{\Ng,1}$. 
One might therefore hope that the quantum representations studied here will
be permutation representations, and in particular that they will have finite images. 
As it turns out, in general they are quotients of permutation representations;
in the main part of this paper we focus on the subset of those quantum representations 
of $\mcg_{\Ng,0}$ and $\mcg_{\Ng,1}$ that actually are permutation representations.

\medskip

The rest of this paper is organized as follows. In Section \ref{sec:Hopfqreps} 
we review the construction of quantum representations from factorizable ribbon 
Hopf algebras obtained in \cite{lyub11}, restricting to surfaces that either 
are closed or have one marked boundary circle. We then specialize this construction 
to the case that the Hopf algebra is the untwisted Drinfeld double $\DG$ of 
a finite group $G$, and for $\ko \eq \complex$ re-derive the formulas 
for the $S$ and $T$ matrices obtained in \cite{kssb}. We can then verify that 
all these quantum representations, both of $\mcg_{\Ng,0}$ and of $\mcg_{\Ng,1}$, are
quotients of permutation representations. Afterwards we identify a subset of these
representations, including all of them for $\mcg_{\Ng,0}$, that possess a particularly 
simple description as permutation representations. They turn out to be linearizations of
the mapping class group actions on moduli spaces of principal $G$-bun\-dles mentioned above. 
In other words, they are equivalent to permutation representations of $\mathrm{Out}(\pi_1)$ 
on $\mathrm{Rep}(\pi_1,G)$. Theorem \ref{thm:notInvBasis} identifies these permutation 
representations explicitly as sub-representations of quantum representations given 
in the language of \cite{lyub11}.
All of these permutation representations actually satisfy the relations of $\mcg_{\Ng,0}$; 
the underlying group action does not depend on the ground field $\Bbbk$.

In Section \ref{sec:mcg2Nconj} we investigate the group actions underlying 
quantum representations. A number of results that help to identify mapping class group 
orbits are established. The main statement (Theorem \ref{thm:Torelli}) is: For genus 
$\Ng \,{\ge}\, 3$ the Torelli group is represented non-trivially iff 
the finite group $G$ is non-abelian; for genus $\Ng \eq 2$ the Torelli group is 
non-trivially represented iff there exists a pair of elements $g,x\iN G$ whose 
commutator $[g,x]$ does not commute with both $g$ and $x$.
Theorem \ref{thm:congruencesubgroup} extends the congruence subgroup result of 
\cite{coga2,bant14}: A sub-representation of the $\mcg_{1,1}$-representation, containing
the $\mcg_{1,0}$-representation relevant for the results of \cite{coga2,bant14},
is identified such that the kernel of the restriction to this sub-representation is a 
congruence subgroup. However, this congruence subgroup property fails to hold in general 
in the case of $\mcg_{1,1}$.

In Appendix \ref{sec:ex} we provide a number of examples, for some particular finite 
groups $G$ and for genus $1$, $2$ and $3$. 
For genus $1$ we display the mapping class group actions for the dihedral groups $D_n$, 
the generalized quaternion groups $Q_{4n}$, and the alternating group $A_5$.
For genus $2$ and $3$ we present results obtained using GAP \cite{GAP} 
for the groups $S_3$, $D_4$ and $Q_8$, and at genus $2$ also for $A_5$.
The groups $D_{4n}$ and $Q_{8n}$ have identical character tables. 
Thus the categories $\mathscr D D_{4n}\Mod$ and $\mathscr DQ_{8n}\Mod$ have isomorphic 
Gro\-then\-dieck rings even though they are not equivalent as monoidal categories. It is
therefore of some interest to compare the mapping class group representations obtained
for these groups. As shown in \cite{cogR}, the matrix $S$ that together with a matrix $T$ 
generates the $\mcg_{1,0}$\,-re\-pre\-sen\-ta\-ti\-ons
is different for these two series of groups. 

         \clearpage   % needed because otherwise the table is put to the end of the file 

\begin{table}[h]\label{table} \caption{\bf List of notations}
\begin{center} \begin{tabular}{l p{12cm}}
   \hline ~\nxl{-2}
          $Z_g$    & the centralizer of a group element $g\iN G$
   \nxl1  $\expG$  & the exponent of $G$
   \nxl1  $C_G$    & the set of conjugacy classes in $G$
   \nxl1  $C_G^{2\Ng}$
          & the set of $2\Ng$-conjugacy classes in $G^{\times 2\Ng}$ (see \eqref{g-action-2N})
   \nxl1  $\gc$    & chosen representative for the conjugacy class $c \iN C_G$
   \nxl1  $[g,h]$  & the group commutator $g\inv h\inv gh$ of two elements $g,h\iN G$
   \nxl1  $\com_G$ & the set $[G,G]$ of commutators in $G$
   \nxl1  $G'$     & the commutator subgroup of $G$, i.e.\ the group generated by $\com_G$ 
   \nxl1  $G^{\mathrm{ab}}$ & the abelianization of $G$, i.e.\ the quotient $G/G'$
   \nxl1  $\tilde\mu_m$ & the multi-commutator map $G^{\times 2m}\To G$ \eqref{tildemu} 
   \nxl1  $\mu_\Ng$ & the map $C_G^{2\Ng}\To C_G^{}$ \eqref{eq:muN} induced by $\tilde\mu_\Ng$
   \nxl1  $\DG$    & the (untwisted) Drinfeld double of $G$
   \nxl1  $\DGa$   & the Lyubashenko coend Hopf algebra in $\DG$\Mod
   \nxl1  \ko\     & an algebraically closed field
   \nxl1  $\mcg_{\Ng,0}$ & the mapping class group of a closed oriented surface of genus $\Ng$
   \nxl1  $\mcg_{\Ng,1}$ 
          & the mapping class group of a compact oriented surface of genus $\Ng$ with one hole
   \nxl1 \hline
\end{tabular} \end{center} \end{table}

{}From our examples it follows that these representations, 
as well as those $\mcg_{1,1}$-repre\-sen\-tations that we study, are indeed non-isomorphic.
For the groups $D_{4}$ and $Q_8$ this non-equivalence generalizes to genus $2$ and $3$. 
In all cases, however, the dimensions of the respective representations coincide.
   
In the appendix we also present examples (see Remark \ref{rem:NoPGroup}) which show 
that at genus larger than $0$ the image of the quantum representation obtained from $\DG$ is
in general not a $p$-group even when $G$ is a $p$-group. Thus a corresponding result 
of \cite{etrw} for the images of the quantum representations of $\mcg_{0,n}$ 
does not generalize to higher genus.

\medskip

Throughout this paper, $G$ denotes a finite group and $Z(G)$ its center,
and we denote by $g^h \,{:=}\, h\inv g h$ and ${}^hg \,{:=}\, h g h\inv$
the right and left conjugation of $g\iN G$ by $h\iN G$.
Our notational conventions for some further basic notions are listed in Table \ref{table}.
A few technical details are presented in two appendices. Appendix \ref{sec:DG} 
reviews the structure of $\DG$ as a factorizable ribbon Hopf algebra.
We determine the center $Z(\DG)$, the space $C(\DG)$ of 
central forms, and the group $\mathcal{G}(\DG)$ of group-like elements, as well as 
the structure of the Lyubashenko coend Hopf algebra $\DGa$ in the category $\DG\Mod$. 
In Appendix \ref{sec:2Nconj} we collect some elementary results concerning 
conjugacy classes of $2\Ng$-tuples of group elements in a finite group $G$.

As we shall see, the quantum representations obtained from $\DG$ with abelian 
group $G$ can be understood through elementary considerations. 
For abelian $G$, any homomorphism from $\pi$ to $G$ factors through to 
$\pi^{\mathrm{ab}} \eq \pi/\pi'$, which coincides with the first homology. Thus the 
Torelli group, being the kernel of the mapping class group action on 
$\pi^{\mathrm{ab}}$, is represented trivially on $\Hom(\pi,G)/G \eq \Hom(\pi,G)$. 
As a consequence we will mainly be interested in non-abelian finite groups $G$. 
We do not impose any further restrictions on the groups $G$, however. 
In particular, there is no reason to restrict to simple groups.

\medskip

Let us finally specify our assumptions on the ground field \ko\ over which $\DG$ and 
its modules are defined. For the construction of quantum representations in \cite{lyub11} to work 
it is sufficient that the category $\DG\Mod$ is a finite tensor category for which, in turn, \ko\
can be any algebraically closed field. The structure of $\DG$-modules depends, however, on 
\ko\ \cite{with2}. Whenever we deal directly with $\DG$-modules and their characters we will
therefore require the ground field to have characteristic $0$, and for simplicity we then 
restrict our attention to $\ko \eq \complex$. This applies to the part of Section
\ref{ssec:LyubDG} that begins after Remark \ref{rem:orbit}. However, with the exception
of appendix B these assumptions are not needed, and accordingly we otherwise
allow the characteristic of \ko\ to be arbitrary. Note that when the characteristic 
of \ko\ divides the order $|G|$ of the group, then the category $\DG\Mod$ is no longer 
semisimple. But the quantum representations that we focus on in this paper are then still 
the permutation representations obtained from the mapping class actions 
discussed in Section \ref{sec:mcg2Nconj}.

%%%%%%%%%%%%%%%%%%%%%%%%%%%%%%%%%%%%%%%%%%%%%%%%%%%%%%%%%%%%%%%%

\section{Quantum representations from Hopf algebras}\label{sec:Hopfqreps}

In this section we first outline Lyubashenko's method \cite{lyub11} for 
constructing projective representations of mapping class groups from a factorizable 
ribbon Hopf algebra $H$. Afterwards we focus on the case that $H$ is the
(untwisted) Drinfeld double \DG\ of a finite group $G$.

\subsection{The coend \texorpdfstring{\boldmath $H^\star\coa$}{H*} and quantum representations}
            \label{sec:Hmod}\label{ssec:Lyubrev}

For a finite-dimensional Hopf algebra $H$ over $\Bbbk$ we denote the multiplication by
$m$, the unit by $\eta$, the comultiplication by $\Delta$, the counit by $\eps$, and the
antipode by $\apo$. $H$ possesses a non-zero left integral $\Lambda \iN H$ and a 
non-zero right cointegral $\lambda\iN\Hs$, which are unique up to scalars.
$H$ is called ribbon iff it is equipped with an R-matrix 
$R \iN H\oti H$ (i.e.\ is quasitriangular) and with a ribbon element $\nu\iN H$ 
satisfying natural properties (see e.g.\ \cite{KAss}). 
A quasitriangular Hopf algebra $H$ is called factorizable iff the monodromy matrix 
$Q \,{:=}\, R_{21}\,{\cdot}\, R \iN H\oti H$ is non-degenerate, i.e.\ iff the Drinfeld 
map $f_Q\colon \Hs \To H$, acting as $\varphi \,{\mapsto}\, (\varphi\oti \id_H)\cir Q$,
is invertible. For a factorizable Hopf algebra the integral $\Lambda$ is two-sided.

In the sequel $H$ stands for a finite-dimensional factorizable ribbon Hopf algebra.
The category $H\Mod$ of left $H$-modules is finite abelian $\Bbbk$-linear; the
structural elements of $H$ endow $H\Mod$ with a natural structure of a ribbon category.
In particular, denoting by $\tau$ the flip map, the braiding $c_{X,Y}$ of two 
$H$-modules $(X,\rho_X)$ and $(Y,\rho_Y)$ is the linear map 
       \\
$\tau_{X,Y}\cir (\rho_X\oti \rho_Y)\cir (\id_{H}\oti \tau_{H,X}\oti \id_Y)\cir 
(R\oti \id_{X\otimes Y})$, while the twist automorphism $\Theta_X$ of $(X,\rho_X)$ 
is the linear map $\rho_X\cir (\nu\inv\oti \id_{X})$.

\medskip

The starting point for constructing the quantum representations is the coend
  \be
  L := \int^X \!\! X^\vee\oti X \,.
  \label{Lcoend}
  \ee
As an object in $H\Mod$, $L$ is isomorphic to the $H$-module $H^\star\coa$, that is,
to the vector space \Hs\
endowed with the left coadjoint $H$-action \cite{lyub6,vire4}. Explicitly, $x\iN H$ 
acts on \Hs\ by mapping $\varphi\iN \Hs$ to $\sum_i\varphi(\apo(x^{(1)}) y_i x^{(2)})\psi_i$, 
where $1 \,{\mapsto}\, \sum_i y_i\oti \psi_i \iN H\oti\Hs$ is the coevaluation. The 
coend $L$ carries a natural structure of a Hopf algebra in $H\Mod$, with structure 
morphisms given by
  \be
  \bearl
  m\coa^\star(x) = \apo(r^{(1)})\, x^{(1)}\, r^{(2)} \otimes \apo(s)x\, ^{(2)} \,,
  \nxl2
  \eta\coa = \eps^\star \,,
  \nxl2
  \Delta\coa = m^\star \,,
  \nxl2
  \eps\coa = \eta^\star \,,
  \nxl2
  \apo\coa^\star(x) = \apo(r)\, \apo(x)\, \apo(u)\inv\, s \,.
  \eear
  \label{eq:Lapo}
  \ee
Here for brevity instead of the multiplication and antipode we present their pre-duals,
and we use the summation-free Sweedler notation $\Delta(x) \eq x^{(1)} \oti x^{(2)}$
for the coproduct of $x\iN H$ and
$R \eq r\oti s$ for the R-matrix; $u \eq \apo(r)\,s$ is the Drinfeld element. 
The Hopf algebra $H^\star\coa$ comes with a Hopf pairing
  \be
  \omega\coa:\quad \alpha\oti \beta \,\mapsto\,
  \alpha\big(Q^{(1)}\big)\, \beta\big(\apo(Q^{(2)})\big) \,.
  \label{Lhopa}
  \ee
Factorizability of $H$ is equivalent to non-degeneracy of $\omega\coa$. 

\medskip

We denote by $\mcg_{\Ng,n}$ the mapping class group of an oriented compact surface 
of genus $\Ng$ obtained from a closed surface by excising $n \,{\ge}\, 0$ open disks. 
More specifically, the appropriate variant of $\mcg_{\Ng,n}$ is obtained by equipping 
each of the $n$ boundary circles with a marked point and requiring homeomorphisms, 
as well as homotopies, to preserve the set of marked points.
As shown in \cite{lyub11}, the structure described above allows one to construct a 
finite-dimensional projective representation of $\mcg_{\Ng,n}$ on the $\Bbbk$-vector space
  \be
  \Hom_{H\Mods}(X^{\otimes n},L^{\otimes \Ng}) \,,
  \ee
for any object $X$ of $H\Mod$. 
We will restrict our attention to surfaces with zero or one holes, $n\iN \{0,1\}$.
$\mcg_{\Ng,1}$ can then be defined as the group of homotopy classes of 
orientation preserving homeomorphisms that preserve the boundary \emph{pointwise}.

\medskip

To define the quantum representations of $\mcg_{\Ng,0}$ and $\mcg_{\Ng,1}$, first 
choose for every genus $\Ng\iN \mathbb{N}$ a representative surface $\Sigma_\Ng$ with one hole, 
together with a suitable collection of simple closed curves on $\Sigma_\Ng$. 
A convenient (non-minimal) choice for the latter consists (see \cite[Sect.\,4.5]{lyub6} 
or \cite[Sect.\,2.3]{fuSs5}) of four simple closed curves $a_i$, $b_i$, $d_i$, 
and $e_i$, for each $i \eq 1,2,...\,, \Ng$, modulo the 
identification $a_1 \eq d_1 \eq e_1$. Furthermore for each $i \eq 1,2,...\,, \Ng$ 
we have to choose a one-holed torus $F_i$ embedded in $\Sigma_\Ng$. 
This collection of data is visualised in the following figure:
  \be
\begin{picture}(400,128) \put(-15,0){
	\put(0,0)       {\includegraphics[scale=0.6]{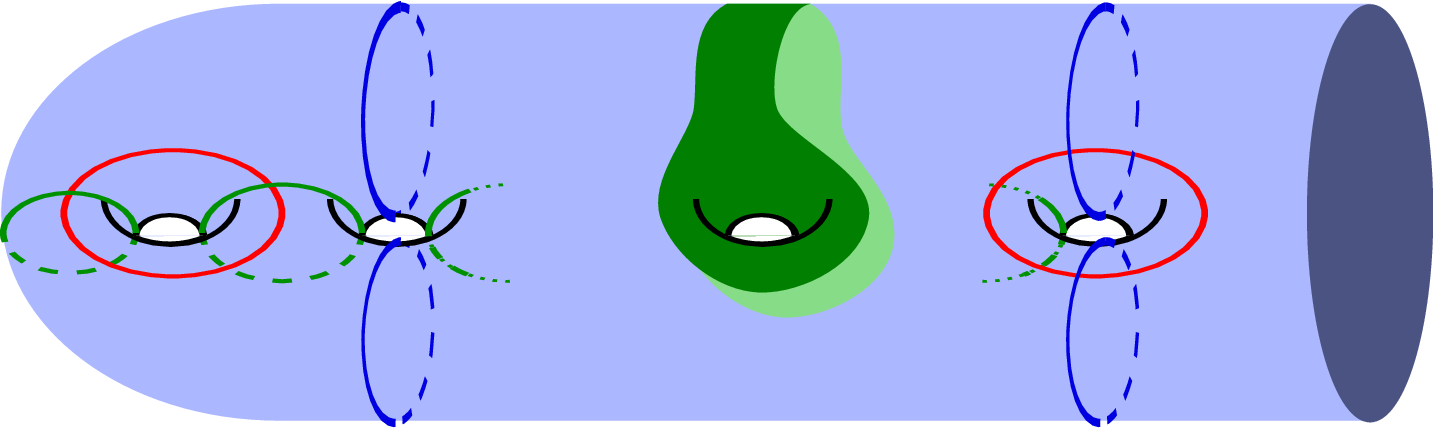}}
	\put(7,68)      {\scriptsize $a_1$}
	\put(90,70)	{\scriptsize $a_2$}
	\put(135,73)	{\scriptsize $a_3$}
	\put(277,73)	{\scriptsize $a_\Ng$}
	\put(50,83)	{\scriptsize $b_1$}
	\put(295,80)	{\scriptsize $b_\Ng$}
	\put(107,90)	{\scriptsize $d_2$}
	\put(107,25)	{\scriptsize $e_2$}
	\put(210,80)	{\scriptsize $F_i$}
	\put(310,93)	{\scriptsize $d_\Ng$}
	\put(310,25)	{\scriptsize $e_\Ng$}
	\put(155,55)	{$\cdots$}
	\put(265,55)	{$\cdots$}
       }
\end{picture}
  \label{thesurface}
  \ee

$\mcg_{\Ng,1}$ has a generating set consisting of (the homotopy classes of) the 
following maps (\cite{lyub6}, see also \cite{fuSs5}):
\def\leftmargini{1.3em}\begin{itemize}
  \item[1.] inverse Dehn twists $T\inv_{b_i}$ and $T\inv_{d_i}$
   around each of the curves $b_i$ and $d_i$, for $i \eq 1,2,...\,, \Ng$;
  \item[2.] inverse Dehn twists $T\inv_{a_i}$ and $T\inv_{e_i}$
   around each of the curves $a_i$ and $e_i$, for $i \eq 2,3,...\,, \Ng$.
\end{itemize}
While these Dehn twists already generate $\mcg_{\Ng,1}$, occasionally it is convenient to add:
\begin{itemize}
  \item[3.] a modular $S$-transformation $S_i$
   with support in the interior of $F_i$, for $i \eq 1,2,...\,, \Ng$.
\end{itemize}
By capping the hole of the surface \eqref{thesurface} one obtains a closed surface of 
genus $\Ng$. Thus the same set of Dehn twists generates $\mcg_{\Ng,0}$, and there is only a 
single additional relation, which implements the condition that a Dehn twist $T_\partial$
along a simple closed curve isotopic to the boundary has to be trivial, so that the 
relation that defines $\mcg_{\Ng,0}$ as a quotient of $\mcg_{\Ng,1}$ is $T_\partial \eq \id$. 
Using the so-called star relation, as e.g.\ given by formula (E) of \cite{gervS2}, one 
checks that $T_\partial$ can be expressed as
  \be
  T_\partial = (T_{d_1}^3 T_{b_1})^3\, (T_{e_2}^2 T_{d_2}^{} T_{b_2}^{})^3_{}\,
  (T_{e_3}^2 T_{d_3}^{} T_{b_3}^{})^3_{} \cdots\, (T_{e_\Ng}^2 T_{d_\Ng}^{} T_{b_\Ng}^{})^3_{} .
  \ee
We note that for genus $\Ng \eq 1$, the twist $T$ around $d_1$ together with the modular 
$S$-transformation generate both $\mcg_{1,0}$ and $\mcg_{1,1}$. In the former case, 
$T$ and $S$ satisfy $(ST)^3 \eq S^2$, $S^4 \eq \id$ and $[S^2,T] \eq \id$, which are 
defining relations for the modular group \slz. In the latter case there is
a short exact sequence
  \be
  1 \longrightarrow \zet\longrightarrow\mcg_{1,1}\longrightarrow \slz\longrightarrow 1 \,,
  \ee
i.e.\ $\mcg_{1,1}$ is a central extension of \slz.

To proceed we define a collection of endomorphisms in the category $H\Mod$. For $h\iN H$ 
we denote by $l_h, r_h\iN \End_\ko(\Hs)$ the duals of the linear endomorphisms of $H$ 
given by left and right multiplication with $h$, i.e.\
  \be
  \big( l_h(\alpha)\big) (k) = \alpha(h\, k) \qquad{\rm and}\qquad
  \big( r_h(\alpha)\big) (k) = \alpha(k\, h)
  \ee
for $\alpha\iN\Hs$ and $k\iN H$. Then for $(X,\rho)\iN H\Mod$ we set
  \be
  \begin{array}{rrl}
  \End_{H\Mods}(H^\star\coa) \,\ni \!\!& T_L: & \alpha \,\mapsto\, l_{\apo(\nu\inv)}(\alpha) \,,
  \Nxl4
  \End_{H\Mods}(H^\star\coa\oti H^\star\coa) \,\ni \!\!& O_L:
  & \alpha\oti \beta \,\mapsto\, r_{Q^{(1)}}(\alpha)\oti l_{\apo(Q^{(2)})}(\beta) \,,
  \Nxl4
  \End_{H\Mods}(H^\star\coa) \,\ni \!\!& S_L:
  & \!\!\! = (\eps\coa\oti\id)\circ O_L\circ (\id\oti \Lambda\coa) \,,
  \Nxl4
  \End_{H\Mods}(H^\star\coa\oti X) \,\ni \!\!& Q_L^X:
  & \alpha\oti m \,\mapsto\, r_{Q^{(1)}}(\alpha)\oti \rho(\apo(Q^{(2)}))\,m \,.
  \label{eq:TL-QL}
  \eear
  \ee
We use these maps to specify a set of endomorphisms of $L^{\otimes \Ng}$:

\begin{defn}\label{def:coendendo}~\\
{\rm (i)}\,
For any $N \iN \NN$ we introduce the following elements of $\End_{H\Mods}(L^{\otimes N})$:
  \begin{align}
	  z_S^{(i)} &:= \id_{L^{\otimes (N-i)}}\oti S_L^{} \oti \id_{L^{\otimes(i-1)}} \,,
        & i=1,2,...\,,N\label{eq:MCGendo1} \,, \nxl2
  z_a^{(i)} &:= \id_{L^{\otimes (N-i)}} \oti \left[ (T_L^{} \oti T_L^{} )\circ O_L^{}  \right]
            \oti \id_{L^{\otimes(i-2)}} 
        & i=2,3,...\,,N\label{eq:MCGendo2} \,, \nxl2
  z_b^{(i)} &:= \id_{L^{\otimes (N-i)}}\oti \left[S_L^{} \circ T_L^{}\circ S_L\inv\right]
         \oti \id_{L^{\otimes(i-1)}} \,,
        & i=1,2,...\,,N\label{eq:MCGendo3} \,, \nxl2
	 z_d^{(i)} &:= \id_{L^{\otimes (N-i)}}\oti T_L^{} \oti \id_{L^{\otimes(i-1)}}
        & i=1,2,...\,,N\label{eq:MCGendo4} \,, \nxl2
  z_e^{(i)} &:= \id_{L^{\otimes (N-i)}}\oti
        \big[ Q_L^{L^{\otimes (i-1)}} \circ (T_L^{} \oti \Theta_{L^{\otimes(i-1)}}) \big] \,,
        & i=2,3,...\,,N\label{eq:MCGendo5} \,.
  \end{align}
  \label{eq:MCGendo1-5}
~\\[-13pt]
{\rm{(ii)}}\,
The quantum representation of $\mcg_{\Ng,1}$ on $\Hom_{H\Mods}(X, L^{\otimes \Ng})$ is defined 
by letting the generators \eqref{eq:MCGendo1}\,--\,\eqref{eq:MCGendo5} act via post-composition. 
Concretely, the inverse Dehn twists 
around the cycles $b_i$, $d_i$ and $e_i$ are represented by post-composition with the
endomorphisms $z^{(i)}_{b}$, $z_{d}^{(i)}$ and $z_e^{(i)}$, respectively, for 
$i \eq 1,2,...\,,\Ng$; the inverse Dehn twist around $a_i$ is represented by post-composition 
with $z_a^{(i)}$, for $i \eq 2,3,...\,,\Ng$; finally the modular $S$-transformation inside 
the one-holed torus $F_i$ is represented by post-composition with $z_S^{(i)}$, for $i \eq 1,2,...\,, \Ng$.
 \\[2pt]
{\rm{(iii)}}\,
The quantum representation of $\mcg_{\Ng,0}$ is obtained by taking $X$ in (ii)
to be the trivial one-dimensional $H$-module $V_0 \eq (\ko,\eps)$.
\end{defn}

%%%%%%%%%%%%%%%%%%%%%%%%%%%%%%%%%%%%%%%%%%%%%%%%%%%%%%%%%%%%%%%%

\subsection{The case \texorpdfstring{\boldmath $H \eq \DG$}{H=DG}}
           \label{ssec:LyubDG}

We now specialize the quantum representations to the case that $H$ is the 
Drinfeld double \DG\ of a finite group $G$. To keep the presentation short, 
pertinent information about $\DG$ and the coend $\DGa$ is relegated to Appendix 
\ref{sec:DG}. The factorizable ribbon Hopf algebra $\DG$ has a basis 
$\{\b g x \,|\, g,x\iN G\}$ in which the structural data
take the form given in \eqref{DGmorphsinbas}. We denote by $\{\d h y\}$ the 
basis of $\DGs$ dual to $\{\b g x\}$; in this basis the \DG-action on $\DGa$ 
is given by \eqref{eq:DGcoa}. The structure of $\DGa$ as a Hopf algebra in 
$\DG\Mod$ is provided in \eqref{eq:DGamult..}, while its two-sided integral, 
cointegral, and Hopf pairing are given in \eqref{eq:DGaintegrals} and 
\eqref{eq:DGaHopfpairing}.

\medskip

When expressing statements in terms of the chosen bases of \DG\ and \DGs\
it is often convenient to use the notation $(g|x)$ for elements in $G\Times G$, 
and analogously
  \be
  (\bh|\by) := \big((h_1|y_1),\ldots,(h_m|y_m)\big)
  \label{def:bhby}
  \ee
for elements of $G^{\times 2m}$. The notation $(\bh^g|\by^g)$ then indicates that $g \iN G$ 
acts on $G^{\times 2m}$ by right conjugation of all pairs $(h_i|y_i)$ in \eqref{def:bhby}, i.e.\ 
  \be
  (\bh^g|\by^g) = \big((h_1^g|y_1^g),\ldots,(h_m^g|y_m^g)\big) \,,
  \ee
and analogously for left conjugation. We define for each $m\iN\NN$ a \emph{multi-commutator} 
map $\tilde\mu_m \colon G^{\times 2m}\To G$ as an (ordered) product of commutators:
  \be
  \tilde\mu_m:\quad  (\bh|\by) \,\mapsto \prod_{i=1}^m\, [h_i,y_i] \,.
  \label{tildemu}
  \ee
In terms of basis elements, the endomorphisms \eqref{eq:TL-QL}, with the object $X$ in $Q_L^X$
specialized to $X \eq (\DGa)^{\otimes m}$, together with the composite 
endomorphism $U\coa^{} \,{:=}\, S\coa^{}\cir T\coa^{}\cir S\coa\inv$, are then
expressed as follows \cite{ffss}:
  \be
  \bearll
  T\coa:  & \d g x \,\mapsto\, \d g{g\inv x} \,,    
  \nxl2
  S\coa:  & \d g x \,\mapsto\, \d {x\inv} {g^x} \,, 
  \nxl2
  U\coa:  & \d g x \,\mapsto\, \d {xg} {x} \,,      
  \nxl2
  O\coa:  & \d g x\oti \d h y  \,\mapsto\, \d g {({}^{gx}h)x} \oti \d {{}^{(g^x)}h}{(g^x)y} \,,
  \nxl3
  Q\coa^m:& \d g x\oti \d \bh \by \,\mapsto\, \d g {gx(g^x\cdot \tilde\mu_m(\bh|\by))\inv}
     \oti \d {{}^{(g^x)}\bh}{{}^{(g^x)}\by} \,.
  \eear
  \label{eq:DGT..}
  \ee
The twist automorphism of $(\DGa)^{\otimes m}$ is given by
  \be
  \Theta_{(\DGa)^{\otimes m}}:\quad
  \d \bh \by \mapsto \d {{}^{\tilde\mu_m(\bh |\by)} \bh} {{}^{\tilde\mu_m(\bh |\by)}\by} \,,
  \qquad
  \label{eq:DGTheta}
  \ee
and the composites $L\coa \,{:=}\, (T\coa\oti T\coa)\circ O\coa$ and 
$M_{\triangleright,m} \,{:=}\, Q\coa^m\circ (T\coa\oti \Theta_{(\DGa)^{\otimes m}})$ take the form
  \be
  \bearll
  L\coa: & \d g x\oti \d h y \,\longmapsto\, \d {g}{({}^xh)g\inv x} \oti \d {{}^{(g^x)}h}{(g^x)h\inv y} \,,
  \nxl3
  M_{\triangleright,m}: & \d g x\oti \d \bh \by \,\longmapsto\, \d g {x(g^x\cdot \tilde\mu_m(\bh|\by))\inv}
  \oti \d {{}^{(g^x)\tilde\mu_m(\bh|\by)} \bh} {{}^{(g^x)\tilde\mu_m(\bh|\by)} \by} \,.
  \label{eq:DGLM}
  \eear
  \ee
 
In the sequel we regard $G^{\times 2\Ng}$ as a $G$-space, with $G$ acting by conjugation,
  \be
  g: \quad  (\bh|\by) \longmapsto (\bh^g|\by^g) \,.
  \label{g-action-2N}
  \ee
We refer to the orbits of the action \eqref{g-action-2N} of $G$ on $G^{\times 2\Ng}$ as 
\emph{$2\Ng$-conjugacy classes}, and denote the set of all $2\Ng$-conjugacy classes by 
$C^{2\Ng}_G$. For elements of $C^2_G$ we alternatively also use the term \emph{diconjugacy classes}. 

We now note that according to \eqref{eq:DGT..} for any pair $(g|x)$ we have
  \be
  \begin{array}{rl}
  S\coa^2: & \d g x \,\mapsto\, \d{(g\inv)^{gx}}{(x\inv)^{gx}} \,,
  \Nxl3
  S\coa^4: & \d g x \,\mapsto\, \d{g^{[g,x]}}{x^{[g,x]}} \,,
  \Nxl3
  (S\coa\,{\circ}\, T\coa)^3: & \d g x \,\mapsto\, \d {(g\inv)^{gx}}{(x\inv)^{gx}}.
  \eear
  \ee
Thus the operators $S\coa$ and $T\coa$ define an \slz-action on diconjugacy 
classes of basis elements and, when restricting to commuting pairs $(g|x)$, 
an \slz-action already on such basis elements.

We simplify notation by renaming the endomorphisms from Definition \ref{def:coendendo} 
as $z_S^{(i)}\,{\rightsquigarrow}\, S\coa^{(i)}$, $z^{(i)}_d\,{\rightsquigarrow}\, T\coa^{(i)}$, 
$z_b^{(i)}\,{\rightsquigarrow}\, U\coa^{(i)}$, $z_a^{(i)}\,{\rightsquigarrow}\, L\coa^{(i)}$
and $z_e^{(i)}\,{\rightsquigarrow}\, M\coa^{(i)}$. We observe:

\begin{lemma}\label{lem:endoproperties}~\\
The endomorphisms $T\coa^{(i)}$, $S\coa^{(i)}$, $U\coa^{(i)}$, $L\coa^{(i)}$ and 
$M\coa^{(i)}$ have the following properties:
\\[2pt]
{\rm (i)} They permute the basis elements $\d \bh \by$ of $G^{\times 2m}$.
\\[2pt]
{\rm (ii)} They are equivariant with respect to the action of $G$ by conjugation 
of the labels $(\bh,\by)$ of the basis elements $\d \bh \by$.
\\[2pt]
{\rm (iii)} They preserve the multi-commutators
$\tilde\mu_m(\bh|\by)$ of the basis elements $\d\bh\by$.
\end{lemma}

\begin{proof}
Property (i) follows immediately from the definition of the endomorphisms. 
\\[2pt]
Property (ii) follows from simple manipulations based on the identity $[g^z,x^z] \eq [g,x]^z_{}$.
\\[2pt]
For proving (iii), note that it suffices to restrict to those pairs $(h_i|y_i)$
on which the endomorphism acts non-trivially. Accordingly, the identities
  \be
  [g,g\inv x] = [g,x] \,, \qquad
  [x\inv,g^x] = [g,x]  \qquand
  [xg,x]      = [g,x]
  \ee
establish (iii) for $T\coa^{(i)}$, $S\coa^{(i)}$ and $U\coa^{(i)}$, respectively. 
Next we note that
  \be
  [g,({}^xh)g\inv x] = [g,x]\, h\inv\, {}^{(g^x)}h \qquad{\rm and}\qquad
  [{}^{(g^x)}h,(g^x)h\inv y] =	{}^{(g^x)}(h\inv)\, h^y \,,
  \ee
which imply 
  \be
  [g,({}^xh)g\inv x] \, [{}^{(g^x)}h,(g^x)h\inv y]
   = [g,x] h\inv\, h^y = [g,x] \, [h,y] \,,
  \ee
thus establishing (iii) for $L\coa^{(i)}$.
Finally by direct, if lengthy, calculation one sees that
  \be
  \bearl
  [g,x(g^x\, \tilde\mu_m(\bh|\by))\inv]
  = [g,x]\, \tilde\mu_m(\bh|\by) \, {}^{(g^{x})}(\tilde\mu_m(\bh|\by)\inv)
  \qquad{\rm and}
  \Nxl4
  \tilde\mu_m\big( {}^{(g^x)\tilde\mu_m(\bh|\by)} \bh|{}^{(g^x)\tilde\mu_m(\bh|\by)} \by\big)
  = {}^{(g^x)}\tilde\mu_m(\bh|\by) \,,
  \eear
  \ee
showing that (iii) holds for $M\coa^{(i)}$ as well.
\end{proof}

Using the explicit form of the $\DG$-action on $(\DGa)^{\otimes \Ng}$
it is easily verified that the $G$-equivariance is equivalent to the linear maps 
\eqref{eq:DGT..}\,--\,\eqref{eq:DGLM} being morphisms in $\DG\Mod$.

\medskip

As we will explain in Remark \ref{rem:orbit}, to fully understand the quantum representations 
of $\mcg_{\Ng,0}$ and $\mcg_{\Ng,1}$ it is enough to understand the mapping class group orbit 
of the identity morphism $\id_{(\DGa)^{\otimes \Ng}}\iN \End_{\DG\Mods}(\DGa^{\otimes \Ng})$. 
Lemma \ref{lem:endoproperties}\,(i) implies that this orbit is a finite set 
of endomorphisms that permute basis elements of the type $\d \bh \by$. In other words, we 
have established the existence of a finite $\mcg_{\Ng,1}$-set whose associated permutation 
representation determines all the quantum representations of $\mcg_{\Ng,1}$ and $\mcg_{\Ng,0}$.
We denote this permutation representation of $\mcg_{\Ng,1}$ by $\DGN$. As a vector space, 
$\DGN$ is spanned by the endomorphisms of $(\DGa)^{\otimes \Ng}$ that are arbitrary words
in the letters $T\coa^{(i)}$, $S\coa^{(i)}$, $U\coa^{(i)}$, $L\coa^{(i)}$ and $M\coa^{(i)}$.
The relation between the permutation representation $\DGN$ and a general 
quantum representation is summarized in

\begin{prop}\label{prop:genqreps}
For any \DG-module $V$ the quantum representation on the morphism space
{\rm $\Hom_{\DG\Mods}\big(V,(\DGa)^{\otimes \Ng}\big)$} coincides with the image of
  \be
  \Hom_{\DG\Mods}\big(V,(\DGa)^{\otimes \Ng}\big)\otimes \DGN 
  \ee
under the composition map 
  \be
  \bearl
  \circ:\quad \Hom_{\DG\Mods}\big(V,(\DGa)^{\otimes \Ng}\big)\otimes
  \End_{\DG\Mods}\big((\DGa)^{\otimes \Ng} \big)
  \nxl2 \hspace*{15em}
  \longrightarrow\, \Hom_{\DG\Mods}\big(V,(\DGa)^{\otimes \Ng}\big) \,,
  \eear
  \ee
 where on the left hand side the mapping class group $\mcg_{\Ng,1}$ acts non-trivially 
only on the second factor.
\end{prop}

As a consequence, we can relate quantum representations to the permutation representation
$\DGN$ as follows:

\begin{cor}
Every quantum representation is isomorphic to a quotient of a direct sum
$(\DGN)^{\oplus d}$ of multiple copies of permutation representations. It follows in particular
that the image of any quantum representation of $\mcg_{\Ng,0}$ or $\mcg_{\Ng,1}$ is finite.
\end{cor}

We will not pursue the full structure of the permutation representations $\DGN$. 
Instead we will identify subspaces of $(\DGa)^{\otimes \Ng}$ that carry representations 
of $\mcg_{\Ng,0}$ and $\mcg_{\Ng,1}$. These representations contain much of the 
information encoded in $\DGN$, but are easier to deal with. 
The multi-commutator map $\tilde\mu_\Ng$ introduced in \eqref{tildemu} intertwines 
the conjugation action, hence it maps every $2\Ng$-conjugacy class surjectively 
onto a single conjugacy class. There is thus an induced map
  \be
  \begin{array}{rcl}
  \mu_\Ng : \quad C^{2\Ng}_G &\!\!\longrightarrow\!\!& C_G \,,
  \\[2pt]
  d &\!\!\longmapsto\!\!& \tilde\mu_\Ng(d)  
  \eear
  \label{eq:muN}
  \ee
from $2\Ng$-conjugacy classes to ordinary conjugacy classes.
For each $2\Ng$-conjugacy class $d\iN C^{2\Ng}_G$ we define an element 
$\gamma_d\iN (\DGa)^{\otimes \Ng}$ as the sum of dual basis elements labeled by
elements in $d$,
  \be
  \gamma_d := \sum_{(\bh|\by)\in d}\! \d \bh \by \,.
  \label{def:gamma_d}
  \ee

The subset $\com_G \eq [G,G]$ of commutators in $G$ is closed under conjugation, 
and thus is a union of conjugacy classes. We denote the collection of these
particular conjugacy classes by $C'_G \,{\subseteq}\, C_G$, so that
  \be
  \com_G = \bigcup_{c\in C'_G}\! c \,.
  \label{com=cupc}
  \ee
By definition we have $\com_G \eq \im(\tilde\mu_1)$ and hence $C'_G \eq \im(\mu_1)$. 

Denote by $V_0 \,{:=}\, (\Bbbk,\varepsilon)$ the one-dimensional trivial $\DG$-module, 
i.e.\ the tensor unit of $\DG\Mod$. We can identify elements in $\Hom_{\DG\Mods}\big(
V_0,(\DGa)^{\otimes \Ng}\big)$ with their images in $(\DGa)^{\otimes \Ng}$; then we have

\begin{thm}\label{thm:InvBasis}
{\rm (i)} The set
  \be
  \Gamm := \{ \gamma_d \,|\, d\iN \mu_\Ng\inv(c_e) \}
  \ee
is a basis of $\,\Hom_{\DG\Mods}\big(V_0,(\DGa)^{\otimes \Ng}\big)$.
\\[4pt]
{\rm (ii)} 
Each of the generators $T\coa^{(i)}$, $S\coa^{(i)}$, $U\coa^{(i)}$, $L\coa^{(i)}$
and $M\coa^{(i)}$ of $\mcg_\Ng$ acts on $\Gamm$ by a permutation.
\end{thm}

\begin{proof} 
(i)\,
We want to characterize the elements $\beta\iN (\DGa)^{\otimes \Ng}$ that satisfy 
$\b g x \cdo \beta \eq \delta_{g,e}\, \beta$. With the ansatz
  \be
  \beta = \sum_{(\bh|\by)\in G^{\times 2\Ng}} \! \alpha(\bh|\by)\, \d\bh\by
  \ee
we have
  \be
  \begin{array}{ll}
  \b g x\cdo \beta \!\! & \dsty
  = \sum_{(\bh|\by)\in G^{\times 2\Ng}} \sum_{k_1,\ldots, k_{\Ng-1}\in G}\!\! \alpha(\bh|\by)\,
  \b {k_1}x\cdo \d{h_1}{y_1}\oti \b{k_1\inv k_2}x\cdo \d{h_2}{y_2} \oti\cdots
  \\[-6pt] &
  \hspace*{20em} \cdots \oti \b{k_{\Ng-1}\inv g}{x}\cdo \d{h_\Ng}{y_\Ng}
  \\[2pt] &\dsty
  = \sum_{(\bh|\by)\in G^{\times 2\Ng}}\!\! \alpha(\bh|\by)\,
  \del{g^x,\tilde\mu_\Ng(\bh|\by)}\, \d{{}^x\bh}{{}^x\by} \,.
  \end{array}
  \ee
Thus we need $\alpha(\bh|\by) \eq 0$ unless $\tilde\mu_\Ng(\bh|\by) \eq e$, 
and $\alpha(\bh^x|\by^x) \eq \alpha(\bh|\by)$ for all $x\iN G$ and all
$(\bh|\by)\iN G^{\times 2\Ng}$. Conversely, for every function $\alpha$ obeying these 
conditions we get an element $\beta$ satisfying $\b g x \cdo \beta \eq \delta_{g,e}\, \beta$. 
Thus every such $\beta$ is a linear combination of elements $\gamma_d$ for 
$d\iN \mu_\Ng\inv(c_e)$. Since these are linearly independent, $\Gamm$ indeed forms a basis. 
\\[2pt]
(ii)\,
By Lemma \ref{lem:endoproperties}, every generator of $\mcg_{\Ng,1}$ permutes basis 
elements $\d \bh \by$ in a way such that it preserves $\tilde\mu_\Ng$, and thus preserves the 
subset $\{\d \bh\by \,|\, (\bh|\by)\iN \tilde\mu_\Ng\inv(e)\}$. Further, the $\mcg_{\Ng,1}$-generators
are $G$-equivariant, so that for any $d\iN \mu_\Ng\inv(c_e)$
they map bijectively the basis elements $\d \bh \by$ satisfying 
$(\bh|\by)\iN d$ to $\d \bh \by$ satisfying $(\bh|\by)\iN d'$ for some $d'\iN \mu_\Ng\inv(c_e)$.
It follows that each generator indeed permutes the elements of the set $\Gamm$.
\end{proof}

The proof also shows that for \emph{any} conjugacy class $c\iN C_G$ the set 
$\{\gamma_d \,|\, d\iN \mu_\Ng\inv(c)\}$ 
is permuted by the generators of $\mcg_{\Ng,1}$. It is therefore natural to ask whether the 
subspace of $(\DGa)^{\otimes \Ng}$ spanned by such elements $\gamma_d$ has some relation to 
quantum representations as well. To answer this question,
consider for a given  $2\Ng$-conjugacy class $d\iN C^{2\Ng}_G$ the subspace 
  \be
  V^{(d)} := \mathrm{span}_\Bbbk\{\d \bh \by \,|\, (\bh|\by)\iN d\} \, \subset (\DGa)^{\otimes \Ng}
  \label{eq:DCsubmodule}
  \ee
of $(\DGa)^{\otimes \Ng}$. The $\DG$-action
$\b g x\cdo \d \bh \by \eq \del{g^x,\tilde\mu_\Ng(\bh|\by)}\,\d{{}^x\bh}{{}^x\by} $
shows that the subspace $V^{(d)}$ is in fact a $\DG$-submodule of $(\DGa)^{\otimes \Ng}$. Moreover, 
if $\mu_\Ng(d) \eq c$, then $\b g x$ annihilates $V^{(d)}$ unless $g\iN c$. It follows that
for any simple \DG-module $V_{(c,\alpha)}$ (as described in Appendix \ref{ssec:DGmod}) we have
  \be
  \Hom_{\DG\Mods}(V_{(c,\alpha)},V^{(d)}) = 0 \qquad{\rm unless}\quad \mu_1(d)=c \,.
  \ee
Now for $c\iN C_G$, define
  \be
  (\DGa)^{\otimes \Ng}_{(c)} := \bigoplus_{\scriptstyle d\in C^{2\Ng}_G\atop\scriptstyle \mu_\Ng(d)=c} \!\! V^{(d)} .
  \ee
Then by the discussion above we have

\begin{prop}
The permutation representation of $\mcg_{\Ng,1}$ on $\DGN$ decomposes according to
  \be
  \End_{\DG\Mods}\big((\DGa)^{\otimes \Ng}\big)
  = \bigoplus_{c\in C_G}\End_{\DG\Mods}\big((\DGa)^{\otimes \Ng}_{(c)}\big)
  \ee
into sub-representations.
\end{prop}

Next denote, for $g \iN G$, by $\DG^{(g)}$ the subalgebra of $\DG$ that is spanned by 
the basis elements $\{\b g x \,|\, x\iN Z_g\}$ (see also the first paragraph of 
Appendix \ref{ssec:DGmod}); we have $\DG^{(g)} \,{\cong}\, \ko[Z_g]$. Let $d\iN C^{2\Ng}_G$ 
satisfy $\mu_\Ng(d) \eq c$, and define the subspace 
  \be
  V^{(d,\gc)} := \mathrm{span}_\ko \{ \d{\bh}{\by} \,|\, (\bh|\by)\iN d\ \,\mathrm{and}\;
  \tilde\mu_\Ng(\bh|\by) \eq \gc\} \,.
  \ee
Note that if $(\bh|\by)\iN d$ satisfies $\tilde\mu_\Ng(\bh|\by) \eq \gc$, then every 
element of $d$ with that property is of the form $({}^x\bh|{}^x\by)$ for some 
$x\iN Z_{\gc}$. In fact, $V^{(d,\gc)}$ is a permutation representation
of $Z_{\gc}$, with the underlying $Z_{\gc}$-set consisting of the single orbit 
$\{({}^x\bh|{}^x\by) \,|\, x\iN Z_{\gc}\}$. It follows that $V^{(d)}$ is the 
$\DG$-module \eqref{eq:IndMod} induced from $V^{(d,\gc)}$, i.e.
  \be
  V^{(d)} = \DG \otimes_{\DG^{(\gc)}} V^{(d,\gc)} .
  \ee
Any permutation representation contains a trivial sub-representation; in the case of 
$V^{(d,\gc)}$, the vector 
  \be
  \gamma_d^{\gc} := \sum_{\scriptstyle (\bh|\by)\in d\atop\scriptstyle \tilde\mu_\Ng(\bh|\by)=\gc}
  \d{\bh}{\by}
  \ee 
is $Z_{\gc}$-invariant. Denote the $\DG$-module that is obtained from induction on this 
trivial module by $V_{(c)}^{(d)}$. This  $\DG$-module is obviously simple and is equivalent 
to $V_{(c,\idscs)}$. Furthermore, $V_{(c)}^{(d)}$ is cyclic, and one particular cyclic 
vector is $\gamma_d \eq \sum_{\gc\in d}\gamma_d^{\gc}$ as defined in \eqref{def:gamma_d}. 
Since every module in the set $\{V_{(c)}^{(d)} \,|\, d\iN C^{2\Ng}_G\}$ is induced from 
a one-dimensional $\DG^{(\gc)}$-mo\-du\-le, it follows that the mapping class group acts 
on this set, and that this action is completely determined by the mapping class action 
on the vectors $\{\gamma_d^{\gc} \,|\, \mu_\Ng(d) \eq c\}$. Equivalently we can study 
the action of $\mcg_{\Ng,1}$ on the set $\{\gamma_d \,|\, \mu_\Ng(d) \eq c\}$. 
We have thus arrived at

\begin{thm}\label{thm:notInvBasis}
For any $c\iN C_G$, $\mcg_{\Ng,1}$ acts on the set $\{\gamma_d \,|\, d\iN \mu_\Ng\inv(c)\}$. 
  \\
The resulting permutation representation is equivalent to a sub-representation of
the quantum representation on $\Hom_{\DG\Mods}\big(V_{(c,\idscs)},(\DGa)^{\otimes \Ng}\big)$.
\end{thm}

Moreover, we have

\begin{prop}
The action of $\,\mcg_{\Ng,1}$ on $\{\gamma_d \,|\, d\iN \mu_\Ng\inv(c)\}$
factors through to an action of $\,\mcg_{\Ng,0}$.
\end{prop}

\begin{proof}
The Dehn twist along a simple closed curve isotopic to the boundary is
represented by the twist automorphism of the $\DG$-module that labels the boundary circle (see 
\cite[Sect.\,4.5]{lyub6}), and thus in the case at hand, by multiplication with the twist phase
$\theta_{(c,\idscs)}$. Now according to \eqref{eq:Vtwist} we have $\theta_{(c,\idscs)} \eq 1$.
It follows that the relevant quantum representations of $\mcg_{\Ng,1}$ indeed satisfy 
the relations of $\mcg_{\Ng,0}$.
\end{proof}

\begin{rem}\label{rem:orbit}
An analogue of Theorem \ref{thm:notInvBasis} is valid for quantum representations
that come from other factorizable ribbon Hopf algebras than $\DG$:
Any quantum representation is contained in the one
obtained by taking $X \eq L^{\otimes \Ng}$ for a suitable value $\Ng$ of the genus, in which 
case one deals with endomorphisms of $L^{\otimes \Ng}$. Moreover, for obtaining the orbit
structure of such a particular representation it suffices to consider the subspace that is generated from the
distinguished element $\id_{L^{\otimes \Ng}}$ in $\End_{H\Mods}(L^{\otimes \Ng})$.
\end{rem}

\medskip

Before continuing, let us verify that we reproduce the explicit form of the 
matrices $T$ and $S$ from \cite{kssb,cogR} which generate the representation of
$\mcg_{1,0} \,{\cong}\, \slz$ on $\Hom_{\DG\Mods}(V_0,\DGa)$. 
For simplicity we assume that $\Bbbk$ has characteristic zero, so that 
we can work with ordinary characters of $\DG$-modules.
For $\Ng \eq1$ Theorem \ref{thm:InvBasis} implies that $\Hom_{\DG\Mods}(V_0,\DGa)$ is
isomorphic to $C(\DG)$, the space of central forms as defined in \eqref{eq:CDG}, which 
has a canonical basis given by the simple $\DG$-characters. 
Isomorphism classes of simple $\DG$-modules are in bijection with pairs $(c,\alpha)$ 
with $c\iN C_G$ and $\alpha$ a simple character of $Z_{\gc}$ for a chosen representative 
$\gc$ of $c$ (see Appendix \ref{ssec:DGmod} for a description of the simple 
$\DG$-modules $V_{(c,\alpha)}$ and their characters $\chi_{(c,\alpha)}$).
Denote by $\Xi_{(c,\alpha)}\colon Z(Z_{\gc})\To \ko$ the central character of a simple 
$\ko[Z_{\gc}]$-module with character $\alpha$; then 
  \be
  \Xi_{(c,\alpha)}(x) = \frac{\alpha(x)}{\alpha(e)} \,.
  \ee
Combining formula \eqref{eq:DGT..} for $T\coa$ with the expression \eqref{eq:DGchar2} 
for $\chi_{(c,\alpha)}$
and the fact that $\gc$ is central in $Z_{\gc}$, it follows that $T\coa$ acts as
  \be
  T\coa :\quad \chi_{(c,\alpha)} \,\longmapsto\, \Xi_{(c,\alpha)}(\gc)\, \chi_{(c,\alpha)} \,.
  \ee
The matrix affording the $T\coa$-transformation on simple $\DG$-characters thus takes the form
  \be
  \big( T\coa \big)_{(c,\alpha),(c',\alpha')} = \del{(c,\alpha),(c',\alpha')}\, \Xi_{(c,\alpha)}(\gc) \,.
  \label{eq:ModTmatrixChar}
  \ee
Further, according to \eqref{eq:Vtwist} $\Xi_{(c,\alpha)}(\gc)$ equals the twist phase 
$\theta_{(c,\alpha)}$ of the simple module $V_{(c,\alpha)}$.
Next consider the matrix affording the $S\coa$-transformation of simple $\DG$-characters by
  \be
  S\coa :\quad \chi_{(c,\alpha)} \,\longmapsto \sum_{(c',\alpha')} \big( S\coa\big)_{(c,\alpha),(c',\alpha')}\,
  \chi_{(c',\alpha')} \,.
  \ee
Let us tensor this formula with $\chi_{(\overline{c'},\alpha')}$ and apply
the Hopf pairing according to \eqref{eq:HopfpairingChar}.
Then using again \eqref{eq:DGchar2} and \eqref{eq:DGT..},
as well as the formula \eqref{eq:DGaHopfpairing} for the Hopf pairing, we obtain
  \be
  \big(S\coa\big)_{(c,\alpha),(c',\alpha')} = \frac{1}{|Z_{\gc}|\,|Z_{g_{c'}}|}
  \sum_{k\in G \colon {}^kg_{c'}\in Z_{\gc},\,\, \gc^k\in Z_{g_{c'}}}  \!\!\!\!
  \alpha\big( {}^{k\,\imath_{c'}\inv}\!(g_{c'}\inv)\big)\, \alpha'\big( (\gc\inv)^k\big) 
  \label{eq:ModSmatrixChar}
  \ee
(here we use that $\mathrm{char}(\Bbbk)$ does not divide $|G|$)
with $\imath_c$ a group element satisfying $g_{\overline c}\inv \eq {}^{\imath_c\!}\gc^{}$
(see the text preceding Equation \eqref{def-imath}). The results \eqref{eq:ModTmatrixChar}
and \eqref{eq:ModSmatrixChar} coincide with the formulas (2.14) and (2.12) of \cite{cogR}.

In the basis $\Gamm \eq \{\gamma_d\}_{d\in \mu_1\inv(c_e)}$
of $C(\DG)$ that was introduced in Theorem \ref{thm:InvBasis} the $T\coa$- and
$S\coa$-transformations read
  \be
  T\coa:\quad \gamma_{d_{(g|x)}} \mapsto \gamma_{d_{(g|g\inv x)}} \qquand
  S\coa:\quad \gamma_{d_{(g|x)}} \mapsto \gamma_{d_{(x\inv|g)}} \,,
 \label{eq:ModSmatrixDC..}
  \ee
hence the matrices representing these transformations are permutation matrices, 
as expected. An immediate consequence of the matrices being of this form is 
that the order of $T\coa$ coincides with the exponent of $G$. Note that even though 
the results \eqref{eq:ModTmatrixChar} and \eqref{eq:ModSmatrixChar} assume a restriction 
on the characteristic of $\ko$, the formulas \eqref{eq:ModSmatrixDC..} hold for 
an arbitrary ground field $\ko$.

%%%%%%%%%%%%%%%%%%%%%%%%%%%%%%%%%%%%%%%%%%%%%%%%%%%%%%%%%%%%%%%%

\section{Mapping class group actions on \texorpdfstring{\boldmath $2\Ng$}{2\Ng}-conjugacy 
         classes}\label{sec:mcg2Nconj}

Henceforth we focus on the subset of quantum representations which we
established in Theorems \ref{thm:InvBasis} and \ref{thm:notInvBasis}. All
of these representations satisfy the relations of $\mcg_{\Ng,0}$; for simplicity we will 
refer to them as representations of $\mcg_{\Ng}$. Moreover, all of 
them are permutation representations obtained by linearizing $\mcg_\Ng$-actions on 
the sets $C^{2\Ng}_G$ of $2\Ng$-conjugacy classes. Therefore we now turn to investigating 
the action of mapping class groups on $2\Ng$-conjugacy classes. 
We have already established, in Lemma \ref{lem:endoproperties} as well as in Theorems 
\ref{thm:InvBasis} and \ref{thm:notInvBasis}, that the function 
$\mu_\Ng\colon C^{2\Ng}_G\To C_G$ \eqref{eq:muN} is $\mcg_\Ng$-invariant, so there is 
at least one orbit, namely $\mu_\Ng\inv(c)$, for each conjugacy class $c\iN \im(\mu_\Ng)$.

\subsection{General properties}

For brevity we re-express the formulas \eqref{eq:DGT..} and \eqref{eq:DGLM} as permutations of the label sets 
$G^{\times 2}$ and $C^{2}_G$, respectively of $G^{\times 2\Ng}$ and $C^{2\Ng}_G$.
Thus we write
  \be
  \bearll
  T_G: & (g|x) \,\longmapsto\, (g|g\inv x) \,, 
  \Nxl3
  S_G: & (g|x) \,\longmapsto\, (x\inv|g^x) \,,
  \Nxl3
  U_G: & (g|x) \,\longmapsto\, (xg|x) \,,
  \Nxl3
  L_G: & \big( (g|x),(h|y)\big)	\,\longmapsto\, \big( (g|({}^xh)g\inv x),({}^{(g^x)}h|(g^x)h\inv y) \big) \,,
  \Nxl3
  M_{G,m}: & \big( (g|x), (\bh|\by) \big) 
  \Nxl2
  & \quad \longmapsto\,
           \big( (g|x((g^x)\tilde\mu_m(\bh|\by))\inv), ({}^{(g^x)\tilde\mu_m(\bh|\by)}
           \bh|{}^{(g^x)\tilde\mu_m(\bh|\by)}\by) \big) \,,
  \eear
  \label{eq:GpermTS..M}
  \ee
where in the last formula $m$ takes values in $\{1,2,...\,,\Ng{-}1\}$.
When acting on $G^{\times 2\Ng}$ for some definite value of $\Ng$ we write $T_G^{(i)}$, 
$S_G^{(i)}$ and $U_G^{(i)}$, with $i \eq 1,2,...,, \Ng$, for the action of the respective 
permutations on the $(\Ng{-}i{+}1)$th pair of labels, i.e.
  \be
  X_G^{(i)} = \big( \id_{G^{\times 2(\Ng-i)}},X_G,\id_{G^{\times 2(i-1)}} \big)
  \label{X=UTS}
  \ee
for $X \eq T, S, U$. Analogously $L_G^{(k)}$ with $k \eq 1,2, ...\,,\Ng{-}1$ is given by
  \be
  L_G^{(k)} = \big( \id_{G^{\times 2(\Ng-k-1)}},L_G,\id_{G^{\times 2(k-1)}} \big) \,.
  \label{Y=OL}
  \ee
Finally, 
  \be
  M_G^{(k)} = \big( \id_{G^{\times 2(\Ng-k)}},M_{G,k-1}^{} \big)
  \label{Y=M}
  \ee
for $k\eq 2,3,...\,,\Ng$. 

\begin{rem}
The permutations $U_G^{(i)}$, $T_G^{(i)}$, and $L_G^{(k)}$ coincide with the generators of 
the $\mcg_{\Ng,0}$-action found in equations (3.12), (3.13) and (3.14) of \cite{bant4}.
\end{rem}

\medskip

Lemma \ref{lem:endoproperties} immediately implies that each of 
the maps \eqref{X=UTS}\,--\,\eqref{Y=M}
maps $2\Ng$-conjugacy classes bijectively to $2\Ng$-conjugacy classes, and therefore induces
a permutation of the set $C^{2\Ng}_G$ that preserves the map $\mu_\Ng \colon C^{2\Ng}_G\To C_G$. 
By abuse of notation we use the same symbols $T_G^{(i)}$ etc.\ for these permutations 
of $C^{2\Ng}_G$. We have shown that these permutations define a group homomorphism
  \be
  \sigma^\Ng_G:\quad \mcg_\Ng \rightarrow \mathrm{Sym}(C^{2\Ng}_G) \,,
  \label{sigmaGN}
  \ee
where $\mathrm{Sym}(X)$ denotes the group of permutations of a finite set $X$. Moreover, 
for each $c\iN C_G$ the subset $\mu_\Ng\inv(c) \,{\subset}\, C^{2\Ng}_G$ is 
$\mcg_\Ng$-invariant. Note, however, that the set $\mu_\Ng\inv(c)$ may be empty for 
some $c\iN C_G$; e.g.\ for $\Ng \eq 1$ it is non-empty iff $c\iN C_G'$.

The rest of this subsection is devoted to establish general features of $C^{2\Ng}_G$ as 
a $\mcg_\Ng$-set. Even though for each conjugacy class $c\iN \im(\mu_\Ng)$ there exists at 
least one orbit, namely the set $\mu_\Ng\inv(c)$, it may happen that the map $\mu_\Ng$ is 
a crude invariant, so $\mcg_\Ng$ does not necessarily act transitively on these orbits.

Let us begin by stating two easy results related to invariant subsets of $C^{2\Ng}_G$.
By construction every $2\Ng$-conjugacy class is a subset of the product of $2\Ng$ 
conjugacy classes. In particular there are precisely $|Z(G)|^{2\Ng}$ $2\Ng$-conjugacy 
classes consisting of a single element. Thus we have

\begin{prop}\label{prop:CenterSubgroupInv}~
\\[2pt]	
{\rm (i)}\, The set $\{d\iN C^{2\Ng}_G \,|\, d\,{\subseteq}\, Z(G)^{\times 2\Ng}\}$ 
is mapping class group invariant.
\\[2pt]	
{\rm (ii)}\, For any subgroup $H \,{\leq}\, G$, the set $C^{2\Ng}_{H\le G}$ consisting of
$2\Ng$-conjugacy classes with representatives in $H^{\times 2\Ng}$ is mapping class group invariant.
\\[2pt]
{\rm (iii)}\, If $H\,{<}\, G$ is a proper subgroup, then $C^{2\Ng}_{H\le G}\,{\subset}\, C^{2\Ng}_G$ 
is a proper subset.
\end{prop}

\begin{proof}
The statements (i) and (ii) are obvious. Claim {\rm (iii)} follows immediately from 
the fact that no proper subgroup $H \,{<}\, G$ of a finite group $G$ can intersect all conjugacy 
classes of $G$: there must exist some $d\iN C^{2\Ng}_G$ such that $d \,{\notin}\, C^{2\Ng}_{H\le G}$.
\end{proof}

Since the center $Z(G)$ is not empty, Proposition \ref{prop:CenterSubgroupInv}\,(i)
implies that for any non-trivial $G$ and any $\Ng$ the mapping class group action 
on $C^{2\Ng}_G$ has at least two orbits. But if $Z(G)$ is non-trivial, 
there is even more information to be obtained from the center.
Namely, let $Z(G)^{\times 2\Ng}$ act on $G^{\times 2\Ng}$ by left multiplication.
This action descends to an action on $2\Ng$-conjugacy classes:
  \be
  (\mathbf{a}|\mathbf{b}):\quad  d_{(\bh|\by)} \mapsto d_{(\mathbf{a}\bh|\mathbf{b}\by)}
  \ee
for $(\mathbf{a}|\mathbf{b})\iN Z(G)^{\times 2\Ng}$.
Note that the $Z(G)^{\times 2\Ng}$-action on $C^{2\Ng}_G$ preserves $\mu_\Ng$.
Now for any word $X$ in the mapping class group generators we have
  \be
   (\mathbf{a}|\mathbf{b})\circ X = X\circ X\inv\circ(\mathbf{a}|\mathbf{b})\circ X
  \equiv X\circ (\mathbf{a}|\mathbf{b})^X
  \ee
as permutations of $C^{2\Ng}_G$. In fact, conjugating the action of $Z(G)^{\times 2\Ng}$ by a mapping 
class is implemented by an automorphism of $Z(G)^{\times 2\Ng}$. A straightforward calculation gives
  \be
  \bearll
  (a|b)^{T_G} = (a|ab) \,,
  & \big((a|b),(c|d)\big)^{L_G} = \big((a|abc\inv),(c|a\inv cd)\big) \,,
  \Nxl3
  (a|b)^{S_G} = (b|a\inv) \,, \quad§
  & \big( (a|b),(\mathbf{c}|\mathbf{d})\big)^{M_{G,m}} = \big((a|ab),(\mathbf{c}|\mathbf{d})\big) \,,
  \Nxl3
  (a|b)^{U_G} = (ab\inv|b) \,.
  \eear
  \ee
These are indeed automorphisms, which straightforwardly extend to all of the 
generators $T^{(i)}_G$, $S^{(i)}_G$, $U^{(i)}_G$, $L^{(i)}_G$ and $M^{(i)}_G$. 
In other words, we have

\begin{prop}\label{prop:HeisenbergAction} ~
\\[2pt]
{\rm (i)}\, Conjugation by a mapping class implements an automorphism on $Z(G)^{\times 2\Ng}\!$.
\\[3pt]
{\rm (ii)}\, The set $C_G^{2\Ng}$ of $2\Ng$-conjugacy classes carries an action of a semidirect product 
of the groups $\,\mcg_\Ng$ and $Z(G)^{\times 2\Ng}\!$.
\\[3pt]
{\rm (iii)}\, The mapping class group action on the set $C^{2\Ng}_G$ permutes isomorphic 
$Z(G)^{\times 2\Ng}$-or\-bits. In particular, the union of all $Z(G)^{\times 2\Ng}$-or\-bits 
of a given isomorphism type is mapping class group invariant.
\end{prop}

\begin{proof}
The statements (i) and (ii) are immediate consequences of the discussion preceding the 
proposition. (iii) follows directly from (i).
\end{proof}

The group $\mathrm{Out}(G)$ of outer automorphisms of $G$ acts on $G^{\times 2\Ng}$ 
through the diagonal embedding and thus induces an action on the set $C^{2\Ng}_G$ of 
$2\Ng$-conjugacy classes. Comparison with \eqref{eq:GpermTS..M} shows that
the action of $\mathrm{Out}(G)$ on $C^{2\Ng}_G$ commutes with the action of $\mcg_\Ng$. 
Thus we have

\begin{prop}\label{prop:OutAction}
The mapping class group action on $C^{2\Ng}_G$ permutes isomorphic $\mathrm{Out}(G)$-orbits. 
In particular, the union of all $\mathrm{Out}(G)$-orbits of a given isomorphism
type is mapping class group invariant.
\end{prop}

For any $g\iN G$ there is a distinguished mapping class group 
orbit in $C^{2\Ng}_G$: the orbit generated from the $2\Ng$-conjugacy class containing
the tuple $\big( (g|e),(e|e),...\,,(e|e)\big)$. More generally, for any $g \iN G$
we can consider all orbits emanating from $2\Ng$-conjugacy classes that are represented 
by elements of the type
  \be
  \big( (g^{m_1}|g^{n_1}),(g^{m_2}|g^{n_2}),...\,,(g^{m_\Ng}|g^{n_\Ng})\big) \,,
  \label{gm1gm2...}
  \ee
and one might suspect that these orbits have a particularly simple structure.
Indeed, denoting by $\mathrm{o}(g)$ the order of $g$, we have

\begin{thm}\label{thm:gTupleOrbits}
~\\
{\rm (i)} The mapping class group orbit in $C^{2\Ng}_G$ that is generated from the 
$2\Ng$-conjugacy class containing the element \eqref{gm1gm2...} includes the 
$2\Ng$-conjugacy classes containing all $2\Ng$-tuples 
$\big( (g^{p_1}|g^{q_1}),...,(g^{p_\Ng}|g^{q_\Ng})\big)$ which satisfy
  \be
  \mathrm{gcd}(m_1,...\,,m_\Ng,n_1,...\,,n_\Ng,\mathrm{o}(g))
  = \mathrm{gcd}(p_1,...\,,p_\Ng,q_1,...\,,q_\Ng,\mathrm{o}(g)) \,.
  \ee
Moreover, the $\mcg_\Ng$-action on $2\Ng$-conjugacy classes of this type factors 
through to an action of $\Sp(2\Ng,\zet)$.
\\[2pt]
{\rm (ii)} 
If any two different elements of the cyclic subgroup $\zet/\mathrm{o}(g)\zet \,{\le}\, G$ 
generated by $g \iN G$ belong to distinct conjugacy glasses,
then the mapping class orbit of the $2\Ng$-conju\-ga\-cy class that contains 
$\big( (g^{m_1}|g^{n_1}),(g^{m_2}|g^{n_2}),...\,,(g^{m_\Ng}|g^{n_\Ng})\big)$ is 
characterized by the number $\mathrm{gcd}(m_1,...\,,m_\Ng,n_1,...\,,,n_\Ng,\mathrm{o}(g))$, 
and the collection of such orbits coincides with the  collection of orbits of the 
defining action of $\Sp(2\Ng,\zet/\mathrm{o}(g)\zet)$ on $(\zet/\mathrm{o}(g)\zet)^{2\Ng}\!$.
\end{thm}

\begin{proof}
(i) We first show that the $\mcg_\Ng$-action on $2\Ng$-tuples of the stated type 
factors through to an $\Sp(2\Ng,\zet)$-action. For $a$ and $b$ oriented simple 
closed curves on a closed oriented surface, denote $[a]$ and $[b]$, respectively, 
their homology classes. A Dehn twist along $b$ acts on $[a]$ as (see e.g.\
\cite[Prop.\,6.3]{FAma})
  \be
  T_b:\quad [a]\mapsto [a]+\langle a,b\rangle\, [b] \,,
  \label{eq:SymplecticRep}
  \ee
where $\langle\cdot\,,\cdot\rangle$ is the intersection form. According
to \eqref{eq:GpermTS..M} we have
  \be
  \hsp{-.8}\bearll
  T_G: & (g^m|g^n) \,\mapsto\, (g^m|g^{n-m})\,,
  \Nxl3
  U_G: & (g^m|g^n) \,\mapsto\, (g^{m+n}|g^n)\,,
  \Nxl3
  L_G: & \big((g^k|g^l),(g^m|g^n)\big) \,\mapsto\, \big((g^k|g^{l+m-k}),(g^m|g^{n+k-m})\big)\,,
  \Nxl3
  M_{G,p}: \!\!& \big((g^k|g^l),(g^{m_1}|g^{n_1}),...\,,(g^{m_p}|g^{n_p})\big)
  \Nxl2& \hsp{7.6}
  \mapsto\, \big((g^k|g^{l-k}),(g^{m_1}|g^{n_1}),...\,,(g^{m_p}|g^{n_p})\big) \,.
  \eear
  \ee
These maps generalize in an obvious way to the action of the $\mcg_\Ng$-generators 
on $2\Ng$-tu\-ples. Recalling from Definition \ref{def:coendendo}\,(ii) the relation between 
the generators $T_G^{(i)}$ etc.\ and Dehn twists, by comparison with the Dehn twist
action \eqref{eq:SymplecticRep} one verifies that the so obtained action coincides with 
the one of Dehn twists on homology. More precisely, there exists a choice of orientations 
of the curves $a_i$, $b_i$, $c_i$, $d_i$, $e_i$ defined in \eqref{thesurface}
such that the action of Dehn twists 
on a tuple $\big((g^{m_1}|g^{n_1}),...\,,(g^{m_\Ng}|g^{n_\Ng})\big)$ can be identified 
with the action on the homology class $m_1[d_1]+n_1[b_1]+...+m_\Ng[d_\Ng]+n_\Ng[b_\Ng]$ via 
\eqref{eq:SymplecticRep} (a useful relation is $[b_i] \eq {\pm}\, ([d_{i-1}]{-}[d_i])$, 
where the sign depends on the choice of relative orientations).
\\[2pt]
To establish the remaining part of (i), note that the identification of first homology 
classes and tuples of the relevant type gives a surjective map from the first homology 
(with coefficients in $\zet$) to a subset of $C^{2\Ng}_G$, equivariant with respect to 
the action of $\Sp(2\Ng,\zet)$, i.e.\ symplectic orbits on homology are mapped to orbits 
on $2\Ng$-conjugacy classes. We therefore first determine the symplectic orbits of the 
defining action on $\zet^{\times 2\Ng}$ (corresponding to the first homology with a 
chosen symplectic basis).
Now the \slz-or\-bits on $\zet^{\times 2}$ are characterized by the greatest 
common divisor, i.e.\ $(m|n),\,(p|q)\iN \zet^{\times 2}$ belong to the same orbit 
iff $\mathrm{gcd}(m,n) \eq \mathrm{gcd}(p,q)$. In particular, every \slz-orbit
has a unique representative of the form $(r|0)$. The group $\Sp(2\Ng,\zet)$ has a generating set 
consisting of the linear transformations 
  \be
  (m_i|n_i) \,\mapsto\, (m_i{+}n_i|n_i) \qquand  
  (m_i|n_i) \,\mapsto\, (n_i| \,{-}m_i)   
  \label{eq:SpT..SpS}
  \ee
acting on an arbitrary single entry of an $\Ng$-tuple
$\big((m_1|n_1),...\,,(m_\Ng|n_\Ng)\big) \iN \zet^{\times 2\Ng}$, and of 
  \be
  \begin{array}{rcl}
  \big( (m_i|n_i),(m_{i+1}|n_{i+1})\big) &\!\!\mapsto\!\!&
  \big((m_i{-}n_{i+1}|n_i),(m_{i+1}{-}n_i|n_{i+1})\big) \qquand  
  \Nxl3
  \big((m_i|n_i),(m_{i+1}|n_{i+1})\big) &\!\!\mapsto\!\!&
  \big((m_{i+1}|n_{i+1}),(m_i|n_i)\big)   
  \eear
  \label{eq:SpMix..SpSwap}
  \ee
acting on two consecutive entries of the tuple (see e.g.\ \cite{burk} for 
$\Ng \eq 2$, and \cite[Sect.\,6.1.1]{FAma} for general $\Ng$).
Clearly for each $i \eq 1,2,...\,,\Ng$ the transformations \eqref{eq:SpT..SpS} 
generate subgroups, each of which is isomorphic to \slz\ acting on pairs of 
integers. Thus every orbit has a representative tuple of the type
$
  \big((r_1|0),(r_2|0),...\,,(r_\Ng|0)\big) $.
Further, by a combination of transformations  of the type
\eqref{eq:SpMix..SpSwap}, for each $i$ we can add any integer multiple of 
$r_{i+1}$ to $r_i$ without changing any other entries. Specifically, choose that 
multiple to be $1$, use \slz\ to replace one of the non-zero entries with 
$\mathrm{gcd}(r_i,r_{i+1})$, and then add the result multiplied by 
$-r_{i+1}/\mathrm{gcd}(r_i,r_{i+1})$ to $r_{i+1}$; hereby the latter entry is set to 
zero. Combining this procedure with further swaps of pairs of entries, inductively we 
end up with the tuple
  \be
  \big((\mathrm{gcd}(m_1,...\,,m_\Ng,n_1,...\,,n_\Ng)|0),(0|0),...\,,(0|0)\big) \,.
  \label{eq:SpRep1}
  \ee
Since the generators \eqref{eq:SpT..SpS} and \eqref{eq:SpMix..SpSwap} preserve the 
$\mathrm{gcd}$ of the entries, the element thus obtained characterizes the orbit. 
\\
In the considerations above we have, however, ignored that the group element $g$ 
has finite order $\mathrm o(g)$. To remedy this omission we must consider the 
symplectic action on $(\zet/\mathrm{o}(g)\zet)^{\times 2\Ng}$. Reducing modulo 
$\mathrm{o}(g)$ it is still true that every orbit has a representative of the form 
\eqref{eq:SpRep1}, albeit the non-zero entry is reduced modulo $\mathrm{o}(g)$. Indeed, in 
$(\zet/\mathrm{o}(g)\zet)^{\times 2\Ng}$ the two elements $\big((r|0),(0|0),...,(0|0)\big)$ and 
$\big((r|\mathrm{o}(g)),(0|0),...,(0|0)\big)$ are equivalent, and we can reduce the latter
to $\big((\mathrm{gcd}(r,\mathrm{o}(g))|0),(0|0),...,(0|0)\big)$ using \slz. 
This concludes the proof of (i).
\\[4pt]
If any two different elements of the subgroup $\langle g\rangle$ generated by $g$ 
belong to distinct conjugacy classes, then any two different $2\Ng$-tuples of integers 
in the interval $[0,\mathrm{o}(g){-}1]$ determine different $2\Ng$-conjugacy classes. In 
other words, the map from $(\zet/\mathrm{o}(g)\zet)^{\times 2\Ng}$ to $C^{2\Ng}_G$ that 
associates to the $2\Ng$-tuple
$\big((m_1|n_1),...\,,(m_\Ng|n_\Ng)\big)$ the $2\Ng$-conjugacy class containing 
$\big((g^{m_1}|g^{n_1}),...,(g^{m_\Ng}|g^{n_\Ng})\big)$ is an injective 
$\Sp(2\Ng,\zet/\mathrm{o}(g)\zet)$-map. This implies that the $\mcg_\Ng$-orbits of 
such $2\Ng$-conjugacy classes coincide with the orbits of the defining action 
of $\Sp(2\Ng,\zet/\mathrm{o}(g)\zet)$ on $(\zet/\mathrm{o}(g)\zet)^{\times 2\Ng}$. 
It remains to be shown that the number $\mathrm{gcd}(m_1,...,m_\Ng,n_1,...,n_\Ng,\mathrm{o}(g))$ 
characterizes the orbits. To see this, we analyze how this number changes under 
the $\Sp(2\Ng,\zet/\mathrm{o}(g)\zet)$-action; it suffices to consider the 
transformation in \eqref{eq:SpT..SpS} and the one in the first line of \eqref{eq:SpMix..SpSwap}. 
Now clearly one has $\mathrm{gcd}(p,q) \eq \mathrm{gcd}(p{+}q,q)$ for any pair 
$(p,q)$ of integers, as well as, reducing modulo $k$, 
$\mathrm{gcd}(p,q,k) \eq \mathrm{gcd}(p{+}q,q,k)$. In the same way it follows that 
$\mathrm{gcd}(m,n,p,q) \eq \mathrm{gcd}(m{-}q, $\linebreak[0]$	  
n, p{-}n,q)$ for any quadruple $(m,n,p,q)$ of integers, as well as
  \be
  \mathrm{gcd}(m,n,p,q,k) = \mathrm{gcd}(m-q,n,p-n,q,k) \,.
  \ee
We conclude that the number $\mathrm{gcd}(m_1,...,m_\Ng,n_1,...\,,n_\Ng,\mathrm{o}(g))$ 
indeed characterizes the orbit of $\big( (m_1|n_1),(m_2|n_2),...\,,(m_\Ng|n_\Ng)\big)$, 
thus finishing the proof of (ii).
\end{proof}

There is an interesting group action on the set $C^{2\Ng}_G$ of $2\Ng$-conjugacy classes.
Denote by $(\zet/\expG\zet)^\times$ the multiplicative group of integers modulo $\expG$, 
and for any $m \iN (\zet/\expG\zet)^\times$ define the power map
  \be
  \mathrm{p}_{m,\Ng}: \quad (\bg|\bx)\, \longmapsto\, \big( (g_1^m|x_1^m),(g_2^m|x_2^m) ,\,...\,,
  (g_\Ng^m|x_\Ng^m)\big) \,.
  \ee
Owing to $\gcd(m,\expG) \eq 1$, $\mathrm{p}_{m,\Ng}$ is a permutation of $G^{\times 2\Ng}$.
Moreover, $\mathrm{p}_{m,\Ng}$ is equivariant with respect to the $G$-action, and so 
induces a well defined permutation of $C^{2\Ng}_G$ for every $\Ng\iN\mathbb{N}$.
For $\Ng \eq 1$ we simplify notation by setting $\mathrm{p}_m \,{:=}\, \mathrm{p}_{m,1}$.

\begin{prop}\label{prop:GaloisAction}~
\\[2pt]
{\rm(i)}\,
The set $\{\mathrm{p}_{m,\Ng} \,|\, m \iN (\zet/\expG\zet)^\times\}$ furnishes an action of 
$(\zet/\expG\zet)^\times$ on $C^{2\Ng}_G$. 
\\[2pt]
{\rm(ii)}
For $\Ng \eq 1$, when restricted to the set of diconjugacy classes that are represented by pairs 
$(g|x)$ such that $[x,g]$ commutes with both $g$ and $x$, this action is equivariant with 
respect to the action of the genus-$1$ mapping class groups generated by $S_G$ and $T_G$. 
In particular, the action on $\mu_1\inv(c_s)$ is $\mcg_1$-equivariant for any $s\iN Z(G)$.
\\[2pt]
{\rm(iii)}
For any $\Ng$, when restricted to $2\Ng$-conjugacy classes that are represented by $2\Ng$-\-tu\-ples 
of the type $\big( (g^{k_1}|g^{l_1}),...\,,(g^{k_\Ng}|g^{l_\Ng})\big)$ for some $g\iN G$, 
this action is $\mcg_\Ng$-equi\-va\-riant.
\end{prop}

\begin{proof}
{\rm (i)}: 
Clearly, the maps  $\mathrm{p}_{m,\Ng}$ furnish a group action; $\mathrm{p}_{1,\Ng}$ is the 
identity and $\mathrm{p}_{m,\Ng} \cir \mathrm{p}_{n,\Ng} \eq \mathrm{p}_{mn,\Ng}$. 
 \\[4pt]
{\rm (ii)}: For $[x,g] \eq e$ we have
  \be
  T_G\circ \mathrm{p}_m\, (g|x) = (g^m|g^{-m}x^m) = (g^m|(g\inv x)^m)
  = \mathrm{p}_m\circ T_G\, (g|x) \,.
  \ee
More generally, in case $[x,g]$ commutes with $g$ and $x$ simple rearrangements of 
group elements give $(g^m|g^{-m}x^m) \eq (g^m|(g\inv x)^m[x,g]^m)$ and
$(g^m|(g\inv x)^m)^g = (g^m|(g\inv x)^m[x,g]^m)$. Thus in this case we still have 
$T_G\cir \mathrm{p}_m(g|x) \eq \mathrm{p}_m\cir T_G(g|x)$. Further, without any assumptions 
on the pair $(g|x)$ we have
  \be
  \bearll
  S_G\circ \mathrm{p}_m:~& (g|x) \,\mapsto\, (x^{-m}|(g^m)^{x^m})
  = ((x^{-m})^{x^{m-1}}|((g^m)^{x})^{x^{m-1}})
  \Nxl3
  {\rm and}
  \Nxl3
  \mathrm{p}_m\circ S_G: & (g|x) \,\mapsto\, (x^{-m}|(g^x)^m) = (x^{-m}|(g^m)^x) \,.
  \eear
  \ee
Hence the action of $\mathrm{p}_m$ commutes with the action of $S_G$ on $C^2_G$, and with 
the action of $T_G$ on the subset represented by pairs $(g|x)$ with the property that
$[x,g]$ commutes with $g$ and $x$, thus proving also (ii).
\\
The proof of (iii) is easy and left to the reader.
\end{proof}

\begin{rem}~
In a holomorphic orbifold conformal field theory the $(\zet/\expG\zet)^\times$-action on 
$\mu_1\inv(c_e)$ as described in Proposition \ref{prop:GaloisAction}\,(i) corresponds to the 
action of the Galois group $\mathrm{Gal}(\mathbb{Q}(\mathrm e^{2 \pi \mathrm i/\expG})/\mathbb{Q})$ 
on the labels of simple $\DG$-modules as described in \cite[Sect.\,2.2]{cogR}.
\end{rem}

\medskip

Theorem \ref{thm:gTupleOrbits} shows that a quantum representation necessarily 
contains sub-represen\-ta\-ti\-ons (coming from the invariant set $\mu_\Ng\inv(c_e)$) 
that factor through to the symplectic action on the first homology of the surface. 
As can be inferred from the proof of the theorem, the underlying group actions can be 
fully understood as quotients of the defining action of $\Sp(2\Ng,\zet)$ on
$\zet^{\times 2\Ng}$. One first mods out by the order of a group element, which results in 
the natural action of $\Sp(2\Ng,\zet/\mathrm{o}(g)\zet)$ on $(\zet/\mathrm{o}(g)\zet)^{\times 2\Ng}$. 
For any given finite group $G$ it is then a straightforward exercise to work out 
the additional identifications of points of $(\zet/\mathrm{o}(g)\zet)^{\times 2\Ng}$ 
that are implied by reducing to $2\Ng$-conjugacy classes. 

A question of obvious interest is therefore whether also in general the mapping class 
group action factors through to the symplectic action. In other words, we would 
like to know whether the Torelli group $\tor_\Ng\,{\leq}\, \mcg_\Ng$ lies in the kernel of
the group homomorphism $\sigma_G^\Ng$ \eqref{sigmaGN}. 
At genus $g \eq 1$ this is not an issue, as
the genus-$1$ Torelli group is trivial. At genus $g \eq 2$ 
the situation is somewhat exceptional: it is known \cite{mess} that the Torelli group 
is generated by Dehn twists around genus-1 bounding simple closed curves. 
For genus $g \,{\geq}\, 3$, a classic result \cite{johnD7} 
shows that the Torelli group is generated by genus-1 bounding pair maps. 
Moreover, the mapping class group $\mcg_{\Ng,1}$ acts on the fundamental group $\pi$ 
of a surface of genus $\Ng$ and with one hole, and thus also on the
nilpotent quotients $\pi/\pi_{[m]}$ of  $\pi$ by the members 
$\pi_{[m]} \,{:=}\, \langle [\pi,\pi_{[m-1]}]\rangle$ of its lower central series.
Define $\mcg_{\Ng,1}(m)$ to be the kernel of 
this action on $\pi/\pi_{[m]}$. This defines a filtration, the \emph{Johnson filtration}
of $\mcg_{\Ng,1}$, with $\mcg_{\Ng,1}(0) \eq \mcg_{\Ng,1}$ and $\mcg_{\Ng,1}(1) \eq \tor_{\Ng,1}$. 
It is implied by the main result of \cite{johnD9}
that for $\Ng \,{\geq}\, 3$, $\mcg_{\Ng,1}(2)$ is the group generated by Dehn twists 
around bounding simple closed curves. Taking the suitable quotients to $\mcg_{\Ng,0}$ 
there are analogous statements for $\mcg_\Ng$, $\tor_\Ng$, and $\mcg_\Ng(m)$.
We are thus ready to state our main result:

\begin{thm} \label{thm:Torelli}
~\\[2pt]
{\rm (i)}\,
For $\Ng \,{\geq}\, 3$ the Torelli group $\tor_{\Ng}$ is represented non-trivially in 
the quantum representation associated with the group $G$ iff $G$ is non-abelian.
~\\[2pt]
{\rm (ii)}\,
The group $\tor_{2}$ is represented non-trivially iff there exists a pair 
$(g|x)\iN G\Times G$ such that $[g,x] \,{\notin}\, Z_{(g|x)}$.
~\\[2pt]
{\rm (iii)}\,
The groups $\mcg_{\Ng}(2)$ are represented non-trivially if there exists a pair 
$(g|x)\iN G\Times G$ such that $[g,x] \,{\notin}\, Z_{(g|x)}$.
\end{thm}

\noindent
It follows that for $\Ng \,{\geq}\, 3$ the quantum representation of $\mcg_\Ng$ 
factors through to the symplectic representation iff $G$ is abelian. For $\Ng \eq2$ 
it can still happen that the quantum representation factors through to the 
symplectic representation even if $G$ is non-abelian.

\begin{proof}
(i) Since $\Ng \,{\geq}\, 3$, $\tor_\Ng$ is generated by genus-1 bounding pair maps. 
All bounding pair maps of a given genus are conjugate in the mapping class group,
as follows e.g.\ from the change of coordinate principle of \cite{FAma}. It is 
thus enough to check the action of a single bounding pair map of genus one; we
choose the one generated by the cycles $d_2$ and $e_2$. Put differently, it 
suffices to check the action of $(T_G^{(1)})^{-1} \,{\circ}\, M_G^{(1)}$,
with $T_G^{(1)}$ and $M_G^{(1)}$ as defined in \eqref{X=UTS} and \eqref{Y=M}. Restricting
attention to the two `right-most' pairs $\big((g|x),(h|y)\big)$ in a $2\Ng$-tuple 
$(\mathbf g|\mathbf y)$, we have
  \be
  \bearl
  (T_G^{(1)})^{-1}\circ M_G^{(1)}:\quad 
  \Nxl3
  \qquad \big( (g|x),(h|y)\big)	\longmapsto
  \big((g \,|\, gx[y,h](g\inv)^{x}),({}^{g^x[h,y]}h \,|\, {}^{g^x[h,y]}y)\big) \,.
  \eear
  \label{eq:TorelN3}
  \ee
Thus if $G$ is abelian, then $\big((g|x),(h|y)\big)$ is mapped to itself, so the bounding 
pair map acts trivially. If, on the other hand, $G$ is non-abelian, then $\big((e|e),(h|y)\big)$ 
is mapped to $\big( (e \,|\, [y,h]),({}^{[h,y]}h \,|\,{}^{[h,y]}y)\big)$. Whenever
$h$ and $y$ do not commute, this is not conjugate to $\big((e|e),(h|y)\big)$, hence
the bounding pair map acts non-trivially. Moreover, since $\Ng \,{\geq}\, 3$, there do
exist $2\Ng$-tuples containing quadruples of the relevant type in $\mu_\Ng\inv(c_e)$ and 
in $\mu_\Ng\inv(c)$ for at least one $c \,{\neq}\, c_e$ (albeit not necessarily for 
every $c\iN \im(\mu_\Ng)$). Thus the proof of (i) is complete.
\\[2pt]
(ii) The group $\tor_2$ is generated by Dehn twists along separating simple closed 
curves of genus $1$ and thus, again by the change of coordinate principle, normally 
generated by a single curve of that type. Consider a subsurface of the following form:
  \be
  \begin{picture}(180,122)
  \put(0,0)    {\includegraphics[scale=0.6]{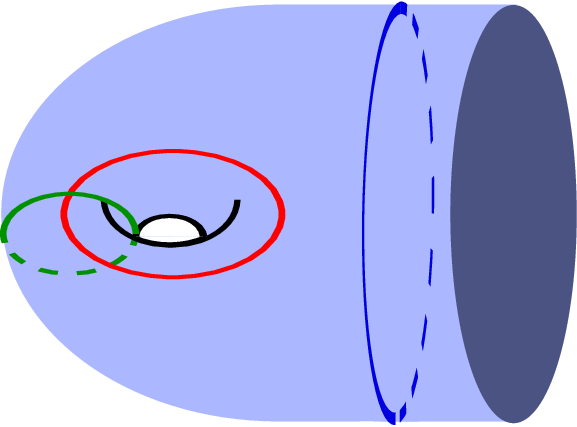}}
  \put(8,68)   {\scriptsize $a$}
  \put(50,82)  {\scriptsize $b$}
  \put(100,80) {\scriptsize $c$}
  \end{picture}
  \label{figureabove}
  \ee
Denote by $T_s$ the mapping class of the Dehn twist around a simple closed curve $s$.
It is easily checked (using, for instance, the star relation, see Equation (E) in 
\cite{gervS2}) that
  \be
  (T_a\circ T_b)^6=T_c 
  \ee
with $a$, $b$, and $c$ the curves shown in \eqref{figureabove}. (Likewise one sees
that $(T_a\inv \,{\circ}\, T_b\inv)^6 \eq T_c\inv$.) The Dehn twists $T_a$ and $T_b$ 
are represented on $2\Ng$-tuples by $T_G^{(1)}$ and $U_G^{(1)}$ 
respectively. Restricting to the pair on which these permutations act non-trivially we have
  \be
  T_G\circ U_G:\quad  (h|y) \mapsto (yh|h\inv)
  \ee
and thus
  \be
  \bearl
  (T_G\circ U_G)^2:\quad (h|y) \mapsto (y^h|h\inv y\inv) \,,
  \Nxl3
  (T_G\circ U_G)^3:\quad (h|y) \mapsto (h\inv [y,h]|(y\inv)^h) \,,
  \eear
  \ee
so that
  \be
  T_c:\quad \big((g|x),(h|y)\big) \mapsto \big((g|x),(h^{[y,h]},y^{[y,h]})\big) \,.
  \label{eq:genus2Torelliel}
  \ee
We conclude that the Dehn twist $T_c$, and thus the group $\tor_2$, is represented trivially iff 
$\big((g|x),(h^{[y,h]},y^{[y,h]})\big)$ is conjugate to $\big((g|x),(h|y)\big)$. This,
in turn, is the case iff $[y,h]\iN Z_{(h|y)}$ or $[y,h]\iN Z_{(g|x)}$. In other words, 
$\tor_2$ acts non-trivially iff there exist two pairs $(g|x)$ and $(h|y)$ such that 
$[y,h] \,{\notin}\, Z_{(h|y)} \,{\cup}\, Z_{(g|x)}$. Furthermore, two pairs with this 
property exist iff there exists a single pair $(g|x)$ such that $[g,x] \,{\notin}\, Z_{(g|x)}$
-- take $(h|y) \eq (g|x)$ or $(x|g)$.  Considering the action on $\big((g|x),(x|g)\big)$, 
it follows that if $\tor_2$ acts non-trivially, then it acts non-trivially on $\mu_2\inv(c_e)$. 
This establishes the statement (ii).
\\[2pt]
(iii) now follows because the group $\mcg_\Ng(2)$ contains in particular all Dehn twists 
along genus-1 separating curves, including the curve $c$ considered in the proof of (ii).
\end{proof}

\begin{rem}\label{rem:kernel}~
\\[2pt]
(i) The kernel of the quantum representation of $\mcg_{1,0}$ is a level-$\expG$ 
congruence subgroup \cite{coga2,bant14,soZh}.
It is natural to wonder whether this property extends to all genera.
Recall that the level-$k$ principal congruence subgroup $\mcg_{\Ng}[k]$ of 
$\mcg_{\Ng}$ is the kernel of the natural surjection to $\Sp(2\Ng,\zet/k\zet)$. 
Theorem \ref{thm:Torelli} immediately provides a negative answer, in general: 
$\mcg_\Ng[k]$ necessarily contains $\tor_\Ng$, which according to the theorem generically 
is non-trivially represented (unless $G$ is abelian).
\\[2pt]
(ii) Theorem \ref{thm:Torelli} implies that for abelian $G$, all quantum representations 
considered here factor through to the symplectic representation. Actually, it follows 
from the fact that in this case all simple objects of the monoidal category of $G$-modules 
are invertible that \emph{all} quantum representations, not only those discussed here, 
share this property (this can e.g.\ be seen with the help of the so-called TFT construction 
of \cite{fjfrs}).
\end{rem}

\medskip

Let us finally give a procedure to trivialize the action of $\tor_\Ng$.

\begin{prop}
The canonical surjection $G^{\times 2\Ng} \To (G^{\mathrm{ab}})^{\times 2\Ng}$ induces 
a surjective $\mcg_{\Ng}$-map $\pi_\Ng:C_G^{2\Ng}\To C_{G^{\mathrm{ab}}}^{2\Ng}$.
\end{prop}

\begin{proof}
It is easy to see that the permutations \eqref{eq:GpermTS..M} are well-defined on cosets.
We content ourselves to illustrate this property for the case of $T_G$. For $p,q\iN G'$
we have
  \be
  T_G: \quad (gp|xq) \,\longmapsto\, (gp|p\inv g\inv xq) 
  = (gp|g\inv x p\inv [p\inv,g\inv x]q)
  \ee
for any pair $g,x\iN G$, so that the claim follows from $p\inv [p\inv,g\inv x] q\iN G'$. 
Further, using that $G'$ is a normal subgroup of $G$, it follows immediately 
that the canonical surjection induces a surjection of $C^{2\Ng}_G$ onto 
$C^{2\Ng}_{G^\mathrm{ab}}$.
\end{proof}

This procedure is very crude, however. Indeed, by Theorem \ref{thm:gTupleOrbits} 
there exist, for any $G$, $\mcg_\Ng$-orbits on $C^{2\Ng}_G$ on which $\tor_\Ng$ acts 
trivially, whereas $G^{\mathrm{ab}} \eq \{e\}$ for any perfect group $G$. 
In other words, the restriction of $\pi_\Ng$ to orbits on which $\tor_\Ng$ acts 
trivially is in general not injective.

%%%%%%%%%%%%%%%%%%%%%%%%%%%%%%%%%%%%%%%%%%%%%%%%%%%%%%%%%%%%%%%%

\subsection{Genus 1: \texorpdfstring{\boldmath{\slz}}{\slz}-action on diconjugacy classes}

At genus $\Ng \eq 1$ there are features that are absent at higher genera. 
In particular, the \slz-action on $\mu_1\inv(c_e)$ can be expressed very concretely.

We first state two simple identities.

\begin{lemma}\label{lem:dis1}
If $[x,g]$ commutes with both $g$ and $x$, then we have
  \be
  (g^ax^b)^c = g^{ac}\,x^{bc}\, [x,g]^{ab\binom{c}{2}}
  \label{eq:almostcommutingpairs1}
  \ee
for any triple $a,b,c\iN\mathbb{Z}$, with the convention that 
$\binom{c}{2} \eq \mathrm{sgn}{(c)}\binom{|c|}{2}$
for $|c| \,{>}\, 1$ and $\binom{c}{2} \eq 0$ otherwise.
\end{lemma}

\begin{proof}
For $a,b,c\iN\mathbb{N}$ this follows easily by induction in $c$. The elementary identities
  \be
  \bearl
  [x,g]\inv = [x,g\inv]^g = [x,g\inv ] \qquad \text{and}
  \Nxl3
  [x,g]\inv = [x\inv ,g]^x = [x\inv ,g]
  \eear
  \ee
then imply that the formula \eqref{eq:almostcommutingpairs1} even holds for all $a,b\iN\mathbb{Z}$.
The extension to negative integers $c$ is directly implied by the stated convention.
\end{proof}        

\begin{lemma}\label{lem:dis2}
If $[x,g]$ commutes with both $g$ and $x$, then for any
$ \big( \begin{smallmatrix} k & l\\[2pt] m & n\end{smallmatrix} \big),
\big( \begin{smallmatrix} p & q\\[2pt] r & s\end{smallmatrix} \big) \iN \slz$ the pair
  \be
  \big( g^{pk+qm}x^{pl+qn}[x,g]^{pqlm+kl\binom{p}{2}+mn\binom{q}{2}} \,\big|\,
  g^{rk+sm}x^{rl+sn}[x,g]^{rslm+kl\binom{r}{2}+mn\binom{s}{2}} \big)
  \label{eq:M1actM2act2}
  \ee
is conjugate to
  \be
  (g^{pk+qm}x^{pl+qn} | g^{rk+sm}x^{rl+sn}[x,g]^P) \,,
  \label{eq:M1actM2act3}
  \ee
with the integer $P$ given by
  \be
  P = rslm + kl \Binom r2 + mn \Binom s2 - \big( (rl{+}sn)^2 - (rk{+}sm)^2 \big)
  \big( pqlm + kl \Binom p2 + mn \Binom q2 \big) \,.
  \label{theP}
  \ee
\end{lemma}

\begin{proof}
As a first step we prove that the integers $pk+qm$ and $pl+qn$ are coprime. To this end
set $d \,{:=}\, \mathrm{gcd}(pk{+}qm,pl{+}qn)$ and write $pk\,{+}\,qm \eq d a$ and
$pl\,{+}\,qn \eq db$. If $p \,{\ne}\, 0$ we can write
$ k \eq (da\,{-}\,qm)/p$ and $l \eq (db\,{-}\,qn)/p$, and because of $kn\,{-}\,lm \eq
\mathrm{det} \big( \begin{smallmatrix} k & l\\[2pt] m & n\end{smallmatrix} \big) 
\eq 1$ we have $n(da\,{-}\,qm)-m(db\,{-}\,qn) \eq p$ if and only if $nda\,{-}\,mdb \eq p$,
which in turn implies that $d$ divides $p$.
Since $qm \eq da-pk$, $qn \eq db-pl$ and $\mathrm{gcd}(p,q) \eq 1$, it then follows 
that in fact $d|m$ and $d|n$, and this implies further that $d \eq 1$ because also 
$\mathrm{gcd}(m,n) \eq 1$. The same conclusion can be reached when instead $p \eq 0$;
together it follows that $pk\,{+}\,qm$ is coprime to $pl\,{+}\,qn$, as claimed. 
Next, by another simple calculation we have
  \be
  (g^ax^b)^{g^cx^d} = g^a\,x^b\,[x,g]^{bc - ad}.
  \label{eq:almostcommutingconj}
  \ee
Combining this identity with the Euclidean algorithm (and remembering that $[x,g]$ commutes 
with $g$ and $x$) it follows that we can choose integers $a$ and $b$ in such a way that 
  \be
  \big(g^{pk+qm}x^{pl+qn}[x,g]^{pqlm+kl\binom{p}{2}+mn\binom{q}{2}}\big)^{g^ax^b}
  = g^{pk+qm}x^{pl+qn}.
  \ee
Conjugating the pair \eqref{eq:M1actM2act2} with $g^ax^b$ then establishes
that it is indeed conjugate to the pair \eqref{eq:M1actM2act3}
for some integer $P$. To determine this number, we note that
  \be
  (pk+qm)\,(rl+sn) - (pl+qn)\, (rk+sm) = (kn-lm)\,(ps-rq) = 1 \,,
  \label{eq:bezoutints}
  \ee
which implies that the integers $a$ and $b$ may be chosen as
  \be
  \bearl
  a = -\, (rl+sn)\, \big( pqlm + kl\, \Binom{p}{2} + mn\, \Binom{q}{2} \big) \,,
  \Nxl2
  b = -\, (rk+sm)\, \big( pqlm + kl\, \Binom{p}{2} + mn\, \Binom{q}{2} \big) \,.
  \eear
  \ee
As a result, the integer $P$ takes the value \eqref{theP}.
\end{proof}        

We can now state

\begin{thm}\label{thm:SL2action}
If $g,x\iN G$ are such that their commutator $[x,g]$ commutes with both $g$ and $x$, then
the element $\big(\begin{smallmatrix} a & b\\ c& d\end{smallmatrix}\big) \iN \slz$ 
acts on the diconjugacy class $(g|x)$ as
  \be
  \big(\begin{smallmatrix} a & b\\ c& d\end{smallmatrix}\big):\quad (g|x)
  \,\longmapsto\, (g^ax^b \,|\, g^cx^d) \,.
  \label{eq:concreteSL2Zaction}
  \ee
In particular, the formula \eqref{eq:concreteSL2Zaction} gives the \slz-action
on $\mu_1\inv(c_s)$ for any $s \,{\in}\, Z(G)$.
\end{thm}

In the particular case of $s \eq e$, the formula \eqref{eq:concreteSL2Zaction} for the
action on $\mu_1\inv(c_e)$ has already been presented in \cite[Eq.\,(5.2)]{kssb}.

\begin{proof}
We first show that the prescription
\eqref{eq:concreteSL2Zaction} defines an action of \slz\ on $\mu_1\inv(c_e)$.
Obviously the unit matrix acts as the identity transformation. Further we have
  \be
  \big( \begin{smallmatrix} p & q\\[2pt] r & s\end{smallmatrix} \big) \,
  \big( \begin{smallmatrix} k & l\\[2pt] m & n\end{smallmatrix} \big)
  = \big( \begin{smallmatrix} pk + qm & pl+qn\\[2pt] rk+sm & rl+sn\end{smallmatrix}\big):\quad
  (g|x) \,\longmapsto\, (g^{pk+qm}x^{pl+qn}|g^{rk+sm}x^{rl+sn}) \,,
  \label{M1M2act}
  \ee
while
  \be
  \big( \begin{smallmatrix} k & l\\[2pt] m & n\end{smallmatrix} \big): \quad
  (g|x) \,\longmapsto\, (g^kx^l|g^mx^n) 
  \ee
and
  \be
  \big( \begin{smallmatrix} p & q\\[2pt] r & s\end{smallmatrix} \big): \quad
  (g^kx^l|g^mx^n) \,\longmapsto\, \big( (g^kx^l)^p(g^mx^n)^q|(g^kx^l)^r(g^mx^n)^s\big) \,.
  \label{M1actM2act}
  \ee
Clearly, the expressions on the right hand sides of \eqref{M1M2act} and \eqref{M1actM2act} 
coincide if $[g,x] \eq e$.
Thus \eqref{eq:concreteSL2Zaction} indeed defines an action of \slz\ on 
$\mu_1\inv(c_e)$, and in fact even on the set of commuting pairs. Next we simply note 
         the specific cases
  \be
  \big( \begin{smallmatrix} 0 & -1\\[2pt] 1 & 0\end{smallmatrix}\big):\quad (g|x)\mapsto (x\inv|g)
  \qquad\text{and}\qquad
  \big( \begin{smallmatrix} 1 & 0 \\[2pt] 1 & 1\end{smallmatrix}\big):\quad (g|x)\mapsto (g|g\inv x) \,.
  \label{eq:STaction}
  \ee
Since these two matrices generate the group \slz, and since the permutations 
$S_G$ and $T_G$ act on commuting pairs in exactly the same way, it follows that $\mcg_1$ acts 
on $\mu_1\inv(c_e)$ as advertised.
 \\[3pt]
In generality, note that $S_G$ maps $(g|x)$ to $(x\inv|g^x) \eq ((x\inv)^x|g^x)$, so the action 
of $S_G$ and $T_G$ is still given by the formulas \eqref{eq:STaction} when interpreted as 
the action on diconjugacy classes, even for non-commuting pairs $(g|x)$. What remains to be
shown is that if $[x,g]$ commutes with both $g$ and $x$, then the expression on the right hand 
side of \eqref{M1actM2act} is conjugate to the one on the right hand side of \eqref{M1M2act}.
To see this, we observe that when combined  with the identity
  \be
  g^ax^b g^cx^d = g^{a+c}x^{b+d}\, [x,g]^{bc} ,
  \label{eq:almostcommutingpairs2}
  \ee
Lemma \ref{lem:dis1} implies that the right hand side of \eqref{M1actM2act} coincides 
with the expression \eqref{eq:M1actM2act2}.
Next we observe that the identity \eqref{eq:almostcommutingconj} implies that
  \be
  \big(g^{rk+sm}x^{rl+sn}\big)^{g^{pk+qm}x^{pl+qn}} \!
  = g^{rk+sm}x^{rl+sn} [x,g]^{(pk+qm)(rl+sn)-(pl+qn)(rk+sm)}.
  \ee
Invoking \eqref{eq:bezoutints}, this in turn implies
  \be
  \begin{array}{r}
  \big( g^{pk+qm}x^{pl+qn} | g^{rk+sm}x^{rl+sn}[x,g]^P \big)^{(g^{pk+qm}x^{pl+qn})^{-P}}
  \quad \Nxl3
  = \big( g^{pk+qm}x^{pl+qn} | g^{rk+sm}x^{rl+sn}\big) \,,
  \eear
  \ee
with $P$ as in Lemma \ref{lem:dis2}. It follows that, under the 
assumed conditions on $g$ and $x$, the right hand side of \eqref{M1actM2act} is 
conjugate to the right hand side of \eqref{M1M2act}. This completes the proof.
\end{proof}

\medskip

Recall from Remark \ref{rem:kernel}\,(i) that the kernel of the \slz-representation that
linearizes the action on $\mu_1\inv(c_e)$ is a level-$\expG$ congruence subgroup. We can 
now state the following generalization, the proof of which follows directly from 
Theorem \ref{thm:SL2action} together with Theorem \ref{thm:notInvBasis}.

\begin{thm}\label{thm:congruencesubgroup}
Let $g,x\iN G$ be such that $[x,g]$ commutes with both $g$ and $x$. Denote by 
$W_{g,x}$ the sub-representation of the quantum representation on 
$\Hom_{\DG\Mods}(V_{(c_{[x,g]},\idscs)},\DGa)$ corresponding to the linearization 
of the \slz-orbit that contains the diconjugacy class represented by $(g|x)$.
Then the kernel of $W_{g,x}$ is a level-$\expG$ congruence subgroup.
\end{thm}

\begin{rem} 
(i)\, Theorem \ref{thm:congruencesubgroup} does not generalize to arbitrary 
pairs $g, x$ of group elements. For instance, for $G \eq A_5$, which has exponent
$30$, the action of \slz\ on the orbit containing the diconjugacy class 
represented by $\big( (12345) \,|\, (345) \big)$ does not satisfy the relations of 
$\mathrm{SL}(2,\zet/P\zet)$ for $P \eq 30$, nor for any other $P\iN\mathbb{N}$.
More generally, writing out the actions of (some of) the relators 
of $\mathrm{SL}(2,\zet/ P\zet)$ from \cite[Lemma\,1]{coga2} yields expressions for
which there is no obvious reason for being trivial. Accordingly we expect that the 
kernels of the linearizations of \slz-actions on $C^2_G\backslash\,\mu_1\inv(c_e)$ 
will generically not be congruence subgroups.
 \\[3pt]
(ii)\, The counter example from part\,(i) shows in particular
that the formula \eqref{eq:concreteSL2Zaction} cannot hold on $C^2_G$ in general.
\end{rem}

Various specific examples of quantum representations are presented in Appendix \ref{sec:ex}.
As we will see, for any (non-trivial) group $G$ the symplectic orbits of $\mu_N\inv(c_e)$
guaranteed by Theorem \ref{thm:gTupleOrbits} take a universal form that is easy to describe
explicitly. The examples also indicate that if $G$ has trivial center,
$G$ has no proper subgroup (see Propositions \ref{prop:CenterSubgroupInv} and 
\ref{prop:HeisenbergAction}), and $\mathrm{Out}(G)$ is trivial (see Proposition 
\ref{prop:OutAction}), then each set $\mu_N\inv(c)$ will typically not have proper sub-orbits,
apart from the case $c \eq c_e$ for which there are generic symplectic orbits. Furthermore, 
the ratio of the numbers of symplectic to non-symplectic orbits is typically the smaller the further away $G$ is from being abelian;
that is, one expects that if $G$ is perfect and centerless, then many of the mapping class 
orbits are non-symplectic.

\newpage
%%%%%%%%%%%%%%%%%%%%%%%%%%%%%%%%%%%%%%%%%%%%%%%%%%%%%%%%%%%%%%%%%
\appendix

\section{Examples}\label{sec:ex}

In this appendix we use results from Section \ref{sec:mcg2Nconj} to determine the mapping 
class group orbits on $C^{2\Ng}_G$ for low genus $\Ng$ and a few specific groups $G$. 
For fixed group $G$ the number of elements of $C^{2\Ng}_G$ tends to grow rapidly 
with $\Ng$. To cope with this complication, various computations, e.g.\ determining 
the orbits for genus $\Ng \,{>}\, 1$, were performed on a computer, using GAP \cite{GAP}.
Concerning names of finite groups, we follow the terminology used in version 4.7.7 of GAP.

According to Theorem \ref{thm:gTupleOrbits}, among the permutation 
representations of $\mcg_\Ng$ produced by the double of a finite group there are 
always orbits obtained as quotients of symplectic orbits. Furthermore, these are 
quotients with respect to the action of power maps as described in Proposition 
\ref{prop:GaloisAction}\,(iii). The structure of such 
orbits is universal in the sense that the order of a group element $g$ together 
with the conjugacy classes containing all powers of $g$ already determines the 
corresponding mapping class orbit. For that reason we first give a short 
description of the possible quotients of symplectic orbits that may arise. 
Afterwards, in the list of examples, we largely refrain from a detailed 
description of this kind of orbits.

%%%%%%%%%%%%%%%%%%%%%%%%%%%%%%%%%%%%%%%%%%%%%%%%%%%%%%%%%%%%%%%%

\subsection{Finite symplectic orbits and their quotients}

By definition, the natural action of the symplectic group $\Sp(2\Ng,\zet)$ on 
$\big(\zet/k\zet\big)^{2\Ng}$ factors through the finite group 
$\Sp(2\Ng,\zet/k\zet)$. The orbits of $2\Ng$-conjugacy classes represented by elements 
of the form \eqref{gm1gm2...} can be identified with  $\Sp(2\Ng,\zet)$-orbits in 
$\big(\zet/k\zet\big)^{2\Ng}$, or orbits in quotients thereof.

\medskip

Let us explain how all orbits of this type may be obtained; 
later on we will largely restrict the discussion to non-symplectic orbits.
Assume first that $\Ng \eq 1$, and consider the action of \slz\ on $\big(\zet/k\zet)^2$.
Write the elements of $\big(\zet/k\zet)^2$ as $(m|n)$. The element $(0|0)$ is a 
fixed point, corresponding to the divisor $k$ of $k$; 
in addition there is one orbit for each divisor $d$ of $k$ with $d \,{<}\, k$, 
containing the element $(d|0)$ (see Theorem \ref{thm:gTupleOrbits}). The orbit 
containing $(d|0)$ consists of all elements $(p|q)$ such that $\mathrm{gcd}(p,q,k) \eq d$. 
The number of orbits is thus $\divfn(k)$, the number of divisors of $k$. 
It remains to describe the possible quotients of these orbits. To this end we must 
characterize the possible equivalence relations on the \slz-set $\big(\zet/k\zet)^2$. 
These are precisely the bijective \slz-maps, i.e.\ the automorphisms of the \slz-set 
$\big(\zet/k\zet)^2$. We have

\begin{lemma}
A bijection $\varphi$ of the set $(\zet/k\zet)^2$ commutes with the natural
\slz-ac\-tion iff $\varphi$ is a power map $\mathrm{p}_m$ for some 
$m\iN(\zet/k\zet)^\times\!$. Moreover, the group of \slz-equi\-variant automorphisms 
of $(\zet/k\zet)^2$ is $(\zet/k\zet)^\times\!$.
\end{lemma}

\begin{proof}
Let $\varphi\colon (m|n)\,{\mapsto}\, (\varphi_1(m,n)|\varphi_2(m,n))$ be an 
\slz-equi\-variant automorphism of $\big(\zet/k\zet\big)^2$.
Since the generator $T$ preserves the first entry of a pair,
$\varphi_1$ only depends on its first argument. We further have
  \be
  \begin{array}{l}
  S\circ\varphi ((m|n)) = (-\varphi_2(m,n)|\varphi_1(m)) \qquad {\rm and}
  \Nxl3  
  \varphi\circ S((m|n)) = (\varphi_1(-n)| \varphi_2(-n,m)) \,,
  \end{array}
  \ee
so that requiring that $\varphi$ commutes with $S$ implies that 
$\varphi_2(m,n) \eq \varphi_2(n) \eq \varphi_1(n)$. Then 
$T\cir\varphi  ((m|n)) \eq (\varphi_1(m)|\varphi_1(n) \,{-}\, \varphi_1(m))$ and
 $\varphi\cir T ((m|n)) \eq (\varphi_1(m)| \varphi_1(n{-}m)$,
hence the requirement that $\varphi$ commutes with $T$ gives 
$\varphi_1(m{+}n) \eq \varphi_1(m) \,{+}\, \varphi_1(n)$. In other words, $\varphi_1$ 
must be an automorphism of the group $\zet/k\zet$. The automorphism group of
$\zet/k\zet$ is $\big(\zet/k\zet\big)^\times$, and it follows that $\varphi$ must 
be a power map $\mathrm{p}_m$ for some $m$. Conversely, every such power map satisfies 
the requirement.
\end{proof}

Let now $\Ng$ be an arbitrary positive integer, and consider the action of $\Sp(2\Ng,\zet)$
on $(\zet/k\zet)^{2\Ng}$ with elements $\big( (m_1|n_1),(m_2|n_2),...\,,(m_\Ng|n_\Ng)\big)$. 
It follows again from Theorem \ref{thm:gTupleOrbits} that there is a fixed point 
$(\bf{0}|\bf{0})$, and in addition one orbit for each divisor $d$ of $k$ with $d \,{<}\, k$, 
containing the element $\big( (d|0),(0|0),...\,,(0|0)\big)$. The latter orbit contains all 
elements for which $\mathrm{gcd}(m_1,n_1,m_2,n_2,...\,,m_\Ng,n_\Ng,k) \eq d$. It remains to 
describe the possible quotients of these symplectic orbits, so again we need to 
investigate automorphisms of the $\Sp(2\Ng,\zet)$-set $\big(\zet/k\zet\big){}_{}^{2\Ng}$. 
Now the automorphisms which may result in the identification 
of elements of the type $\big( (g^{k_1}|g^{l_1}),...\,,(g^{k_\Ng}|g^{l_\Ng})\big)$ and 
$\big( (g^{m_1}|g^{n_1}),...\,,(g^{m_\Ng}|g^{n_\Ng})\big)$ are implemented by conjugation. 
In particular, for each $i \iN \{1,2,...\,,\Ng\}$ the exponent $m_i$ can only depend on $k_i$, 
but not on any other exponent $k_j$ nor on any $l_j$, and likewise $n_i$ can only depend on $l_i$.
Thus such an automorphism $\varphi$ acts as $k_i \,{\mapsto}\, m_i \eq \varphi_i(k_i)$
and $l_i \,{\mapsto}\, n_i \eq \psi_i(l_i)$. More generally we can state

\begin{lemma}\label{lemma:2NSymplecticAut}
A bijection $\varphi$ of the set $\big(\zet/k\zet\big)^{2\Ng}$ of the form
  \be
  \label{varphi-varphi-psi}
  \begin{array}{l}
  \varphi:\quad \big( (m_1|n_1),(m_2|n_2),...\,,(m_\Ng|n_\Ng)\big)
  \Nxl3 \qquad
  \longmapsto\, \big( (\varphi_1(m_1)|\psi_1(n_1)),
  (\varphi_2(m_2)|\psi_2(n_2)),...\,,(\varphi_\Ng(m_\Ng)|\psi_\Ng(n_\Ng))\big)
  \end{array}
  \ee
commutes with the $\Sp(2\Ng,\zet)$-action iff it acts on each pair
$(m_i|n_i)$ as one and the same power map $\mathrm{p}_m$. 
\end{lemma}

\begin{proof}
Suppose $\varphi$ is an automorphism of the form \eqref{varphi-varphi-psi}. We demand 
that $\varphi$ commutes with all generators of $\Sp(2\Ng,\zet)$ as given in \cite{burk}.
Equivariance with respect to the generator that acts as $(m_i|n_i) \,{\mapsto}\, (m_i{+}n_i|n_i)$ 
gives $\varphi_i(m_i+n_i) \eq \varphi_i(m_i)+\psi_i(n_i)$, while equivariance with respect to
$(m_i|n_i) \,{\mapsto}\, (n_i|{-}m_i)$ yields $\psi_i(n_i) \eq \varphi_i(n_i)$ and
$\varphi_i(-m_i) \eq {-}\varphi_i(m_i)$. These equalities must hold for all $i \eq 1,2,...,\Ng$ 
and for all $m_i$ and $n_i$. It follows that $\psi_i \eq \varphi_i$ for every $i$, and that
each $\varphi_i$ is a group automorphism of $\zet/k\zet$. This implies, in turn, that the 
action of $\varphi$ on a pair $(m_i|n_i)$ is implemented by a power map $\mathrm{p}_{l_i}$ 
for some $l_i\iN \big(\zet/k\zet\big)^\times$. Imposing that $\varphi$ also commutes with all 
transpositions $\big( (m_i|n_i),(m_{i+1}|n_{i+1})\big)  \,{\mapsto}\, 
\big( (m_{i+1}|n_{i+1}), (m_i|n_i)\big)$ then implies that $\varphi_i \eq \varphi_j$ 
for all $i,j \eq 1,2,...,\Ng$. It is easily verified that $\varphi$ then
commutes with all generators of the type $\big( (m_i|n_i),(m_{i+1}|n_{i+1})\big) \,{\mapsto}\, 
\big( (m_i{-}n_{i+1}|n_{i}), (m_{i+1}{-}n_i|n_{i+1})\big)$ as well.
\end{proof}

Lemma \ref{lemma:2NSymplecticAut} enables us to describe any (quotient of a) symplectic 
orbit that may occur in the mapping class group action on $2\Ng$-conjugacy classes. 
In order to illustrate this point, we present the \slz-orbits on $\big(\zet/k\zet\big)^2$ 
for $k \eq 3,4,5$ together with their quotients. The following figures show
the sets $\big(\zet/k\zet\big)^2$ for $k \eq 2, 3, 4, 5$, together with the 
actions of the generators $T$ (displayed as \red\ lines) and $S$ (\blue\ lines) of \slz.
In all cases there is the fixed point $(0|0)$ and the orbit containing $(1|0)$, and for
$k \,{=}\, 4$ there is a further length-3 orbit containing $(2|0)$. 
     
     %%%%%%%%%%%%%%%%%%%%%%%%%%%%%%%%%%%%%%%%%%%%%%%%%%%%%%%%%%%%

  \negspace{35}

\be
  \diagZtwoZtwoTwoorbits {1} {.8} {-2} {.8}
  \label{fig:Z2orbit}
\ee

  \negspace{31}

\be
  \diagZthreeZthreeTwoorbits {1} {.8} {-2} {1.6} 
\ee

  \negspace{39}

\be
  \diagZfourZfourThreeorbits {1} {.9} {-2} {2.7} 
\ee

  \negspace{39}

\be 
  \hspace*{-1.5em}
  \diagZfiveZfiveTwoorbits {1} {.9} {-1.9} {3.9} 
\ee

We illustrate the possible quotients in the cases $k \eq 3,4,5$. 
For $k \eq 3$ we need to mod out by $\mathrm{p}_2$, since
$\big(\zet/3\zet\big)^\times {=}\, \{1,2\}$. The $\mathrm{p}_2$-action has one fixed point $(0|0)$ and four length-2 orbits.
Choosing representatives of these orbits, the quotient looks a sfollows:

     \negspace{39}

\be
  \diagZthreeZthreeModPtwo {1} {.8} {-2} {.8}
\ee

Likewise, as the only non-trivial element in $\big(\zet/4\zet\big)^\times$ is $3$, we must mod out the action of $\mathrm{p}_3$, 
   which has four fixed points $(0|0)$, $(0|2)$, $(2|0)$ and $(2|2)$ and six length-2 orbits.
The \slz-action on the corresponding quotient looks as follows:

     \negspace{37}

\be
  \diagZfourZfourModPthree {1} {.8} {-2} {1.6}
  \label{eq:Z4quotorbits}
\ee

The group $\big(\zet/5\zet\big)^\times$ has three non-trivial elements, $2$, $3$ and $4$;
the equivalence classes under $\mathrm{p}_2$ and $\mathrm{p}_3$ coincide. The $\mathrm{p}_2$-action on $(\zet/5\zet)^2$ has $(0|0)$ as single fixed point
   and six length-4 orbits, and the $\mathrm{p}_4$-action amounts to split each of the latter 
   further into two length-2 orbits.
   The quotients can be illustrated as

     \negspace{39}

\be
\begin{array}{l}
  \diagZfiveZfiveModPtwo  {1.1} {.8} {-1.8} {.8}
  \\[-7pt]
  \diagZfiveZfiveModPfour {1} {.8} {-1.8} {1.6}
\end{array}
\label{fig:Z5modP4}
\ee

In the sequel we will focus on non-symplectic orbits.

%%%%%%%%%%%%%%%%%%%%%%%%%%%%%%%%%%%%%%%%%%%%%%%%%%%%%%%%%%%%%%%%

\subsection{Dihedral groups \boldmath{$D_n$}, genus 1}\label{sec:Dn}

\paragraph{Basic data.}

The dihedral group $D_n$, for $n\iN \mathbb{N}$, has the presentation
  \be
  D_n =\langle x,y \,|\, x^n \eq e \eq y^2,\, yxy \eq x\inv \rangle \,.
  \ee
The set $\{y^{a}x^p\,|\, a \eq 0,1,\, p \eq 0,1,...\,,n{-}1\}$ exhausts the 
elements of $D_n$, so that $|D_n| \eq 2n$. We assume $n \,{>}\, 2$, 
since otherwise $D_n$ is abelian; also note the isomorphism $D_3 \,{\cong}\, S_3$.

$D_n$ has a normal abelian subgroup $\langle x\rangle \,{\cong}\, \zet/n \zet$; 
for odd $n$ the subgroups of $\langle x\rangle$ are the only normal subgroups. 
In $D_{2p}$ in addition the subgroup $\langle x^2,y\rangle \,{\cong}\, D_p$ is normal. 
Some additional pertinent properties of $D_n$ are given in the following list.
\\[7pt]
{\boldmath{$n \eq 2p+1$}:}
\\[7pt]
\begin{tabular}{ll}
        Exponent:           &  $2n = 4p+2$
\nxl3
        Center:	            &  $Z(D_n) = \{ e\}$
\nxl3
        Conjugacy classes:~ &  $c_e = \{e\}$\,,
\nxl2
             &  $c_{x^k} = \{x^k,x^{-k}\}\,\ (k\iN\{1,2,...\,,p\}) \,, \quad
                c_y = \{y,xy,...\,,x^{n-1}y\}$
\nxl2
        &  $|C_{D_n}| =	p+2=\frac12\,(n+3)$
\nxl3
        &  $C'_{D_n} =	\{c_{x^k} \,|\, k\eq 0,1,...\,,p\}$
\nxl3
        Centralizers:	&
	$ Z_{x^k} = \langle x\rangle$ ~for $k \eq 1,2,...\,,p\,, \quad
          Z_y =	\langle y\rangle$\,.
\end{tabular}
%%%%%%%%%%%%%%%%%%%%%
         ~\nxl2

\noindent 
{\boldmath{$n \eq 2p$}:}
\\[-3pt]~
\begin{tabular}{ll}
        Exponent:    &   $ \!\begin{cases}
                         \, n=2p  & \text{for } p \text{ even}\\
                         \, 2n=4p & \text{for } p \text{ odd}
                         \end{cases} $
\nxl7
        Center:	     &  $Z(D_n) = \langle x^p\rangle\cong \zet/2\zet$
\nxl3
        Conjugacy classes:~ &  $ c_e =	\{e\} $\,,
\nxl2
                & $ c_{x^k} = \{x^k,x^{-k}\}\,\ (k\iN\{1,2,...\,,p{-}1\})\,, \qquad
                  c_{x^p} = \{x^p\} $\,,
\nxl2
                & $ c_{y}	= \{y,x^2y,...,x^{n-2}y\}\,, \quad~~
                  c_{xy} = \{xy,x^3y,...\,,x^{n-1}y\} $
\nxl3
                & $ |C_{D_n} \,|\, = p+3 = \frac12\,(n+6) $
\nxl2
                & $ C'_{D_n} = \begin{cases}
                               \, \{c_{e},c_{x^2},..\,.c_{x^{p-1}}\} & \text{for } p\text{ odd}\\
                               \, \{c_e,c_{x^2},...\,,c_{x^p}\}  & \text{for } p\text{ even}
                               \end{cases} $
\nxl7
Centralizers:	& $ Z_{x^k} = \langle x\rangle \quad{\rm for}~ k \eq 1,2,...\,,p{-}1 $
\nxl2
            & $ Z_{y}   = \langle x^p,y\rangle \,\cong\, \zet/2\zet\times \zet/2\zet $
\nxl2
            & $ Z_{xy}  = \langle x^p,x^{p+1}y\rangle \,\cong\, \zet/2\zet\times\zet/2\zet $
\end{tabular}

%%%%%%%%%%%%%%%%%%%%%%%%%%%%%%%%%%%%%%%%%%%%%%%%%%%%%%%%%%%%%%%%

\paragraph{\boldmath{$\mcg_1$-orbits on $C^2_{D_n}$ for $n \eq 2p+1$}.}

By the formulas \eqref{eq:NumberTBlocks} and \eqref{eq:Number0ptTBlocks} 
for the number of diconjugacy classes and of diconjugacy classes mapped by $\mu_1$
to $c_e \eq \{e\}$ we have
 $|C^2_{D_n}| \eq 2p^2+5p+4$ and $|\mu_1\inv(c_e)| \eq 2p^2+2p+4$.
The following pairs of elements form a complete set of representatives of diconjugacy classes:
  \be
  \begin{array}{l}
  (x^k|x^l)  \quad \text{for\, $k \eq 0$ and $l \eq 0,1,...\,,p$
             ~~or~~ $k \eq 1,2,...\,,p$ and $l \eq 0,1,...\,,2p$} \,,
  \Nxl3
  (y|x^k)\,,~~ (x^k|y)\,,~~ (y|x^ky) \quad \text{for } \,k \eq 0,1,...\,,p \,. 
  \eear
  \label{Dn-diconjcl}\ee
The number of such pairs, and thus of diconjugacy classes, is 
$ \frac12\,(n^2\,{+}\,1) \eq 2p^2\,{+}\,2p\,{+}\,1$ for the pairs of type $(x^k|x^l)$,
and $ \frac12\,(n\,{+}\,1) \eq p\,{+}\,1$ for the other three types.
By abuse of notation, we use these pairs to denote the diconjugacy classes 
they represent. Then
  \be
  \mu_1\inv(c_e) = \{ (x^k|x^l),(y|e),(e|y),(y|y) \}
  \ee
with $k$ and $l$ as in the first line of the list \eqref{Dn-diconjcl} and
  \be
  \mu_1\inv(c_{x^{2k}}) = \{ (y|x^k),(x^k|y),(y|x^ky) \}
  \quad {\rm for}~ k \eq 1,2,...\,,p \,.
  \ee

Note that $\mu_1\inv(c_e)$ consists of all those diconjugacy classes that are of the 
form analyzed in Theorem \ref{thm:gTupleOrbits}. It follows that $\mu_1\inv(c_e)$ splits 
into symplectic orbits, or possibly quotients of symplectic orbits. Because of $y^2 \eq e$ 
the set $\{(y|e),(e|y),(y|y)\}$ can be identified with the non-trivial symplectic orbit of 
$\big(\zet/2\zet\big){}_{}^2$. The element $x$ has order $2p{+}1 \,{=}\, n$, so the remaining 
orbits can be described starting from the symplectic orbits of $(\zet/n\zet)^2$. 
Apart from the trivial fixed point, there is one such orbit for each divisor $d$ of $n$
with $d \,{<}\, n$. There are thus $\divfn(n)$ symplectic orbits, with $\divfn(n)$ 
the number of divisors of $n$. As noted above, the corresponding diconjugacy classes contain 
two elements, since $yx^ky \eq x^{-k}$. The map that takes $(x^k|x^l)$ to $(x^{-k}|x^{-l})$ 
coincides with the power map corresponding to $n \,{-}\, 1\iN \big(\zet/n\zet\big)^\times$. 
If $k$ divides $n$, write $n \eq rk$, so that $\mathrm{gcd}(n,n{-}k) \eq k$. 
It follows that dividing out the relevant automorphism does not change the number of orbits. Thus the 
set of diconjugacy classes represented by pairs of the form $(x^k|x^l)$ consists of 
$\divfn(n)$ orbits. These are the quotients of symplectic orbits by the 
automorphism defined by the power map corresponding to $2p \eq n \,{-}\, 1$. 
The set $\mu_1\inv(c_e)$ thus consists of $\divfn(n)\,{+}\,1$ orbits.

\medskip

The remaining orbits are quickly determined. Each of the sets $\mu_1\inv(c_{x^k})$
with $k=1,2,...\,,p$, is $\mcg_1$-invariant, and its structure is given by 
 \be
  \bearl
  T:\quad (y|x^k) \mapsto (y|x^ky) \mapsto (y|x^k) \,, \quad (x^k|y) \mapsto (x^k|y) \,,
  \Nxl2
  S:\quad (y|x^k) \mapsto (x^k|y) \mapsto (y|x^k)  \,, \quad (y|x^ky) \mapsto (y|x^ky) \,.
  \eear
  \ee 
To summarize, $C^2_{D_n}$ consists of $\divfn(n)+p+1$ $\mcg_1$-orbits when $n=2p+1$.

%%%%%%%%%%%%%%%%%%%%%%%%%%%%%%%%%%%%%%%%%%%%%%%%%%%%%%%%%%%%%%%%

\paragraph{\boldmath{$\mcg_1$-orbits on $C^2_{D_n}$ for $n \eq 2p$}.}

Equations \eqref{eq:NumberTBlocks} and \eqref{eq:Number0ptTBlocks} give
 $|C^2_{D_n}| \eq 2p^2\,{+}\,6p\,{+}\,8$ and $|\mu_1\inv(c_e)| \eq 2p^2\,{+}\,14 $.
We again denote diconjugacy classes by chosen representatives. Then we get the types
  \be
  \begin{array}{l}
  (x^k|x^l)  \quad \text{for\, $k \eq p,0$ and $l \eq 0,1,...\,,p$}
  \\[3pt] \hspace*{10.8em}
             \text{or~~ $k \eq 1,2,...\,,p{-}1$ and $l \eq 0,1,...\,,2p{-}1$} \,,
  \\[6pt]
  (y|x^k)\,,~~ (x^k|y)\,,~~ (xy|x^k)\,,~~ (x^k|xy)\,,~~ (y|x^ky)\,,~~ (xy|x^{k+1}y)
  \\[3pt] \hspace*{20.6em}
  \quad \text{for } \,k \eq 0,1,...\,,p \,.
  \eear
  \ee
The numbers of these are $\frac12\,(n^2+4) \eq 2p^2+2$ for type $(x^k|x^l)$ and
$\frac12\,(n+2) \eq p+1$ for each of the other types. Further, these types are distributed as
  \be
  \mu_1\inv(c_e) = \{(x^k|x^l),(y|x^{\epsilon p}),(x^{\epsilon p}|y), (x^{\epsilon p}|xy),
  (xy|x^{\epsilon p}), (y|x^{\epsilon p}y),(xy|x^{\epsilon p+1}y)\} 
  \ee
with $\epsilon \eq 0,1$ and
  \be
  \mu_1\inv(c_{x^{2k}}) = \{((y|x^k),(x^k|y),(xy|x^k),(x^k|xy),(y|x^ky),(xy|x^{k+1}y)\}
  \ee
with $k \eq 1,2,...\,,p{-}1$.

Again there are many symplectic orbits, all of which are contained in $\mu_1\inv(c_e)$. 
These are the orbits containing the diconjugacy classes of the pairs 
of type $(x^k|x^l)$, $(y|e)$, $(e|y)$, $(y|y)$, $(xy|e)$, $(e|xy)$, and $(xy|xy)$. The 
latter six diconjugacy classes form two orbits, each equivalent to the non-trivial
orbit in $\big(\zet/2\zet\big){}_{}^2$. The classes of type $(x^k|x^l)$
constitute the orbits of the quotient $\big( \zet/n\zet\big){}_{}^2/\mathrm{p}_{n-1}$. 
As in the case of odd $n$ there are $\divfn(n)$ such orbits. In 
total we have thus obtained $\divfn(n)\,{+}\,2$ (quotients of) symplectic orbits.

\medskip

The remaining orbits are easily determined by hand. If $p$ is odd, then the 
remaining part of $\mu_1\inv(c_e)$ forms a single orbit of the form
\be
  \diagDevenPoddMuOneInvCe {1.3cm}
\ee
If instead $p$ is even, then the diconjugacy classes represented by $(x^p|y)$, $(y|x^p)$ 
and $(y|x^py)$ form an invariant subset, by Proposition \ref{prop:CenterSubgroupInv}\,(ii). 
Hence for even $p$ the remaining part of $\mu_1\inv(c_e)$ splits into two orbits:

     \negspace{41}
\be
  \diagDevenPevenMuOneInvCe {1.9} {1.1}
\ee

Proposition \ref{prop:CenterSubgroupInv}\,(ii) also implies that the orbits in 
$\mu_1\inv(c_{x^{2k}})$ with $k \eq 1,2,...\,,p\,{-}\,1$ follow a similar pattern. We have
\be
  \diagDevenPevenMuOneInvCxKodd {1.3cm} {-3.6} {0}
  \hspace*{2.5em}
\ee
and

     \negspace{47}
\be
  \diagDevenPevenMuOneInvCxkk {1.9} {1.1}
\ee

To summarize the $n\,{=}\,2p$ situation: The set $C^2_{D_n}$ of diconjugacy classes 
splits into $\divfn(n)\,{+}\,3\,\frac{p-1}{2}\,{+}\,3$ orbits if $p$ is odd, and
into $\divfn(n)\,{+}\,2\,{+}\,\frac{3p}{2}$ orbits if $p$ even is even.

%%%%%%%%%%%%%%%%%%%%%%%%%%%%%%%%%%%%%%%%%%%%%%%%%%%%%%%%%%%%%%%%

\subsection{Generalized quaternion groups \boldmath{$Q_{4p}$}, genus 1}\label{sec:Q4p}

The generalized quaternion group $Q_{4p}$ has the presentation
  \be
  Q_{4p} = \langle x,y \,|\, x^{2p} \eq e,\, x^p \eq y^2,\, y\inv xy \eq x\inv\rangle \,.
  \ee
The order of $Q_{4p}$ is $4p$, and we may enumerate the elements as 
  \be
  y^{2\epsilon}x^ky^a \quad \mathrm{with}\quad \epsilon,a\iN \{0,1\},\ k\iN \{0,...\,,,p{-}1\} \,,
  \ee
where $x^{-k} \eq y^2x^{p-k}$ and  $(x^ky)\inv \eq y^2x^ky \,$. The exponent of
$Q_{4p}$ is $2p$ for $p\iN 2\zet$ and $4p$ for $p\iN 2\zet{+}1$,
and the center is $Z(Q_{4p}) \eq \langle y^2\rangle \,{\cong}\, \zet/2\zet$.
The conjugacy classes are
  \be
  \bearll
  c_e = \{e\}\,, &
  c_{x^k} = \{x^k,x^{-k}\} ~~\mathrm{for}~~ k\iN \{1,2,...\,,p{-}1\} \,,
  \\[6pt]
  c_{y^2} = \{y^2\} \,, \quad & c_y = \{x^{2k}y \,|\, k \eq 0,1,...\,,p{-}1 \} \,,
  \\[6pt]
  & c_{xy} = \{x^{2k+1}y \,|\, k \eq 0,1,...\,,p{-}1 \} \,,
  \eear
  \ee
so that
  \be
  |C_{Q_{4p}}| = p+3  \qquand
  C'_{Q_{4p}} = \left\{ \bearll
  \{ c_e,c_{x^2},...\,, c_{x^{p-1}}\} & \text{for } p \text{ odd} \,,
  \\[6pt]
  \{ c_e,c_{x^2},...\,, c_{x^{p-2}}, c_{y^2}\}     & \text{for } p \text{ even} \,.
  \eear \right.
  \ee
The centralizers are
  \be
  \bearl
  Z_{x^k} = \langle x\rangle\cong \zet/2p\zet ~~\mathrm{for}~~ k\iN \{1,2,...\,,p{-}1\} \,,
  \\[6pt]
  Z_y = \langle y\rangle \cong \zet/4\zet
  \qquand Z_{xy} = \langle xy\rangle \cong \zet/4\zet \,.
  \eear
  \ee
Equations \eqref{eq:NumberTBlocks} and \eqref{eq:Number0ptTBlocks} give
  $ |C^2_{Q_{4p}}| \eq 2p^2+6p+8$ and $ |\mu_1\inv(c_e)| \eq 2p^2+14$.
The subgroup $\langle x\rangle \,{\cong}\,\zet/2p\zet$ is normal, as is the subgroup 
generated by $y$ and $x^2$. The similarity with $D_{2p}$ should not come as a
surprise, as $Q_{4p}$ and $D_{2p}$ have the same character table.

\medskip
Labeling diconjugacy classes by representatives, the following list exhausts $C^2_{Q_{4p}}$
(the number in brackets gives the total number of classes of the respective type).
     $$
  \bearll
  (x^k|x^l)    \quad \mathrm{for}~ k,l\iN\{0,...\,,p{-}1\},\ (k,l)\neq (0,0)   & (p^2-1)
  \\[6pt]
  (x^k|y^2x^l) \quad \mathrm{for}~ k\iN \{1,...\,,p{-}1\},~ l\iN\{0,...\,,p{-}1\} & (p^2\,{-}\,p)
  \\[6pt]
  (y^2|x^k)    \quad \mathrm{for}~ k\iN \{1,...\,,p{-}1\}    & (p-1)
  \\[6pt]
  (e|e),~~(e|y^2),~~(y^2|e),~~ (y^2|y^2)	                        & (4)
  \\[6pt]
  (e|y),~~(y|e),~~(y|y),~~(y|y^2),~~(y^2|y),~~(y|y^3)			& (6)
  \\[6pt]
  (e|xy),~~(xy|e),~~(xy|xy),~~(xy|y^2),~~(y^2|xy),~~(xy,y^2xy)	& (6)
  \eear \hspace*{2.23em}
     $$
     ~\\[-26pt]
     \be
     \bearll
     \phantom{(e|xy),~~(xy|e),~~(xy|xy),~~(xy|y^2),~~(y^2|xy),~~(xy,y^2xy)} & \\[-6pt]
  (x^k|y)      \quad \mathrm{for}~ k\iN\{1,...\,,p{-}1\}     & (p-1)
  \\[6pt]
  (y|x^k)      \quad \mathrm{for}~ k\iN\{1,...\,,p{-}1\}     & (p-1)
  \\[6pt]
  (x^k|xy)     \quad \mathrm{for}~ k\iN\{1,...\,,p{-}1\}     & (p-1)
  \\[6pt]
  (xy|x^k)     \quad \mathrm{for}~ k\iN\{1,...\,,p{-}1\}     & (p-1)
  \\[6pt]
  (y|x^ky)     \quad \mathrm{for}~ k\iN\{1,...\,,p{-}1\}     & (p-1)
  \\[6pt]
  (xy|x^kxy)   \quad \mathrm{for}~ k\iN\{1,...\,,p{-}1\}     & (p\,{-}\,1)
  \eear
  \ee
The first six rows of this list exhaust $\mu_1\inv(c_e)$. The diconjugacy classes 
contained in rows $1$\,--\,$4$ make up a quotient of the symplectic orbits on 
$(\zet/2p\zet)^2$, namely the quotient by the power map 
corresponding to $2p{-}1\iN(\zet/2p\zet)^\times$, which implements the 
conjugacy relating $(x^k|x^l)$ and $(x^{-k}|x^{-l})$. There are always the orbits 
containing $(y^2|e)$ (of length three), corresponding to the divisor $2$ of $2p$,
and $(x|e)$, corresponding to the divisor $p$ of $2p$.
In total there are of course $\divfn(2p)$ orbits residing in the first four rows. 
Rows $5$ and $6$ are covered by Theorem \ref{thm:gTupleOrbits} as well; they correspond
to powers of $y$ and of $xy$, respectively. Because of $(xy)^2 \eq y^2$,
potentially there are identifications between these orbits. Indeed,
for $p$ even (when $y$ is conjugate to $y\inv$) both row $5$ and row $6$ 
separately constitute the length $6$-orbit of the quotient $(\zet/4\zet)/p_3$ 
shown in \eqref{eq:Z4quotorbits}, while for $p$ odd (when $y\inv$ is conjugate to $xy$) 
the elements in these two rows form a single length-12 orbit as follows:

     \negspace{18}
  \be
  \diagQppppPodddiconj {2.1} {1.0} {1.5} {0.3}
  \label{gluedmaporbit}
  \ee
The set $\mu_1\inv(c_e)$ thus consists of $\divfn(2p){+}1$ orbits for odd $p$, 
and of $\divfn(2p){+}2$ orbits for even $p$.
This concludes the description of the \slz-orbits in $\mu_1\inv(c_e)$. 

Before turning to the remaining $\mcg_1$-orbits, let us see how the rest of 
$C_{Q_{4p}}^2$ fibers over $C_{Q_{4p}}'$. We have
  \be
  \begin{array}{ll}
  \mu_1\inv(c_{x^{2k}}) \!\!& = \big\{ (x^k|y), (x^{p-k}|y), (y|x^k), (y|x^{p-k}),
  (xy|x^k), (xy|x^{p-k}) \,,
  \nxl2
  & ~~~ (x^k|xy), (x^{p-k}|xy), (y|x^ky), (y|x^{p-k}y), (xy|x^{k+1}y), (xy,x^{p-k+1}y)
  \big\}
  \nxl2
  & \hspace*{11em} \text{for } \left\{ \begin{array}{ll}
  k\in \{1,...\,,\frac{p-1}{2}\} & \text{if } p \iN 2 \zet{+}1 \,,
  \nxl1
  k\in\{1,...\,,\frac{p}{2}-1\} & \text{if } p \iN 2 \zet \,,
  \end{array} \right.
  \Nxl3
  \mu_1\inv(c_{y^2}) \!\!& 
  = \big\{ (x^{p/2}|y), (y|x^{p/2}), (xy|x^{p/2}), (x^{p/2}|xy), (y|x^{p/2}y), (xy|x^{p/2+1}y) \big\}
  \nxl2
  & \hspace*{19.4em} \text{if } p \iN 2 \zet \,,
  \end{array}
  \ee
Consider first the case $k \eq 2l{+}1$. The set $\mu_1\inv(c_{x^{2k}})$ then forms 
a single \slz-orbit:
 
                 \negspace{12}
\be
\diagQuaternionsB {2.2} {1.6} {0.8} {7.2} {5.5} {1.5} {0.3}
\ee

\medskip
\noindent
Next consider the case $k \eq 2l$, first with $p$ odd. The set 
$\mu_1\inv(c_{x^{4l}})$ again consists of a single orbit, shown in the following figure:

                 \negspace{12}
\be
\diagQuaternionsC {1.9} {1.5} {-6.5} {1.4} {-6.4} {0.7}
\ee
If instead $p$ is even, then the diconjugacy classes that have representatives in the subgroup 
$H \eq \langle y,x^2\rangle$ form proper invariant subsets of $\mu_1\inv(c_{x^{2k}})$.
By Proposition \ref{prop:CenterSubgroupInv}\,(ii) 
it follows that the set $\mu_1\inv(c_{x^{4l}})$ for $2l \,{\ne}\, p$ consists of two orbits,
as shown in the following figure:

                 \negspace{15}
  \be
  \diagQuaternionsD {1.8} {1.4} {-5.8} {1.5} {-6.06} {0.9} {-4.0} {2.8}
  \ee
The same argument implies that for even $p$ the orbits on $\mu_1\inv(c_{y^2})$ 
depend on $p\bmod 4$, as shown in the following pictures.

                 \negspace{33}
  \be
  \diagQuaternionsE {1.9} {1.1} {-4.7} {-0.5}
  \ee

                 \negspace{27}
  \be
  \diagQuaternionsF {60} {1.3cm} {-3.9} {0}  \hspace*{4em}
  \ee

\medskip
\noindent
To summarize, when $p$ is odd, then the set $C^2_{Q_{4p}}$ consists of $\divfn(2p)+(p{+}3)/2$ 
$\mcg_1$-orbits. When $p \,{\equiv}\, 0\bmod{4}$ there are $\divfn(2p)+(3p+8)/4$ orbits, and 
when $p \,{\equiv}\, 2\bmod{4}$ there are $\divfn(2p)+(3p{+}6)/4$ orbits.

%%%%%%%%%%%%%%%%%%%%%%%%%%%%%%%%%%%%%%%%%%%%%%%%%%%%%%%%%%%%%%%%

\subsection{Genus 2 and 3}

\subsubsection{\boldmath{$S_3$}}

Recall that $S_3$ is isomorphic to $D_3$. We denote the non-trivial elements of $S_3$ as
  \be
  a_1:=(23)\,,\quad a_2:=(13)\,,~\quad a_3:=(12)\,,~\quad b:=(123)\,,~\quad b\inv=(132) \,.
  \ee
Then the three conjugacy classes of $S_3$ are $c_e \eq \{e\}$, $c_a \eq \{a_1,a_2,a_3\}$
and $c_b \eq \{b,b\inv\}$, while $C'_{S_3} \eq \{c_e,c_b\}$. 

\paragraph{\boldmath{$\mcg_2$}-orbits:}

Clearly, $\im(\mu_2) \,{=}\, C'_{S_3}$. Using GAP we obtain
  \be
  |C^4_{S_3}| = 251\,,~\quad |\mu_2\inv(c_e)| = 116\,,~\quad |\mu_2\inv(c_b)| = 135\,.
  \ee
Let us determine the orbits of $\mu_2\inv(c_e)$ and $\mu_2\inv(c_b)$.\,%
 \footnote{~An explicit list of the sets $\mu_2\inv(c_e)$ and $\mu_2\inv(c_b)$
 is given in Tables 2 and 3 of {\tt arXiv:1506.03263v1}.}
The quotient of the mapping class group that acts effectively on $\mu_2\inv(c_e)$ has order
$394\,419\,752\,309\,411\,020\,800$; it is a centerless group, i.e.\ the hyperelliptic 
involution (the only non-trivial central element in $\mcg_2$) is represented trivially.

%%%%%%%%%%%%%%%%%%%%%%%%%%%%%%%%%%%

Clearly $d_1 \eq \big( (e|e),(e|e) \big) \iN \mu_2\inv(c_e)$ is a fixed point,
Further, there are 40 $4$-conjugacy classes $d_2$\,--\,$d_{41}$ which fall into the 
conditions of both Proposition \ref{prop:CenterSubgroupInv}\,(ii) and Theorem 
\ref{thm:gTupleOrbits}, implying that this set is mapping class invariant, and 
15 classes $d_{42}$\,--\,$d_{56}$ that satisfy the conditions of Theorem 
\ref{thm:gTupleOrbits} and thus again form a mapping class invariant subset, as well 
as $60$ further classes $d_{57}$\,--\,$d_{166}$. Let us denote these three subsets 
by indicating their order, as $X_{15}$, $X_{40}$, and $X_{60}$, respectively. 

The set $X_{15}$ consists of a single $\mcg_2$-orbit that coincides with the non-trivial 
symplectic orbit on $(\zet/2\zet)^4$. Inspection reveals that all generators 
have order two, and indeed are
products of disjoint transpositions. The group generated by these permutations, 
a quotient of the image of the quantum representation, has order $720$ and is isomorphic 
to $S_6$. The stabilizer of any given $4$-conjugacy class has order $48$ and is isomorphic
to $\zet/2\zet\Times S_4$. An alternative characterization of the mapping class orbit 
$X_{15}$ is that it is equivalent 
to the orbit of the action of $S_6$ on two-element subsets of $\{1,2,3,4,5,6\}$.

The set $X_{40}$ consists of a single $\mcg_2$-orbit, too, coinciding with the 
non-trivial orbit in the quotient $\big(\zet/3\zet\big)^4/\mathrm{p}_{2,2}$.
The group generated by the action of $T^{(i)}$ and $U^{(i)}$, $i\eq 1,2$, and $L$ 
on $X_{40}$ has order $25\,920$ and is isomorphic to the simple group $O(5,3)$.
The stabilizer of any one of the elements of $X_{40}$ has order $648$.
Also the $\mcg_2$-set $X_{60}$ consists of a single orbit. 
Pertinent data for these orbits are summarized in the following table: 
  \begin{center}
  \begin{tabular}{cccccc}
   Name	    &  Representative               & Length  & $\mcg_2$-quotient   &   Stabilizer &   \Symplectic
   \\[3pt] \hline ~\\[-8pt]
   --       &  $\big( (e|e),(e|e)\big)$	    &   $1$   &   $\{e\}$         &  --          &   yes
   \\[3pt]
   $X_{15}$ &  $\big( (a_1|e),(e|e)\big)$   &	$15$  & $S_6$     & $\zet/2\zet\times S_4$ &   yes
   \\[3pt]
   $X_{40}$ &  $\big( (b|e),(e|e)\big)$	    &   $40$  & $O(5,3)$            &  $648$       &   yes
   \\[3pt]
   $X_{60}$ &  $\big( (a_1|e),(a_2|e)\big)$ &	$60$  & $15\, 216\, 811\, 431\, 690\, 240$ &
                                                             $253\, 613\, 523\, 861\, 504$ &   no
  \end{tabular}
  \end{center}

\smallskip
\noindent
This table lists the following 
data for each orbit: a representative $4$-tuple of one $4$-con\-ju\-gacy class generating the 
orbit; the length of the orbit; information about the quotient of $\mcg_2$ acting 
effectively on the orbit, either in the form of the name of a finite group isomorphic 
to the quotient or as the order of the quotient; and information about the stabilizer of one 
of the elements of the orbit, either as the name of a group isomorphic to the stabilizer 
or as the order of the stabilizer. The final datum indicates whether the $\mcg_2$-action 
on the orbit factors through the symplectic group (yes) or not (no) or, equivalently,
whether the Torelli group $\tor_2$ acts trivially (yes) or not (no).

The Torelli group $\tor_2$ acts non-trivially on $X_{60}$; the action of the element 
$T_c$ from \eqref{eq:genus2Torelliel} is indicated in the following figure:
\begin{center}
\begin{tikzpicture}
[commutative diagrams/every diagram]
  \node	at  (-5,5)  {$X_{60}$, $T_c$:};
  \foreach \x in {0,...,59}
  { \node  (n\x)  at  (6*\x-180:5.1cm)   {\scriptsize$\bullet$}; }
  \foreach \y in {57,...,116}
  { \node  (m\y)  at  (6*\y-162:5.5cm) {\scriptsize$d_{\y}$}; }
  \path[blue,commutative diagrams/.cd, every arrow, every label]
  (n14) edge  (n31)
  (n31) edge  [bend right]  (n28)
  (n28) edge  (n14)
  (n15) edge  (n55)
  (n55) edge  (n33)
  (n33) edge  (n15)
  (n19) edge  (n44)
  (n44) edge  (n32)
  (n32) edge  (n19)
  (n23) edge  (n49)
  (n49) edge  (n36)
  (n36) edge  (n23)
  (n29) edge  [bend left]  (n35)
  (n35) edge  (n51)
  (n51) edge  (n29)
  (n30) edge  (n54)
  (n54) edge  (n42)
  (n42) edge  (n30)
  (n34) edge  (n56)
  (n56) edge  [bend right] (n48)
  (n48) edge  (n34)
  (n40) edge  [bend left] (n43)
  (n43) edge  [bend right] (n41)
  (n41) edge  [bend right](n40)
  (n53) edge  [bend left](n57)
  (n57) edge  [bend left] (n59)
  (n59) edge  [bend right](n53)
  ;
\end{tikzpicture}
\end{center}

\medskip

Finally we consider the set $\mu_2\inv(c_b)$, which has order 135. In this case none of 
the results of Section \ref{sec:mcg2Nconj} 
is applicable, so we cannot draw any immediate conclusions regarding the existence of 
invariant subsets. A GAP calculation shows that $\mu_2\inv(c_b) \,{=:}\, X_{135} $ is a 
single $\mcg_2$-orbit, and the quotient of $\mcg_2$ acting effectively on $X_{135}$ has 
order $102\,662\,358\,058\,319\,877\,641\,002\,337\,100\,103\,680$. This group is 
centerless and non-simple, with derived subgroup of index $2$. The Torelli group 
$\tor_2$ acts non-trivially on $X_{135}$.
  
\smallskip

\paragraph{\boldmath{$\mcg_3$}-orbits:} 
We have $|C^6_{S_3}| \eq 8\,051$, $|\mu_3\inv(c_e)| \eq 2\,948$ and $|\mu_3\inv(c_b)| 
\eq  5\,103$. The $6$-con\-ju\-ga\-cy classes in  $\mu_3\inv(c_e)$ represented by elements 
of the form $\big( (a_1^{k_1}|a_1^{l_1}), (a_1^{k_2}|a_1^{l_2}), (a_1^{k_3}|a_1^{l_3})\big)$ 
make up the two $\mcg_3$-orbits on $( \zet/2\zet)^{6}$, one fixed point and one orbit 
of length $63$. Replacing $a_1$ by $b$ we get instead the non-trivial orbit on 
$(\zet/3\zet)^6/\mathrm{p}_{2,3}$, which is of length $364$. The remaining 
$6$-conjugacy classes in $\mu_3\inv(c_e)$ form a single orbit of length $2\,520$; 
calculating the action on this orbit of the element \eqref{eq:TorelN3} shows that 
the Torelli group $\tor_3$ acts non-trivially. Some features of these $\mcg_3$-orbits 
are summarized in the following table:
   \begin{center}
   \begin{tabular}{ccccc}
   Representative                   &  Length   &   $\mcg_3$-quotient  &  Stabilizer        & \Symplectic
   \\[3pt] \hline ~\\[-8pt]
   $\big((e|e),(e|e),(e|e)\big)$    &  $1$      &  $\{e\}$             &  --                &   yes
   \\[3pt]
   $\big( (a_1|e),(e|e),(e|e)\big)$ &  $63$     &  $1\,451\,520$       &  $23\,040$         &	yes
   \\[3pt]
   $\big( (b|e),(e|e),(e|e)\big)$   &  $364$    &  $4\,585\,351\,680$  &  $\,12\,597\,120$  &   yes
   \\[3pt]
   $\big( e|e),(a_1|e),(e|a_3)\big)$&  $2\,520$ &  $\sim 10^{284}$     &  $\sim 10^{281}$   &   no
   \end{tabular}
   \end{center}
The set $\mu_3\inv(c_b)$ consists of a single $\mcg_3$-orbit:
   \begin{center}
   \begin{tabular}{ccccc}
   Representative                     &   Length  &  $\mcg_3$-quotient  &  Stabilizer       & \Symplectic
   \\[3pt] \hline ~\\[-8pt]
   $\big((a_3|a_1),(e|e),(e|e)\big)$  & $5\, 103$ &  $\sim 10^{420}$    &  $\sim 10^{417}$  &  no
   \end{tabular}
   \end{center}

%%%%%%%%%%%%%%%%%%%%%%%%%%%%%%%%%%%%%%%%%%%%%%%%%%%%%%%%%%%%%%%%

\subsubsection{\boldmath{$D_4$}}

We refer to Section \ref{sec:Dn} for notation; recall that $C'_{D_4} \eq \{ c_e,c_{x^2}\}$
and that $x^2$ is central.

\paragraph{\boldmath{$\mcg_2$}-orbits:}
We have $|C^4_{D_4}| \eq 1\, 216$, distributed over $\mu_2\inv(c_e)$ of order $736$ and 
$\mu_2\inv(c_{x^2})$ of order $480$. Since $x^2$ is central, it follows from 
Theorem \ref{thm:Torelli}\,(ii) that the Torelli group $\tor_2$ acts trivially on 
$C^4_{D_4}$. The set $\mu_2\inv(c_e)$ consists of $11$ $\mcg_2$-orbits. The $4$-con\-ju\-gacy 
classes represented by elements of the form $\big( (x^{k_1}|x^{l_1}),(x^{k_2}|x^{l_2})\big)$ 
describe the three symplectic orbits, of lengths $1$, $15$ and $120$, of the quotient 
$(\zet/4\zet)^4/\mathrm{p}_{3,2}$. The diconjugacy classes of the form 
$\big((y^{k_1}|y^{l_1}),(y^{k_2}|y^{l_2})\big)$ constitute the non-trivial symplectic 
orbit in $(\zet/2\zet)^4$, of length $15$; there is an isomorphic symplectic orbit 
consisting of diconjugacy classes of the form 
$\big(((xy)^{k_1}|(xy)^{l_1}),((xy)^{k_2}|(xy)^{l_2})\big)$. 
In addition there are two orbits of length 45, two of length 60, and two of length 180.
Some properties of the $\mcg_2$-orbits are listed in the following table:
   \begin{center}
   \begin{tabular}{ccccc}
   Representative                &  Length  & $\mcg_2$-quotient &  Stabilizer               & \Symplectic
   \\[3pt] \hline ~\\[-8pt]
   $\big((e|e),(e|e)\big)$       &  $1$	    &  $\{e\}$          &  --                       &  yes
   \\[3pt]
   $\big( (y|e),(e|e)\big)$      &  $15$    &  $S_6$            &  $\zet/2\zet \times S_4$  &  yes
   \\[3pt]
   $\big( (x^2|e),(e|e)\big)$    &  $15$    &  $S_6$            &  $\zet/2\zet \times S_4$  &  yes
   \\[3pt]
   $\big( (xy|e),(e|e)\big)$     &  $15$    &  $S_6$            &  $\zet/2\zet \times S_4$  &  yes
   \\[3pt]
   $\big( (x|e),(e|e)\big)$      &  $120$   &  $368\, 640$      &  $3\, 072$                &  yes
   \\[3pt]
   $\big( (y|e),(x^2|e)\big)$    &  $45$    &  $S_6$            &  $\zet/2\zet\times D_4$   &  yes
   \\[3pt]
   $\big( (x^3y|e),(x|e)\big)$   &  $45$    &  $S_6$            &  $\zet/2\zet\times D_4$   &  yes
   \\[3pt]
   $\big( (y|x^2),(e|e)\big)$    &  $60$    &  $S_6$            &  $D_6$                    &  yes
   \\[3pt]
   $\big( (x^3y|x^2),(e|e)\big)$ &  $60$    &  $S_6$            &  $D_6$                    &  yes
   \\[3pt]
   $\big( (y|e),(x|e)\big)$      &  $180$   &  $368\, 640$      &  $2\, 048$                &  yes
   \\[3pt]
   $\big( (y|x^2),(x|e)\big)$    &  $180$   &  $368\, 640$      &  $2\, 048$                &  yes
   \end{tabular}
   \end{center}
The set $\mu_2\inv(c_{x^2})$ consists of a single $\mcg_2$-orbit: 
\begin{center}
   \begin{tabular}{ccccc}
   Representative             &  Length  &  $\mcg_2$-quotient &  Stabilizer  & Symplectic
   \\[3pt] \hline ~\\[-8pt]
   $\big( (y|x),(e|e)\big)$   &  $480$	 &  $11\, 796\, 480$  &  $24\, 576$  &  yes
   \end{tabular}
\end{center}

\paragraph{\boldmath{$\mcg_3$}-orbits:} 
We have $|C^6_{D_4}| \eq 68\,608$, $|\mu_3\inv(c_e)| \eq 36\,352$ and $|\mu_3\inv(c_{x^2})| \eq 32\,256$. 
The set $\mu_3\inv(c_e)$ consists of $11$ $\mcg_3$-orbits. The $6$-conjugacy classes represented 
by tuples of the form $\big( (x^{k_1}|x^{l_1}),(x^{k_2}|x^{l_2}),(x^{k_3}|x^{l_3})\big)$ split 
into three orbits: the trivial fixed point, an orbit of length $63$ equivalent to the 
non-trivial orbit on $(\zet/2\zet)^6$, and an orbit of length $2016$ generated 
from $\big( (x|e),(e|e),(e|e)\big)$. The two orbits generated from $\big( (y|e),(e|e),(e|e)\big)$ 
and $\big( (xy|e),(e|e),(e|e)\big)$ are equivalent to the non-trivial orbit of 
$( \zet/2\zet)^6$, too. There are in addition two orbits 
of length $945$, two of length $1\, 008$, and two of length $15\,120$.
The quotient acting effectively on the orbit is isomorphic to $O(7,2)$ in all
cases except for three, in the latter cases the group is perfect but not simple.
Some pertinent data are collected in the following table:
\begin{center}
   \begin{tabular}{ccccc}
   Representative                   &  Length     &  $\mcg_3$-quotient          &  Stabilizer  & \Symplectic
   \\[3pt] \hline ~\\[-8pt]
   $\big((e|e),(e|e),(e|e)\big)$    &  $1$        &  $\{e\}$			&  --		&   yes		
   \\[3pt]
   $\big( (x^2|e),(e|e),(e|e)\big)$ &  $63$       &  $O(7,2)$			&  $23\, 040$	&   yes
   \\[3pt]
   $\big( (y|e),(e|e),(e|e)\big)$   &  $63$       &  $O(7,2)$			&  $23\, 040$	&   yes
   \\[3pt]
   $\big( (xy|e),(e|e),(e|e)\big)$  &  $63$       &  $O(7,2)$			&  $23\, 040$	&   yes
   \\[3pt]
   $\big( (x|e),(e|e),(e|e)\big)$   &  $2\, 016$  & $1\, 522\, 029\, 035\, 520$& $754\, 974\, 720$ & yes
   \\[3pt]
   $\big( (y|e),(x^2|e),(e|e)\big)$ &  $945$	  &  $O(7,2)$			&  $1\, 536$	&   yes
   \\[3pt]
   $\big( (xy|e),(x^2|e),(e|e)\big)$&  $945$	  &  $O(7,2)$			&  $1\, 536$	&   yes
   \\[3pt]
   $\big( (y|x^2),(e|e),(e|e)\big)$ &  $1\, 008$  &  $O(7,2)$			&  $1\, 440$	&   yes
   \\[3pt]
   $\big( (xy|x^2),(e|e),(e|e)\big)$&  $1\, 008$  &  $O(7,2)$			&  $1\, 440$	&   yes
   \\[3pt]
   $\big( (y|e),(x|e),(e|e)\big)$   & $15\, 120$
                                         & $49\, 873\, 847\, 435\, 919\, 360$ & $3\, 298\, 534\, 883\, 328$ & no
   \\[3pt]
   $\big( (y|x^2),(x|e),(e|e)\big)$
                            & $15\, 120$ & $49\, 873\, 847\, 435\, 919\, 360$ & $3\, 298\, 534\, 883\, 328$ & no
   \\[3pt]
   \end{tabular}
\end{center}
The set $\mu_3\inv(c_{x^2})$ is a single $\mcg_3$-orbit. The quotient acting effectively
on it is a non-sim\-ple perfect group with trivial center. The stabilizer of any element
in this orbit is a non-perfect group with center 
$(\zet/2\zet){}^{\times 5}$ and with a derived subgroup of index $16$. 
A few more details are given in the following table:
\begin{center}
   \begin{tabular}{ccccc}
   Representative                 &  Length    &  $\mcg_3$-quotient                  &  Stabilizer    & \Symplectic
   \\[3pt] \hline ~\\[-8pt]
   $\big((y|x),(e|e),(e|e)\big)$ & $32\, 256$ & $3\,191\,926\,235\,898\, 839\, 040$ & $98\,956\,046\,499\,840$ & no
   \end{tabular}
\end{center}

%%%%%%%%%%%%%%%%%%%%%%%%%%%%%%%%%%%%%%%%%%%%%%%%%%%%%%%%%%%%%%%%

\subsubsection{\boldmath{$Q_8$}}

For pertinent notation see section \ref{sec:Q4p}; recall that $C'_{Q_8} \eq \{ c_e,c_{y^2}\}$ 
and that $y^2$ is central.

\paragraph{\boldmath{$\mcg_2$}-orbits:}
We have $\im(\mu_2) \eq C'_{Q_8}$, and $|C^4_{Q_8}| \eq 1\, 216$, which is distributed over $C'_{Q_8}$ 
as $|\mu_2\inv(c_e)| \eq 736$ and $|\mu_2\inv(c_{y^2})| \eq 480$. Again it follows from 
Theorem \ref{thm:Torelli}\,(ii) that the Torelli group $\tor_2$ acts trivially on $C^4_{Q_8}$.

The set $\mu_2\inv(c_e)$ consists of $6$ $\mcg_2$-orbits. The $4$-conjugacy classes 
represented by elements of the form $\big( (x^{k_1}|x^{l_1}),(x^{k_2}|x^{l_2})\big)$ 
split into the three orbits of lengths $1$, $15$, and $120$ corresponding to the quotient 
$(\zet/4\zet)^4/\mathrm{p}_{3,2}$. The $4$-conjugacy classes represented 
by elements $\big( (y^{k_1}|y^{l_1}),(y^{k_2}|y^{l_2})\big)$ with 
$\mathrm{gcd}(k_1,l_1,k_2,l_2,4) \eq 1$ describe another orbit of length $120$ inside
$(\zet/4\zet)^4/\mathrm{p}_{3,2}$, and yet another orbit of the same type is obtained 
from the latter by replacing $y$ with $xy$ in the representatives. 
The three orbits of length $120$ are isomorphic. Finally there is an 
orbit of length $360$, generated from the $4$-conjugacy class represented by 
$\big( (x|e),(y|e)\big)$. Some properties of these and
of $\mu_2\inv(c_{y^2}) \,{\subset}\, C^4_{Q_8}$ are given in the next two tables:
\begin{center}
   \begin{tabular}{ccccc}
   Representative              &  Length & $\mcg_2$-quotient &  Stabilizer  &  \Symplectic
   \\[3pt] \hline ~\\[-8pt]
   $\big((e|e),(e|e)\big)$     &  $1$    &  $\{e\}$          &  --          &   yes
   \\[3pt]
   $\big( (y^2|e),(e|e)\big)$  &  $15$   &  $S_6$	     &  $\zet/2\zet \times S_4$ & yes
   \\[3pt]
   $\big( (x|e),(e|e)\big)$    &  $120$  &  $368\, 640$      &  $3\, 072$   &	yes
   \\[3pt]
   $\big( (y|e),(e|e)\big)$    &  $120$  &  $368\, 640$      &  $3\, 072$   &	yes
   \\[3pt]
   $\big( (xy|e),(e|e)\big)$   &  $120$  &  $368\, 640$	     &  $3\, 072$   &	yes
   \\[3pt]
   $\big( (x|e),(y|e)\big)$    &  $360$  &  $11\, 520$       &  $5\, 760$   &	yes
   \end{tabular}
   \end{center}
\begin{center}
   \begin{tabular}{ccccc}
   Representative            &  Length  &  $\mcg_2$-quotient  &  Stabilizer  & \Symplectic
   \\[3pt] \hline ~\\[-8pt]
   $\big( (x|y),(e|e)\big)$  &  $480$   &  $11\, 796\, 480$   &  $24\, 576$  &  yes
   \end{tabular}
   \end{center}

\begin{rem}\label{rem:NoPGroup}
A result of \cite{etrw} states that if $G$ is a $p$-group, then the images of the 
quantum representations of $\mcg_{0,n}$ obtained from $\DG$ are $p$-groups as well.
Now $Q_8$ is a $2$-group, while the quotients of $\mcg_2$ acting effectively 
on the orbits of lengths $360$ and $480$ are not. In fact,
as verified by computer, they aren't $p$-groups for any $p$. Thus the result 
of \cite{etrw} does not generalize to higher genus.
\end{rem}

\paragraph{\boldmath{$\mcg_3$}-orbits:} 
Similarly to the case of $D_4$ we have $|C_{Q_8}^6| \eq 68\,608$, 
$|\mu_3\inv(c_e)| \eq 36\,352$ and $|\mu_3\inv(c_{y^2})| \eq 32\,256$. 
The set $\mu_3\inv(c_e)$ is distributed over $6$ $\mcg_3$-orbits. The $6$-conjugacy classes 
represented by tuples of the form $\big( (x^{k_1}|x^{l_1}), (x^{k_2}|x^{l_2}), (x^{k_3}|x^{l_3}) \big)$ 
make up the three orbits of the quotient $(\zet/4\zet)^6/\mathrm{p}_{3,3}$: the 
trivial fixed point, an orbit of length $63$ generated from $\big( (x^2|e),(e|e),(e|e)\big)$, 
and an orbit of length $2\, 016$ generated from $\big( (x|e), (e|e), (e|e)\big)$. Moreover, 
the two tuples $\big( (y|e), (e|e), (e|e)\big)$ and $\big( (xy|e), (e|e), (e|e)\big)$ both 
generate orbits of length $2\, 016$ isomorphic to the first one. Finally the $6$-conjugacy 
class represented by $\big( (x|e),(y|e),(e|e)\big)$ generates an orbit of length $30\,240$. 
The quotients of $\mcg_3$ acting effectively on the non-trivial orbits are perfect, and 
in the case of the length-$63$ orbit even simple; for the length-$30\, 240$ orbit the group 
is \emph{not} a $p$-group for any $p$.
Here is a summary of pertinent data:
\begin{center}
   \begin{tabular}{ccccc}
   Representative                   &  Length    &  $\mcg_3$-quotient           &  Stabilizer        & \Symplectic
   \\[3pt] \hline ~\\[-8pt]
   $\big((e|e),(e|e), (e|e)\big)$   &  $1$       &  $\{e\}$                     &  --                &  yes
   \\[3pt]
   $\big( (y^2|e),(e|e),(e|e)\big)$ &  $63$      &  $O(7,2)$                    &  $23\, 040$        &  yes
   \\[3pt]
   $\big( (x|e),(e|e), (e|e)\big)$  &  $2\, 016$ &  $1\, 522\, 029\, 035\, 520$ &  $754\, 974\, 720$ &  yes
   \\[3pt]
   $\big( (y|e),(e|e), (e|e)\big)$  &  $2\, 016$ &  $1\, 522\, 029\, 035\, 520$ &  $754\, 974\, 720$ &  yes
   \\[3pt]
   $\big( (xy|e),(e|e), (e|e)\big)$ &  $2\, 016$ &  $1\, 522\, 029\, 035\, 520$ &  $754\, 974\, 720$ &  yes
   \\[3pt]
   $\big( (x|e),(y|e), (e|e)\big)$  &  $30\, 240$& $49\, 873\, 847\, 435\, 919\, 360$ &	$1\, 649\, 267\, 441\, 664$	& no
   \end{tabular}
   \end{center}
The set $\mu_3\inv(c_{y^2})$ consists of a single $\mcg_3$-orbit. The quotient of
$\mcg_3$ acting effectively is perfect with trivial center, but not simple.
The stabilizer of any point is not perfect, has center $(\zet/2\zet)^{\times 5}$ and a 
derived subgroup of index $16$. We have:
\begin{center}
   \begin{tabular}{ccccc}
   Representative                   &  Length    &  $\mcg_3$-quotient  &  Stabilizer  & \Symplectic
   \\[3pt] \hline ~\\[-8pt]
   $\big( (x|y), (e|e), (e|e)\big)$ & $32\, 256$ & $3\,191\, 926\, 235\, 898\, 839\, 040$
                                                           & $98\, 956\, 046\, 499\, 840$ &  no
   \end{tabular}
   \end{center}

%%%%%%%%%%%%%%%%%%%%%%%%%%%%%%%%%%%%%%%%%%%%%%%%%%%%%%%%%%%%%%%%

\subsection{\boldmath{$A_5$}}

For the non-abelian simple group $A_5$ of order $60$ we have 
$Z(A_5) \eq \langle e\rangle$, $A_5' \eq A_5^{}$ and $n_{A_5} \eq 30$, and there are
five conjugacy classes $c_e \eq \{e\}$ and $c_a$, $c_b$,  $c_c$,  $c_d$ containing, 
respectively, the group elements
  \be
  a := {(1\,2)(3\,4)}    \,,\quad
  b := {(1\,2\,3)}       \,,\quad
  c := {(1\,2\,3\,4\,5)} \,,\quad
  d := {(1\,2\,3\,5\,4)} \,.
  \ee
Every element of $A_5$ is a commutator, so $C'_{\!A_5} \eq C_{\!A_5}$.
The chosen representatives of the conjugacy classes have centralizers as follows:
  \be
  \begin{array}{rl}
  Z_{(1\,2)(3\,4)}	\!&= \langle (1\,2)(3\,4), (1\,3)(2\,4)\rangle\cong \zet/2\zet\Times\zet/2\zet \,,
  \\[4pt]
  Z_{(1\,2\,3)}	        \!&= \langle (1\,2\,3)\rangle       \cong \zet/3\zet	 \,,
  \\[4pt]
  Z_{(1\,2\,3\,4\,5)}	\!&= \langle (1\,2\,3\,4\,5)\rangle \cong	\zet/5\zet \,,
  \\[4pt]
  Z_{(1\,2\,3\,5\,4)}	\!&= \langle (1\,2\,3\,5\,4)\rangle \cong	\zet/5\zet \,.
  \eear
  \ee
The formulas \eqref{eq:NumberTBlocks} and \eqref{eq:Number0ptTBlocks} result in
$|C^2_{A_5}| \eq 77$ and $|\mu_1\inv(c_e)| \eq 22$.

%%%%%%%%%%%%%%%%%%%%%%%%%%%%%%%%%%%%%%%%%%%%%%%%%%%%%%%%%%%%%%%%

\subsubsection{\boldmath{$\mcg_1$}-orbits}

The set $\mu_1\inv(c_e)$ consists of five $\mcg_1$-orbits.
The pre-images partition $C^2_{A_5}$ according to
\be
  \begin{array}{rl}
  \mu_1\inv(c_e) = \!\! & \big\{ \,  (e|e), (c|e), (c^2|e), (e|c), (e|c^2), (c|c), (c|c^2),
  \\[4pt]
  &  (c|c^3), (c|c^4), (c^2|c^2), (c^2|c^4), (c^2|c^3), (c^2|c),
  \\[4pt]
  &  (b|e), (e|b), (b|b), (b|b^2),(a|e), (e|a), (a|a),
     (a|a^b), (a|{}^ba) \, \big\} \,.
  \end{array}
  \label{eq:A5dicCe}
\ee
Here the first two rows constitute the quotient 
$(\zet/5\zet)^2/\mathrm{p}_4$, as shown in \eqref{fig:Z5modP4}. The four 
diconjugacy classes represented by $(b|e)$, $(e|b)$, $(b|b)$ and $(b|b^2)$ make up 
the non-trivial orbit of the quotient $(\zet/3\zet)^2/\mathrm{p}_2$, illustrated 
in \eqref{fig:Z3modP2}, while the classes represented by $(a|e)$, $(e|a)$ and $(a|a)$
form the non-trivial orbit of $(\zet/2\zet)^2$, depicted in \eqref{fig:Z2orbit}. 
Finally the classes represented by $(a|a^b)$ and $(a|{}^ba)$ form a length-2 orbit 
for which both $S$ and $T$ act as the elementary transposition. 
The set $\mu_1\inv(c)$ for $c \iN \{ c_a, c_b, c_c, c_d \}$ are given by
  \be
  \begin{array}{ll}
  \mu_1\inv(c_a) = \!\!\! &
  \big\{\, (b^2|a), (b^2|b^a), (b^2|b^c), (b|a), (b^2|ac), (b^2|b^d), (a|b^2), (a|b) \,\big\} \,,
  \Nxl3
  \mu_1\inv(c_b) = \!\!\!& \big\{ \, (b|a^c), (a|a^{(c^2)}), (a^c|b), (b|b^{(c^2)}),
  (b|c), (b^2|d), (b|(b^2)^{(c^2)}), (b|d),
  \\[4pt]
  &~~ (b^2|c), (a|c^2), (a|c^d), (c|b), (c|b^2), (c|a^{(b^2)}), (c|d), (c^4|d),
  \\[4pt]
  &~~ (d|b^2), (d|b), (d|a^{(b^2)}), (d|c), (d|c^4) \, \big\} \,,
  \Nxl3
  \mu_1\inv(c_c) = \!\!\!& \big\{ \, (a|d), (d|a), (a|ad), (b|bd^2), (b|d^2), (b|d^3), (a|c), 
  \\[4pt]
  &~~  (a|bcb), (c|a), (c|[c,b]), (c|[b,c]), (c|d^2), (c|d^3) \, \big\} \,,
  \Nxl3
  \mu_1\inv(c_d) = \!\!\!& \big\{ \, (a|ad^2), (a|d^2), (d^2|a), (b|bc^2), (b|c^3), (b|c^2), (a|cb), 
  \\[4pt]
  &~~ (a|b^2c), (bc^3|a), (c^3|b), (c^2|b), (d|c^2), (d|c^3) \, \big\} \,.
  \eear
  \label{eq:A5dicCa}
  \ee

The set $\mu_1\inv(c_a)$  and consists of a single orbit, illustrated as follows.

                 \negspace{19}
  \be
  \diagSfiveMuinvCa {1.9} {0.6} {-4.2} {1.2}
  \ee
The set $\mu_1\inv(c_b)$
is preserved by the non-trivial outer automorphism $\varphi$ of $A_5$. The diconjugacy 
classes $(b|a^c)$, $(a^c|b)$ and $(a|a^{(c^2)})$ are fixed points under $\varphi$, while
the others form orbits of length two.
Proposition \ref{prop:OutAction} implies that $\mu_1\inv(c_b)$ 
splits into two $\mcg_1$-orbits. One orbit has length 3 and looks as follows:

              \negspace{33}
\be
  \begin{tikzpicture}[scale=1.5,commutative diagrams/every diagram]
  \node (a)  [listobj]  at  (-2,0)  {\scriptsize$(b|a^c)$};
  \node (b)  [listobj]  at  (0,0)   {\scriptsize$(a^c|b)$};
  \node (c)  [listobj]  at  (2,0)   {\scriptsize$(a|a^{(c^2)})$};
  \path[commutative diagrams/.cd, every arrow, every label]
  (a)  edge  [\red,out=145,in=215,loop=5mm]  (a)
  (a)  edge  [\blue,<-> ]  (b)
  (b)  edge  [\red,<-> ]   (c)
  (c)  edge  [\blue,out=-35,in=35,loop=5mm]  (c)
  ;
  \end{tikzpicture}
  \ee

              \negspace{17}
\noindent
The remaining $18$ diconjugacy classes in $\mu_1\inv(c_b)$ form another $\mcg_1$-orbit:

              \negspace{17}
  \be
  \diagSfiveMuinvCn {1.2} {1.2} {1.6} {2.3} {1.0}
  \ee
 
\noindent
Concerning $\mu_1\inv(c_c)$, note that the subgroup of $A_5$ generated by $a$ and $d$
is isomorphic to $D_5$. Proposition \ref{prop:CenterSubgroupInv}\,(ii) implies
that the diconjugacy classes $(a|d)$, $(d|a)$ and $(a|ad)$ form an invariant subset.
In fact, there are exactly two orbits: 

              \negspace{31}
\be
  \begin{tikzpicture}[scale=1.5,commutative diagrams/every diagram]
  \node (a)  [listobj]  at  (-2,0)  {\scriptsize$(d|a)$};
  \node (b)  [listobj]  at  (0,0)   {\scriptsize$(a|d)$};
  \node (c)  [listobj]  at  (2,0)   {\scriptsize$(a|ad)$};
  \path[commutative diagrams/.cd, every arrow, every label]
  (a)  edge  [\red,out=145,in=215,loop=5mm]  (a)
  (a)  edge  [\blue,<->]  (b)
  (b)  edge  [\red,<->]   (c)
  (c)  edge  [\blue,out=-35,in=35,loop=5mm]  (c)
  ;
  \end{tikzpicture}
\ee

              \negspace{19}
  \be
  \diagSfiveMuinvCc {2.0} {1.4} {1.3} {0.82}
  \ee

\bigskip
\noindent
Finally, $\mu_1\inv(c_d)$ is isomorphic to $\mu_1\inv(c_c)$ as a $\mcg_1$-set. This 
follows from Proposition \ref{prop:OutAction}, since the non-trivial outer automorphism 
$\varphi$ of $A_5$ induces a bijection between $\mu_1\inv(c_c)$ and $\mu_1\inv(c_d)$.

%%%%%%%%%%%%%%%%%%%%%%%%%%%%%%%%%%%%%%%%%%%%%%%%%%%%%%%%%%%%%%%%

\subsubsection{\boldmath{$\mcg_2$}-orbits}

The number of $4$-conjugacy classes is $|C^4_{A_5}| \eq 216\,341$, and these are distributed
according to $|\mu_2\inv(c_e)| \eq 5\,110$, $|\mu_2\inv(c_a)| \eq 50\,432$, 
$|\mu_2\inv(c_b)| \eq 72\,549$, $|\mu_2\inv(c_c)| \eq |\mu_2\inv(c_d)| \eq 44\,125$.

The set $\mu_2\inv(c_e)$ consists of $12$ orbits. The non-trivial symplectic orbits have 
lengths $15$ (emanating from $\big( (a|e),(e|e)\big)$), $40$ (generated from 
$\big( (b|e), (e|e)\big)$) and $312$ (emanating from $\big( (c|e),(e|e)\big)$);
these are (the non-fixed points) of the type $(\zet/2\zet)^4$, 
$( \zet/3\zet)^4/\mathrm{p}_{2,2}$ and $(\zet/5\zet)^4/\mathrm{p}_{4,2}$. 
In addition there are orbits of lengths $30$, $40$, $60$, $160$, $180$, $240$, $1\, 152$
and $2\, 880$. Some properties of these orbits are given in the following table.
\begin{center}
   \begin{tabular}{ccccc}
   Representative              &  Length    &  $\mcg_2$-quotient    &  Stabilizer             & \Symplectic
   \\[3pt] \hline ~\\[-8pt]
  $\big((e|e),(e|e)\big)$      &  $1$	    &  $\{e\}$              &  --                     &  yes
  \\[3pt]
  $\big( (a|e),(e|e)\big)$     &  $15$	    &  $S_6$                &  $\zet/2\zet\times S_4$ &  yes
  \\[3pt]
  $\big( (b|e),(e|e)\big)$     &  $40$	    &  $O(5,3)$             &  $648$                  &  yes
  \\[3pt]
  $\big( (c|e),(e|e)\big)$     &  $312$	    &  $O(5,5)$             &  $15\, 000$             &  yes
  \\[3pt]
  $\big( (a|e),(dca|e)\big)$   &  $30$	    &  $S_6$                &  $\zet/2\zet\times A_4$ &  yes
  \\[3pt]
  $\big( (a|dca),(e|e)\big)$   &  $40$	    &  $S_6$                &  $\zet/3\zet\times S_3$ &  yes
  \\[3pt]
  $\big( (a|e), (d^4a|e)\big)$ &  $60$	    &  $15\, 216\, 811\, 431\, 690\, 240$  &  $-$     &  no
  \\[3pt]
  $\big( (b|bac), (b,bac)\big)$&  $160$	    &  $3\, 397\, 386\, 240$               &  $-$     &  no
  \\[3pt]
  $\big( (a|e), (d|e)\big)$   & $180$ 
           & \!\!\!\! $338\, 533\, 188\, 894\, 720\, 000\, 000\, 000\, 000\, 000$ \!\!\!\!
                                                                        &  $-$      &  no
  \\[3pt]
  $\big( (b|e), (bac|e)\big)$  &  $240$	&  $41\, 304\, 285\, 341\, 123\, 293\, 285\, 177\, 098\, 240$  & $-$ & no
  \\[3pt]
  $ \!\!\!\!\! \big((b|cb^2a),(cda|c^2a)\big)$ \!\!\!\!\! & $1\, 152$	&  $-$  
                                                                        &  $-$      &  no
  \\[3pt]
  $\big( (b|e), (c^4|e)\big)$  &  $2\, 880$	&  $-$                  &  $-$      &  no
  \end{tabular}
  \end{center}

\noindent
The set $\mu_2\inv(c_a)$ consists of two orbits, of length $1\, 280$ and $49\, 152$,
respectively, while $\mu_2\inv(c_b)$ consists of three orbits of respective length
$135$, $44\, 712$ and $27\, 702$. None of these orbits is symplectic. Some of their
properties are summarized in the following table.
\begin{center}
	\begin{tabular}{ccccc}
   	Representative              &  Length        &  $\mcg_2$-quotient  & \Symplectic
   	\\[3pt] \hline ~\\[-8pt]
	$\big( (e|e),(b|a^d)\big)$	&	$1\, 280$	&	$\sim 10^{78}$	&	no
	\\[3pt]
	$\big( (e|b),({}^cb|{}^{c^2}a)\big)$	& $49\, 152$	&	--	&	no
	\\[3pt]
	$\big( (e|e),(b|a^c)\big)$	&	$135$	&
	             $102\, 662\, 358\, 058\, 319\, 877\, 641\, 002\, 337\, 100\, 103\, 680$  &  no
	\\[3pt]
	$\big( (e|e),(b|b^{c^3})\big)$	&	$44\, 712$	&	--	&	no
	\\[3pt]
	$\big( (e|b),(a^d|a^c)\big)$	&	$27\, 702$	&	--	&	no
   	  \end{tabular}
\end{center}
Note that also for $S_3$ there is a $\mcg_2$-orbit of length $135$, and the 
quotient of $\mcg_2$ acting effectively on these two orbits have the same order.
It is natural to suspect that these two $\mcg_2$-orbits are in fact isomorphic.
Finally, owing to Proposition \ref{prop:OutAction} the sets $\mu_2\inv(c_c)$ and 
$\mu_2\inv(c_d)$ are isomorphic as $\mcg_2$-sets, so we only list properties of the 
former. There are three orbits; the following table lists some of their properties.
\begin{center}
  \begin{tabular}{ccccc}
   Representative              &  Length        &  $\mcg_2$-quotient      & \Symplectic
  \\[3pt] \hline ~\\[-8pt]
  $\big( (b|d^3),(e|e)\big)$	&	$25\, 000$	&	--	  &   no
  \\[3pt]
  $\big( (a|d),(e|e)\big)$	&	$375$	&	$\sim 10^{55}$	  &   no
  \\[3pt]
  $\big( (a|d),((d^3)^b|e)\big)$	&	$18\, 750$	&   --    &   no
  \end{tabular}
\end{center}

~

%%%%%%%%%%%%%%%%%%%%%%%%%%%%%%%%%%%%%%%%%%%%%%%%%%%%%%%%%%%%%%%%

\section{The Drinfeld double of a finite group}
\label{sec:DG}

In this appendix we briefly review the (untwisted) Drinfeld double $\DG$ of a finite group $G$, 
together with some aspects of its representation theory that are used in this paper. 
The structure of $\DG$ as a quasitriangular Hopf algebra as well as further information about
its representation theory can be found in \cite{dipR,with2,maso,dongJ,daiL}.

\subsection{The double \texorpdfstring{\boldmath $\DG$}{DG} of a finite group
            \texorpdfstring{\boldmath $G$}{G}}\label{ssec:DGproperties}

Given a finite group $G$, its double $\DG$ is a Hopf algebra in 
$\mathrm{Vec}_f(\Bbbk)$. It has a basis $\{\b g x\}_{g,x\in G}$ in which the 
structure morphisms take the following form:
  \be
  \bearll
  \text{multiplication}~~ 
  m:\quad & \b g x\oti \b h y \,\mapsto\, \b g x \, \b h y := \del{g^x,h}\, \b g {xy} \,,
  \Nxl3
  \text{unit}~~ \eta:\quad &\dsty 1 \,\mapsto\, \sum_{g\in G} \b g e \,,
  \Nxl1
  \text{comultiplication}~~
  \Delta: ~~ &\dsty \b g x \,\mapsto\, \sum_{h\in G} \b h x\oti \b{h\inv g}x \,,
  \Nxl2
  \text{counit}~~  \eps: \quad & \b g x \,\mapsto\, \del{g,e} \,,
  \Nxl2
  \text{antipode}~~ \apo: \quad & \b g x \,\mapsto\, \b {(g\inv)^x}{x\inv} \,.
  \eear
  \label{DGmorphsinbas}
  \ee
$\DG$ is even a factorizable ribbon Hopf algebra, with R-matrix
$ R \eq \sum_{g,h\in G} \b g e \oti \b h g $
and ribbon element $ \nu \eq \sum_{g\in G} \b g {g\inv} $,
and so with monodromy matrix $ Q \eq \sum_{g,h\in G} \b h g \oti \b g {h^g_{}} $
and Drinfeld element $ u \eq \sum_{g\in G} \b g {g\inv} \eq \nu $.
The special group-like element is trivial, $u\inv \nu \eq \eta$. 

Factorizability means that the Drinfeld map 
$f_Q\colon \DGs\To \DG$ furnished by the monodromy matrix is invertible. 
Denoting the basis of $\DGs$ dual to $\{\b g x\}$ by $\{ \d g x\}_{g,x\in G}$, we have
$ f_Q^{}(\d g x) \eq \b x {g^x} $ and $ f_Q\inv(\b g x) \eq \d {{}^gx} g $.
$\DG$ has a two-sided integral $\Lambda \iN \DG$ and a two-sided cointegral 
$\lambda \iN \DGs$ given by
  \be
  \Lambda:\quad 1 \mapsto \sum_{g\in G} \b e g \qquand
  \lambda = \sum_{g\in G}\d g e \,.
  \label{eq:DGint}
  \ee
We denote the group of group-like elements by $\mathcal{G}(\DG)$ and the space of 
central forms by $C(\DG)$, i.e.
  \be
  C(\DG): = \{ \beta\iN \DGs \,|\, \beta(b\, b') \eq \beta(b'\, b)\ \forall\, b,b'\iN \DG\} \,.
  \label{eq:CDG}
  \ee

\begin{prop}~\label{prop:ZCG}\\[1pt]
{\rm{(i)}} 
The center $Z(\DG)$ consists of the elements $\sum_{g,x\in G} a(g,x)\,\b g x$, 
for all functions $a\colon G\Times G\To \Bbbk$ satisfying
  \be
       a(g,x) = 0 ~~\mathrm{for}~~ [g,x] \neq e \qquand
  a(g^z,x^z) = a(g,x)~~ \mathrm{for~all}~~ g,x,z\iN G \,.
  \label{functions-a}
  \ee
{\rm{(ii)}} 
The space $C(\DG)$ of central forms consists of the elements $\sum_{g,x} a(g,x)\,\d g x$ 
for all functions $a\colon G\Times G\To \Bbbk$ satisfying {\rm \eqref{functions-a}}.
\\[2pt]
{\rm{(iii)}} 
The group of group-like elements in $\DG$ is given by
  \be
  \mathcal{G}(\DG) = \big\{ \, \sum_{g\in G} \xi(g)\,\b g x \,|\,
  x\iN G,\, \xi\iN H_1(G,\Bbbk^\times) \, \big\} \,,
  \ee 
and satisfies $\,\mathcal{G}(\DG) \,{\cong}\, G\Times H_1(G,\Bbbk^\times)$.
\end{prop}

\begin{proof}
(i)\, Writing
  \be
  b = \sum_{g,x\in G} a(g,x) \, \b g x \,, \quad~
  \beta = \sum_{g,x\in G} \alpha(g,x)\, \d g x  \quad\text{ and }\quad
  \Gamma = \sum_{g,x\in G} \gamma(g,x)\, \b g x
  \ee
with arbitrary functions $a$, $\alpha$ and $\gamma$ from $G\Times G$ to $\Bbbk$, 
for $b\iN Z(\DG)$ we need $a({}^{x\!}(h^y),xy\inv) 
   $\linebreak[0]$
{=}\, 0 \eq a(h^y,y\inv x)$,
unless ${}^{x\!}(h^y) \eq h$ or, equivalently, $x \eq z y$ for some $z\iN Z_h$. 
This is, in turn, equivalent to requiring $a({}^z h,z) \eq 0$ unless $z\iN Z_h$,
i.e.\ to $a(g,x) \eq 0$ unless $[g,x] \eq e$. The remaining constraint is
equivalent to $a(g^z,x^z) \eq a(g,x)$ for all triples $g,x,z\iN G$.
\\[2pt]
(ii)\, From
  \be
  \bearl \dsty
  \beta(\b g x\, \b h y) = \del{g^x,h} \sum_{p,q\in G} \alpha(p,q) \, \d p q(\b g {xy})
  = \del{g^x,h}\, \alpha(g,xy) \qquand
  \Nxl2 \dsty
  \beta(\b h y\, \b g x) = \del{h^y,g} \sum_{p,q\in G} \alpha(p,q)\, \d p q(\b h {yx})
  = \del{h^y,g}\, \alpha(h,yx)
  \eear
  \ee
it follows that $\beta\iN C(\DG)$ iff
$ \del{g^x,h}\,\alpha(g,y) \eq \del{g^{y\inv x},h}\,\alpha(g^x,y^x) $ for all $g,h,x,y\iN G$.
This implies that $a(g,y) \eq 0$ unless $[g,y] \eq e$, and that 
$\alpha(g^x,y^x) \eq \alpha(g,y)$ for all $g,x,y\iN G$.
 \\[2pt]
(iii)\, For $\Gamma \iN \DG$ to be group-like we need $\Gamma\oti\Gamma \eq \Delta(\Gamma)
\,{\equiv}\, \sum_{p,g,x\in G} \gamma(pg,x)\, \b p x\oti \b g x$ as well as
$\Gamma\,{\neq}\, 0$. It follows that if $\gamma(g,x)\,{\neq}\, 0$ for some $g\iN G$, then 
$\gamma(g,x)\,\gamma(h,x) \eq \gamma(gh,x)$ for all $g,h\iN G$. Moreover, the
property $\eps(\Gamma) \eq 1$ of group-like elements translates to $\gamma(e,x) \eq 1$.
Together this means that $\gamma(\,\cdot\,,x)\iN H_1(G,\Bbbk^\times)$ 
(for those $x\iN G$ for which $\gamma(g,x)$ is non-zero).
Further, setting $\Gamma_{a,\xi} \eq \sum_{g\in G}\xi(g)\, \b g a$ and 
$\Gamma_{a',\xi'} \eq \sum_{g\in G}\xi'(g)\, \b g{a'}$, one checks that
$ \Gamma_{a,\xi}\,\Gamma_{a',\xi'} \eq \Gamma_{aa',\xi\xi'} $ and
$ (\Gamma_{a,\xi})\inv_{} =\Gamma_{a\inv,\xi\inv} $.
\end{proof}

\medskip

Next recall from \eqref{eq:muN} the map $\mu_1\colon C^2_G\To C_G$ 
from diconjugacy classes to conjugacy classes.
It follows immediately from Proposition \ref{prop:ZCG} that a basis for $Z(\DG)$ 
is given by $\{ v_d\}_{d\in \mu_1\inv(c_e)}$, and a basis for $C(\DG)$
by $\{\gamma_d\}_{d\in \mu_1\inv(c_e)}$, where
  \be
  v_d := \sum_{(g|x)\,\in\, d}\b g x \qquand
  \gamma_d := \sum_{(g|x)\,\in\, d} \d g x 
  \ee
for any diconjugacy class $d\iN C^2_G$.

Denote by $\tilde{\pi}\colon G\Times G\rightarrow G$ the 
projection to the first factor and define $\pi\colon C^{2}_G\To C_G$ by 
  \be
  \pi:\quad d\mapsto \tilde\pi(d) \,.
  \label{eq:leftproj}
  \ee

\begin{prop}\label{prop:Zbasis}
Let $d,d'\iN \mu_1\inv(c_e)$ and let $(g|x)\iN d$. Then
  \be
  v_d\, v_{d'} = \sum_{\{y\in G|(g|y)\in d'\}}\! v_{d_{(g|xy)}} \,.
  \label{eq:vdprod}
  \ee
In particular $v_d\cdot v_{d'} \eq 0$ unless $\pi(d) \eq \pi(d')$.
\end{prop}

\begin{proof}
Note first that since $\{v_d\}_{d\in \mu_1\inv(c_e)}$ forms a basis of $Z(\DG)$, 
the product $v_d\cdot v_{d'}$ is bound to be a sum of basis elements of the same type.
We have
  \be
  v_d\, v_{d'} \, = \! \sum_{\scriptstyle (p|q)\in d\atop \scriptstyle (r|s)\in d'}\! \b p q\, \b r s
  \, = \! \sum_{\scriptstyle (p|q)\in d\atop \scriptstyle (r|s)\in d'}\! \del{p,r}\, \b p {qs}
  \, = \!\!\! \sum_{\scriptstyle (p|q)\in d\atop \scriptstyle s\in G\text{ s.t.}\, (p|s)\in d'}\!\!\! \b p {qs} \,.
  \label{eq:vdprod2}
  \ee
It is quickly verified that the right hand sides of \eqref{eq:vdprod} and
\eqref{eq:vdprod2} coincide.
\end{proof}

%%%%%%%%%%%%%%%%%%%%%%%%%%%%%%%%%%%%%%%%%%%%%%%%%%%%%%%%%%%%%%%%

\subsection{The category \texorpdfstring{\boldmath $\DG$}{DG}-mod for
           char\texorpdfstring{\boldmath $(\Bbbk)\,{\nmid}\, |G|$}{(k) not dividing |G|}}\label{ssec:DGmod}

For any group element $g\iN G$ we denote by $\DG^{(g)}$ the subspace of $\DG$ that is
spanned by $\{\b gx \,|\, x\iN Z_g\}$; this is a subalgebra of \DG\ isomorphic
to $\Bbbk[Z_g]$. By induction, any $\DG^{(g)}$-module $V$ defines a $\DG$-module $\widehat{V}$,
  \be
  \widehat{V} := \DG\otimes_{\DG^{(g)}} V.
  \label{eq:IndMod}
  \ee
It follows that $\widehat{V}$ has a vector space decomposition 
$\widehat{V} \eq \bigoplus_{h\in c_g}\! V_h$ in terms of elements of the conjugacy class 
$c_g$ of $g$, and $\dim_\Bbbk(\widehat{V}) \eq |c_g|\, \dim_\Bbbk(V)$.

Now assume that $\mathrm{char}(\Bbbk)$ does not divide $|G|$. The simple $\DG$-modules 
are then precisely the modules induced from simple $\DG^{(g)}$-modules. More precisely, 
choose a representative $\gc\iN c$ for each conjugacy class $c\iN C_G$. 
Isomorphism classes of simple $\DG$-modules are then in bijection with pairs 
$(c,\alpha)$, where $c\iN C_G$ and $\alpha$ is a simple $Z_{\gc}$-character. If $V$ is 
a $Z_{\gc}$-representation with character $\alpha$, the induced module $\widehat{V}$ is 
a simple $\DG$-module corresponding to the pair $(c,\alpha)$. We denote the set 
$\{ (c,\alpha) \}$ of labels of indecomposable \DG-modules by $I$ and choose for each 
such label a representative module $V_{(c,\alpha)}$. We denote the trivial 
one-dimensional \DG-module $V_{(e,\idscs)}$ by $V_0$.

\medskip

Taking a simple module to its dual amounts to an involution on the set $I$ of labels:
$V_{(c,\alpha)}^\vee \,{\cong}\, V_{\overline{(c,\alpha)}}$. To describe this 
involution concretely we impose a condition on the choice of representatives $\{\gc\}$
of conjugacy classes: For $g \iN c$ denote the conjugacy class containing $g\inv$ by 
$\overline{c}$; this defines an involution on the set $C_G$ of conjugacy classes. 
For $\overline{c}\nE c$ we choose the representatives in such a way that 
$g_{\overline{c}}^{} \eq \gc\inv$. If instead a conjugacy class $c$ contains both $\gc$ 
and $\gc\inv$, we choose a group element $\imath_c\iN G$ such that 
$\gc\inv \eq {}^{\imath_c\!}\gc^{}$. If $\overline{c} \nE c$, then we just set 
$\imath_c \,{:=}\, e$. Note that because of $Z_g \eq Z_{g\inv}$, conjugation by $\imath_c$ 
is an automorphism of $Z_{\gc}$. Now if $\alpha$ is a $Z_{\gc}$-character, we 
define $\alpha^{\imath_c}$ to be the character obtained from $\alpha$ by the 
automorphism given by $\imath_c$, i.e.\ $\alpha^{\imath_c}(x) \eq 
\alpha(\imath_c\inv x\,\imath_c^{})$. With these conventions we finally have
  \be
  \overline{(c,\alpha)} = (\overline{c},\overline{\alpha}^{\imath_c}) \,,
  \label{def-imath}
  \ee
with $\overline{\alpha}$ the character dual to $\alpha$, i.e. 
$\overline{\alpha}(x) \eq \alpha(x\inv)$.

\medskip

The character of a simple module corresponding to $(c,\alpha)$ acts as
  \be
  \chi_{(c,\alpha)}(\b g x) = \left\{ \begin{array}{ll}
  \alpha(x^z) & \text{if}~~ g \eq {}^{z\!}\gc\ \text{ and } x\iN Z_{\gc} \,, \nxl1
  0 &\text{otherwise\,.} \end{array} \right.
  \label{eq:DGchar}
  \ee
Equivalently,
  \be
  \chi_{(c,\alpha)} = \frac{1}{|Z_{\gc}|}\sum_{z\in G}\sum_{x\in Z_{\gc}}
  \alpha(x)\, \d {{}^{z\!} \gc} {{}^{z\!} x} \,.
  \label{eq:DGchar2}
  \ee
It follows from \eqref{eq:DGchar2} that the simple character 
$\chi_{(c,\alpha)}$ lies in the linear span of \\
$ \{ \gamma_d \,|\, d \,{\in}\, \mu_1\inv(c_e)\,{\cap}\, \pi\inv(c) \} \,{\subset}\, \DGs $,
with $\pi$ the map from $C^{2}_G$ to $C_G$ defined in \eqref{eq:leftproj}.

\medskip

Let us list some further facts about the category $\DG\Mod$. 
Since \DG\ is a fini\-te-di\-men\-sional ribbon Hopf algebra,
for any field $\Bbbk$, $\DG\Mod$ is a finite $\Bbbk$-linear ribbon category.
If moreover $\chark \eq 0$, then the Grothendieck ring $K_0(\DG\Mod)$ 
is isomorphic to $\bigoplus_{c\in C_G}\!Z\big(\Bbbk[Z_{\gc}]\big)$ \cite{with2}.

Assume now $\mathrm{char}(\Bbbk)\,{\nmid}\, |G|$. Then $\DG\Mod$ is semisimple 
and even modular. The rank of $\DG\Mod$ is $\sum_{c\in C_G} |C_{Z_{\gc}}|$ which, 
as is shown in Proposition \ref{prop:ordermu1inv}, coincides with $|\mu_1\inv(c_e)|$ 
and thus with  the dimension of $C(\DG)$ and $Z(\DG)$. The dimension $\mathrm{D}_{\DG\Mods}$ 
of $\DG\Mod$ is $|G|$. Moreover, since $\DG\Mod$ is a 
Drinfeld center, namely the center of the category of $G$-graded vector spaces
(in other words, its class in the Witt group of modular tensor categories is trivial),
the gluing anomaly of the Reshetikhin-Turaev TQFT constructed from $\DG\Mod$ 
vanishes, and hence all mapping class group representations 
are genuine linear representations.

All quantum dimensions in $\DG\Mod$ are integers:
  \be
  \mathrm{d}_{(c,\alpha)} = |c|\, \alpha(e) = \mathrm{dim}_\Bbbk(V_{(c,\alpha)}) \,.
  \label{qdims}
  \ee
The twist phases are given by
  \be
  \theta_{(c,\alpha)} = \frac{\alpha(\gc)}{\alpha(e)} \,.
  \label{eq:Vtwist}
  \ee
Since $\gc$ is central in $Z_{\gc}$, $\theta_{(c,\alpha)}$ is a root of unity. According
to \eqref{qdims} the invertible objects of $\DG\Mod$ correspond to labels $(c_a,\zeta)$
with $a\iN Z(G)$ and $\zeta$ a linear character of $G$. The character of an
invertible object takes the simplified form
  \be
  \chi_{(c_a,\zeta)} = \sum_{x\in G} \zeta(x)\, \d ax
  = \sum_{d \,\in\, \mu_1\inv(c_e)\cap\pi\inv(c_a)} \zeta(x_d)\, \gamma_d \,,
  \ee
with $x_d$ such that $(\gc|x_d)$ for some $\gc$ is a representative of
the diconjugacy class $d$.

\begin{rem}
The isomorphism classes of invertible objects form a subgroup of the 
Grothendieck ring $K_0(\DG\Mod)$,
called the the Picard group $\mathrm{Pic}(\DG\Mod)$. This group is given by
$ \mathrm{Pic}(\DG\Mod) \,{\cong}\, Z(G)\,{\times}\, H_1(G,\Bbbk^\times)
\,{\cong}\, Z\big(\mathcal{G}(\DG)\big) $.
\end{rem}

%%%%%%%%%%%%%%%%%%%%%%%%%%%%%%%%%%%%%%%%%%%%%%%%%%%%%%%%%%%%%%%%

\subsection{The coend \texorpdfstring{\boldmath $\DGa$}{DGa}}

To any finite tensor category $\mathcal C$ there is associated a standard coend, namely
\eqref{Lcoend}, i.e.\ the coend $L$
of the functor $(U,V) \,{\mapsto}\, U^\vee \oti V$. In the context of representations
of mapping class groups, this coend has been studied in \cite{lyub6}; we refer to it
as the \emph{Lyubashenko coend}. In this appendix we describe this coend 
for the category $\DG\Mod$. As an object of $\DG\Mod$, the Lyubashenko coend is the 
left coadjoint module $\DGa$ \cite{lyub6,vire4}; in terms of the standard bases
of \DG\ and \DGs, the action of \DG\ on $\DGa$ is 
  \be
  \b g x \,.\, \d h y = \del{g^x,[h,y]}\, \d {{}_{}^{x\!}h}{{}^{x\!}y} \,.
  \label{eq:DGcoa}
  \ee
Note that the action \eqref{eq:DGcoa} preserves the subspaces of \DGa\ that are spanned by those
$\d g x$ for which $(g|x)\iN d$ for a fixed diconjugacy class $d$.

The Lyubashenko coend $L$ of a braided finite tensor category has a natural Hopf 
algebra structure \cite{lyub6}. If $\mathcal C$ is the representation category of a 
factorizable Hopf algebra, this structure is given by the formulas
\eqref{eq:Lapo}. Specializing to $H \eq \DG$ these give
  \be
  \bearll
  m\coa:      & \d g x \oti \d h y \mapsto \del{[x,g]x,y}\, \d {(g^x)hyx\inv}{x} \,, 
  \Nxl2
  \eta\coa:   & 1 \mapsto \sum_{x\in G}\d e x \,,             
  \Nxl2
  \Delta\coa: & \d gx \mapsto \d gy \oti \d{g^y}{y\inv x} \,, 
  \Nxl2
  \eps\coa:   & \d g x \mapsto \del{x,e} \,,                  
  \Nxl2
  \apo\coa:   & \d gx \mapsto \d {g\inv [x,g]}{(x\inv)^g} \,. 
  \label{eq:DGamult..}
  \eear
  \ee
The Hopf algebra \DGa\ in $\DG\Mod$ comes equipped with a two-sided integral 
$\Lambda\coa$ and a two-sided cointegral $\lambda\coa$ given by
  \be
  \Lambda\coa:\quad 1 \mapsto \sum_{g\in G} \d g e \qquad\text{and by}\qquad
  \lambda\coa:\quad \d g x \mapsto \del{g,e} \,,
  \label{eq:DGaintegrals}
  \ee
respectively. 
Note that $\lambda\coa\circ\Lambda\coa \eq |G| \eq \lambda\circ\Lambda$.
\DGs\ also comes with a Hopf pairing 
  \be
  \omega\coa:\quad \d gx \oti \d hy \,\longmapsto\, \del{g\inv[x,g],h}\, \del{(x\inv)^g,y} \,,
  \label{eq:DGaHopfpairing}
  \ee
obtained by specializing \eqref{Lhopa}. This Hopf pairing has the following properties:

\begin{prop}~\\[1pt]
{\rm (i)} 
For $d,d'\iN C^2_G$ we have
  \be
  \omega\coa(\gamma_d\oti\gamma_{d'}) = \del{d',\overline{d}}\, |d| \,,
  \label{hopa-gaga}
  \ee
where $\overline{d}$ denotes the diconjugacy class satisfying 
$(g\inv|x\inv) \iN \overline d$ if $(g|x)\iN d$.
\\[2pt]
{\rm (ii)} 
For each $d\iN C^2_G$ such that $\pi(d) \eq c$, let $(c,\alpha)$ label a simple $\DG$-module, 
and choose a representative $(\gc|x_d)\iN d$. If $\mathrm{char}(\Bbbk) \,{\nmid}\, |G|$, 
then the character of $(c,\alpha)$ can be expanded as
  \be
  \chi_{(c,\alpha)} = \sum_{d\in \mu_1\inv(c_e)\cap\pi\inv(c)} \alpha(x_d)\, \gamma_d \,.
  \label{chi=sumalphagamma}
  \ee
{\rm (iii)} 
Applied to the characters of simple \DG-modules labeled by $(c,\alpha)$ and 
$(c',\alpha')$, the Hopf pairing gives
  \be
  \omega\coa(\chi_{(c,\alpha)}\oti\chi_{(c',\alpha')})
  = \del{c',\overline{c}}\, \del{\alpha',\alpha}\, |G| \,.
  \label{eq:HopfpairingChar}
  \ee
\end{prop}

Note that $\omega\coa$ does not necessarily pair dual $\DG$-characters.

\begin{proof}
(i) By the elementary equalities $g\inv[x,g] \eq (g\inv)^{xg}$ and
$(x\inv)^g \eq (x\inv)^{xg}$ we have
  \be
  \omega\coa(\gamma_d\oti\gamma_{d'}) = \sum_{\scriptstyle (g|x)\in d \atop \scriptstyle (h|y)\in d'}
  \del{(g\inv)^{xg},h}\, \del{(x\inv)^{xg},y} \,, 
  \ee
which directly implies \eqref{hopa-gaga}.
\\[2pt]
(ii) If $|G|$ is non-zero in $\Bbbk$, then the cardinality of any diconjugacy 
class $d$ is non-zero as well, since $|d| \eq [G \,{:}\, Z_{(\gc|x_d)}] 
\eq [G \,{:}\, Z_{\gc}]\, [Z_{\gc} \,{:}\, Z_{(\gc|x_d)}]$.
Also note that $|\overline d| \eq |d|$. Characters 
can be expanded in the basis $\{\gamma_d\}_{d\in \mu_1\inv(c_e)}$
of $C(\DG)$, i.e.\ $\chi_{(c,\alpha)} \eq \sum_{d\in \mu_1\inv(c_e)} a(d)\, \gamma_d$
with suitable coefficients $a(d)$. According to \eqref{hopa-gaga} the latter are 
$ a(d) \eq |d|^{-1} \omega\coa(\chi_{(c,\alpha)}\oti \gamma_{\overline{d}}) $.
Applying the character formula \eqref{eq:DGchar} thus yields
  \be
  \begin{array}{r}
  a(d)  \dsty
  = \frac{1}{|d|\, |Z_{\gc}|} \sum_{\scriptstyle z\in G\atop \scriptstyle x\in Z_{\gc}} \! 
  \sum_{(h|y)\in d} \alpha(x)\, \omega\coa(\d {{}^{z\!}\gc}{{}^zx}\oti\d {h\inv}{y\inv})
  \qquad\quad
  \\[-7pt] \dsty
  = \frac{|G|}{|d|\, |Z_{\gc}|}\sum_{x\in Z_{\gc}} \alpha(x)\, \del{d,d_{(\gc,x)}} \,.
  \eear
  \ee
Now if $(\gc,x)\iN d$, then $(\gc|y)\iN d$ for every $y$ that is $Z_{\gc}$-conjugate 
to $x$, so that the sum on the right hand side, if non-zero, takes the value 
$\alpha(x_d)\, [Z_{\gc} \,{:}\, Z_{(\gc|x_d)}]$. Moreover, the sum is non-zero iff 
$\pi(d) \eq c$. We conclude that if $d\in \mu_1\inv(c_e) \,{\cap}\, \pi\inv(c)$, then
  \be
  a(d) = \frac {[G \,{:}\, Z_{\gc}]\, [Z_{\gc} \,{:}\, Z_{(\gc|x_d)}]} {|d|}\,
  \alpha(x_d) = \alpha(x_d) \,,
  \ee
while otherwise $a(d) \eq 0$. This proves (ii).
\\[2pt]
(iii) Combining \eqref{hopa-gaga} and \eqref{chi=sumalphagamma} gives
  \be
  \bearll
  \omega\coa(\chi_{(c,\alpha)}\oti\chi_{(c',\alpha')}) 
  & \displaystyle = \sum_{\scriptstyle d\in \mu_1\inv(c_e)\cap\pi\inv(c) \atop
  \scriptstyle d'\in \mu_1\inv(c_e)\cap\pi\inv(c')}
  \Nxl3
  & \displaystyle = \del{c',\overline c} \sum_{d\in \mu_1\inv(c_e)\cap\pi\inv(c)}
  \!\! \alpha(x_d)\, \alpha'(x_{\overline d})\, |d|\,.
  \label{eq:sumdcpi}
  \eear
  \ee  
We can choose $x_{\overline{d}} \eq x_d^{-1}$. Now note that the number of elements 
$(g|x)$ in $d$ with the first entry fixed is $[Z_{g} \,{:}\, Z_{(g|x)}]$. Using also
that $|d| \eq [G \,{:}\, Z_{\gc}]\, [Z_{\gc} \,{:}\,Z_{(\gc|x_d)}]$, the 
right hand side of \eqref{eq:sumdcpi} can be rewritten as
  \be
  \omega\coa(\chi_{(c,\alpha)}\oti\chi_{(c',\alpha')}) 
  = \del{c',\overline{c}}\, \frac{|G|}{|Z_{\gc}|}\sum_{d\in \mu_1\inv(c_e)\cap\pi\inv(c)} \!
  \frac{|Z_{\gc}|}{|Z_{(\gc|x_d)}|}\, \alpha(x_d)\, \alpha'(x_{d}^{-1}) \,.
  \ee
This equals $\del{c',\overline{c}}\,\del{\alpha,\alpha'}\, |G| $,
thus establishing \eqref{eq:HopfpairingChar}.
\end{proof}

\medskip

In the semisimple case, i.e.\ for $\mathrm{char}(\Bbbk)\,{\nmid}\, |G|$, the Lyubashenko coend 
is given by $\bigoplus_{i} X_i^\vee\oti X_i$, where the (finite) summation ranges over 
representatives of isomorphism classes of simple objects. Concretely,
$\DGs \,{\cong}\, \bigoplus_{(c,\alpha)} V_{\overline{(c,\alpha)}}\oti V_{(c,\alpha)}$ 
as a left $\DG$-module.

%%%%%%%%%%%%%%%%%%%%%%%%%%%%%%%%%%%%%%%%%%%%%%%%%%%%%%%%%%%%%%%%

\section{Properties of \texorpdfstring{\boldmath $2\Ng$}{2\Ng}-conjugacy classes}\label{sec:2Nconj}

For $c \iN C_G$ we denote by $\overline{c}$ the conjugacy class containing $g\inv$ for 
$g\iN c$.  For $c_1,c_2\iN C_G$, the subset $c_1\,{\cdot}\, c_2 \,{:=}\,\{g_1g_2 \,|\, 
g_1\iN c_1,\,g_2\iN c_2\}$ is a disjoint union of conjugacy classes, in which a
conjugacy class $c \,{\subseteq}\, c_1\,{\cdot}\, c_2$ appears with a multiplicity  $\mathcal{K}_{c_1,c_2}^{\phantom{c_1c_2}c}$
given by the number of solutions to $g_1g_2 \eq \gc$ with $g_1\iN c_1$ and $g_2\iN c_2$, 
modulo conjugation. 

Recall from \eqref{com=cupc} that $C'_G \,{\subseteq}\, C_G$ is the set of conjugacy 
classes of commutators, so that $\com_G \eq \bigcup_{c\in C'_G}\! c$. We have
  
\begin{lemma}
{\rm (i)} 
A conjugacy class $c$ lies in $C'_G$ iff there exists a $c_1\iN C_G$ such that 
$c \,{\subseteq}\, \overline{c}_1\,{\cdot}\, c_1$, i.e.\ such that 
$\mathcal{K}_{\overline{c}_1,c_1}^{\phantom{\overline{c}_1c_1}c} \,{\ne}\, 0$.
\\[2pt]
{\rm (ii)} 
We have $C'_G \eq \im(\mu_1) \,{\subseteq}\, \im(\mu_2) \,{\subseteq}\,
\im(\mu_3) \,{\subseteq}\,\cdots$.
\\[2pt]
{\rm (iii)} 
If $G$ is perfect, then there exists an $N\iN \NN$ such that $\im(\mu_n) \eq C_G$ 
for $n \,{>}\,N$.
\end{lemma}

\begin{proof}
One has $c\iN C'_G$ iff $\gc$ is a commutator.
Since $[g,x] \eq g\inv (g^x)$, this proves (i).
The proofs of (ii) and (iii) are immediate.
\end{proof}

Enumeration of $2\Ng$-conjugacy classes is in general a difficult task.
But the following general formulas for the sizes and number of $2\Ng$-conjugacy classes
are easily obtained.

\begin{lemma}
{\rm (i)} 
If $(\bg|\bx)\iN d$, then the size of $d \iN C^{2\Ng}_G$ is
  \be
  |d| = \frac{|G|}{|Z_{(\bg|\bx)}|} \,.
  \ee
{\rm (ii)} 
For any $\Ng\iN \mathbb{N}$ we have
  \be
  |C^{2\Ng}_G| = \frac{1}{|G|}\sum_{(\bg|\bx)\in G^{\times 2\Ng}}|Z_{(\bg|\bx)}| \,.
  \ee
{\rm (iii)} 
Moreover,
  \be
  |C^2_G| = \sum_{c\in C_G}\frac{1}{|Z_{\gc}|}\sum_{x\in G}|Z_{(\gc|x)}| \,.
  \label{eq:NumberTBlocks}
  \ee
\end{lemma}

\begin{proof}
(i) is an instance of the orbit-stabilizer theorem, (ii) is the corresponding
instance of Burnside's lemma, and (iii) also follows immediately from Burnside's lemma.
\end{proof}

A diconjugacy class $d$ is a subset of the Cartesian product of two conjugacy classes, 
$d\subseteq c_1\Times c_2$, and there is a subset $D_{c_1,c_2} \,{\subset}\, C^2_G$ 
such that $c_1\Times c_2 \eq \bigcup_{d\in D_{c_1,c_2}}\! d$.

\begin{prop}
For any $c,c'\iN C_G$ the number of diconjugacy classes in $c\Times c'$ is
  \be
  |D_{c,c'}| = \frac{1}{|Z_{\gc}|\, |Z_{g_{c'}}|}\,
  \sum_{x\in G} \big| Z_{\gc}\,{\cap}\, {}_{}^{x\!}Z_{g_{c'}} \big|
  = |Z_{\gc} \!\backslash G/ Z_{g_{c'}}| \,.
  \label{propB3}
  \ee
\end{prop}

\begin{proof}
For any pair of subgroups $H,K \,{\le}\, G$ the cardinality $|H\backslash G/K|$ of the set
of double cosets is given by Burnside's lemma:
  \be
  |H\backslash G/K| = \frac{1}{|H|\,|K|} \sum_{(h,k)\in H\times K} | G^{(h,k)}| 
  = \frac{1}{|H|\,|K|} \sum_{x\in G}|H\cap {}_{}^{x\!} K| \,.
  \label{B5+B6}
  \ee
Here $G^{(h,k)}$ is the set of elements $x\iN G$ satisfying $hxk \eq x$, i.e.\ ${}^x k\eq h^{-1}$.
The second equality in \eqref{B5+B6} holds because for ${}^x k\eq h^{-1}$
one needs $h\iN H \,{\cap}\, {}^xK$.
We have thus established the second equality in \eqref{propB3}.
\\[2pt]
Now note that every $d \,{\subseteq}\, c\times c'$ has a representative of the form 
$(\gc|{}^xg_{c'})$ for some $x\iN G$, and that the cardinality $|D_{c,c'}|$ coincides 
with the number of $Z_{\gc}$-orbits on the set of such representatives. 
Further, for $z\iN Z_{\gc}$ we have $(\gc|{}^{zx}g_{c'}) \eq (\gc|{}^xg_{c'})$
iff $z\iN Z_{\gc} \,{\cap}\, {}_{}^{x\!}Z_{g_{c'}}$. Again using Burnside's lemma,
the number of orbits then becomes
  \be
  \frac{1}{|Z_{\gc}|\,|Z_{g_{c'}}|}\, \sum_{x\in G} |Z_{\gc}\,{\cap}\, {}_{}^{x\!}Z_{\gc'}| \,,
  \ee
where the prefactor $1/|Z_{g_{c'}}|$ rectifies overcounting. This completes the proof.
\end{proof}

We are particularly interested in the number $|\mu_1\inv(c)|$ of diconjugacy classes 
which are mapped by $\mu_1$ to a fixed conjugacy class $c\iN C'_G$. Assume that 
$c' \,{\subseteq}\, \overline{c} \,{\cdot}\, c$; then there exists an $x\iN G$ such that 
$\gc\inv (\gc^x)\iN c'$. More generally, if 
$\mathcal{K}_{\overline{c},c}^{\phantom{\overline{c},c}c'} \, {\geq} \,1$, then we can, and will, 
choose elements $x_i$, $i \eq 1,2,...\,, \mathcal{K}_{\overline{c},c}^{\phantom{\overline{c}c}c'}$, 
representing the $\mathcal{K}_{\overline{c},c}^{\phantom{\overline{c}c}c'}$ inequivalent 
solutions, i.e.\ we have $\gc\inv (\gc^{x_i})\iN c'$ for every $i$, and there does 
not exist any $z\iN Z_{\gc}$ such that $\gc^{x_iz} \eq \gc^{x_j}$ for $i \,{\neq}\, j$. 
The representatives $x_i$ are of course not unique -- for every pair $z_1,z_2\iN Z_{\gc}$, 
$z_1 x_i z_2$ represents the same solution. Moreover, two elements $x,x'\iN G$ 
satisfying $\gc\inv(\gc^x)\iN c'$ and $\gc\inv(\gc^{x'})\iN c'$ represent the same 
solution iff there exist $z_1,z_2\iN Z_{\gc}$ such that $x' \eq z_1xz_2$.

\begin{lemma}
Let $c\iN C_G$ and $c'\iN C_G'$ be such that $c' \,{\subseteq}\, \overline{c}\,{\cdot}\, c$. 
The number of diconjugacy classes in $c\Times G$ that are mapped to $c'$ by $\mu_1$ is
  \be
  \sum_{i=1}^{\mathcal{K}_{\overline{c},c}^{\phantom{\overline{c}c}c'}} \zeta_{c,c',i}
  \qquad{\rm with}\quad~
  \zeta_{c,c',i} := \frac{1}{|Z_{(\gc|x_i)}|}
  \sum_{y\in Z_{(\gc^{}|x_i)}}\! |Z_{\gc} \,{\cap}\, Z_{y}| \,.
  \label{defzeta}
  \ee
\end{lemma}

\begin{proof}
Fix a value $i\iN \{ 1,2,...\,, \mathcal{K}_{\overline{c},c}^{\phantom{\overline{c}c}c'} \}$. 
All pairs $(\gc|z_1 x_i z_2)$ with $z_1,z_2\iN Z_{\gc}$ represent the same solution
$x$ of the condition $\gc\inv (g^x)\iN c'$. Due to the identity
$((\gc)^{z_1\inv}|(z_1x_iz_2)^{z_1\inv}) $\linebreak[0]$ {=}\,
(\gc|x_iz_2z_1\inv)$ it is enough to consider pairs of the type $(\gc|x_iz)$ for 
$z\iN Z_{\gc}$. Conjugating this pair by $y\iN Z_{(\gc|x_i)}$ gives $(\gc|x_i\,z^y)$, 
so the number of diconjugacy classes representing the solution $\gc\inv (g^{x_i})$ is 
given by the number of $Z_{(\gc|x_i)}$-or\-bits in $Z_{\gc}$, where the subgroup acts by 
conjugation. By Burnside's lemma the number of orbits is precisely
$\zeta_{c,c',i}$ as defined in \eqref{defzeta}.
\end{proof}

\begin{cor}\label{cor:nodiconj}
The number of diconjugacy classes of $G$ that are mapped by $\mu_1$ to $c'\iN C_G'$ is
  \be
  |\mu_1\inv(c')| = \sum_{c\in C_G}\sum_{i=1}^{\mathcal{K}_{\overline{c},c}^{\phantom{\overline{c}c}c'}}
  \frac{1}{|Z_{(\gc|x_i)}|} \sum_{y\in Z_{(\gc|x_i)}}\! |Z_{\gc} \,{\cap}\, Z_{y}| \,.
  \label{eq:nodiconj}
  \ee
\end{cor}

Specializing to $\mu_1\inv(c_e) \,{\subseteq}\, C^2_G$, by a simple calculation we get

\begin{prop}\label{prop:ordermu1inv}
The number of diconjugacy classes of $G$ mapped by $\mu_1$ to $c_e \eq \{e\}$ is 
  \be
  |\mu_1\inv(c_e)| = \sum_{c\in C_G} |C_{Z_{\gc}}| \,.
  \label{eq:Number0ptTBlocks}
  \ee
\end{prop}

%%%%%%%%%%%%%%%%%%%%%%%%%%%%%%%%%%%%%%%%%%%%%%%%%%%%%%%%%%%%%%%%
\vskip 3em

{\small \noindent{\sc Acknowledgments:}
We are grateful to Terry Gannon and Gregor Masbaum for helpful comments.
JFu is supported by VR under projects no.\ 621-2013-4207 and 2017-03836. He thanks
the Erwin-Schr\"odinger-Institute (ESI) for the
hospitality during the program ``Modern trends in topological field theory''
while part of this work was done; this program was also supported by the network
``Interactions of Low-Dimensional Topology and Geometry with Mathematical
Physics'' (ITGP) of the European Science Foundation.
JFj received support from the VR Swedish Research Links Programme under project 
no.\ 348-2008-6049, by the 985 Project Grants from the Chinese Ministry of Education, 
by the Priority Academic Program Development of Jiangsu Higher Education Institutions 
(PAPD), and by the China Science Postdoc Grant 2011M500887. He is grateful to the 
Department of Physics of Nanjing University where part of this research was carried out.

%%%%%%%%%%%%%%%%%%%%%%%%%%%%%%%%%%%%%%%%%%%%%%%%%%%%%%%%%%%%%%%%

%%%%%%%%%%%%%%%%%%%%%%%%%%%%%%%%%%%%%%%%%%%%%%%%%%%%%%%%%%%%%%%%

\end{document}